\documentclass[10pt]{amsart}
      \usepackage{amsmath,amsfonts}
\def\rg{\hbox to 30pt{\rightarrowfill}}
\def\lg{\hbox to 30pt{\leftarrowfill}}

      \parskip\smallskipamount
          \newtheorem{theorem}{Theorem}[section]
      \newtheorem{definition}[theorem]{Definition}
      \newtheorem{proposition}[theorem]{Proposition}
      \newtheorem{corollary}[theorem]{Corollary}
      \newtheorem{lemma}[theorem]{Lemma}

      \makeatletter
      \@addtoreset{equation}{section}
      \makeatother

      \newcommand{\BB}{{\mathbb B}}
      \newcommand{\CC}{{\mathbb C}}
      \newcommand{\NN}{{\mathbb N}}
      
      \newcommand{\ZZ}{{\mathbb Z}}
      \newcommand{\DD}{{\mathbb D}}
      
      \newcommand{\FF}{{\mathbb F}}
      \newcommand{\TT}{{\mathbb T}}

\newcommand{\HH}{{\mathbb H}}

      \newcommand{\cA}{{\mathcal A}}
      
      \newcommand{\cC}{{\mathcal C}}
      \newcommand{\cD}{{\mathcal D}}
      \newcommand{\cE}{{\mathcal E}}
      
      \newcommand{\cG}{{\mathcal G}}
      \newcommand{\cH}{{\mathcal H}}
      \newcommand{\cK}{{\mathcal K}}
      \newcommand{\cL}{{\mathcal L}}
      \newcommand{\cM}{{\mathcal M}}
      \newcommand{\cN}{{\mathcal N}}
      \newcommand{\cQ}{{\mathcal Q}}

      \newcommand{\cS}{{\mathcal S}}
      
      \newcommand{\cU}{{\mathcal U}}
      \newcommand{\cV}{{\mathcal V}}
      \newcommand{\cY}{{\mathcal Y}}
      \newcommand{\cX}{{\mathcal X}}

      \newcommand{\supp}{\hbox{\rm{supp}}\,}
      \newcommand{\rank}{\hbox{\rm{rank}}\,}

      \newdimen\expt
      \def\boxit#1{\setbox0\hbox{$\displaystyle{#1}$}
            \hbox{\lower.4\expt
       \hbox{\lower3\expt\hbox{\lower\dp0
            \hbox{\vbox{\hrule height.4\expt
       \hbox{\vrule width.4\expt\hskip3\expt
            \vbox{\vskip3\expt\box0\vskip2\expt}%
       \hskip3\expt\vrule width.4\expt}\hrule height.4\expt}}}}}}
      
\begin{document}
       \pagestyle{myheadings}
      \markboth{ Gelu Popescu}{  Noncommutative polydomains,   Berezin transforms,   and operator model theory    }

      \title [  Berezin transforms on Noncommutative varieties  in polydomains]
      {          Berezin transforms on Noncommutative  varieties in polydomains }
        \author{Gelu Popescu}
\date{April 10, 2013}
      \thanks{Research supported in part by an NSF grant}

      \address{Department of Mathematics, The University of Texas
      at San Antonio \\ San Antonio, TX 78249, USA}
      \email{\tt gelu.popescu@utsa.edu}

\begin{abstract} Let $\cQ$ be a set of polynomials in noncommutative indeterminates
$Z_{i,j}$, $i\in \{1,\ldots,k\}$, $j\in \{1,\ldots, n_i\}$. In this paper, we study noncommutative varieties
$$
\cV_\cQ(\cH):=\{{\bf X}=\{X_{i,j}\}\in {\bf D}(\cH):\ g({\bf X})=0 \text{ for all } g\in \cQ\},
$$
where ${\bf D}(\cH)$ is a {\it regular polydomain} in $B(\cH)^{n_1+\cdots +n_k}$ and $B(\cH)$ is the algebra of bounded linear operators on a Hilbert space $\cH$.
Under natural conditions on $\cQ$, we show that there is a {\it universal model}
${\bf S}=\{{\bf S}_{i,j}\}$ such that $g({\bf S})=0$, $g\in \cQ$, acting on a subspace of a tensor product of full Fock spaces. We characterize the variety $\cV_\cQ(\cH)$ and  its pure part in terms of the universal model  and a class of completely positive linear maps. We obtain a characterization of  those elements   in $\cV_\cQ(\cH)$ which admit characteristic functions and prove that the characteristic function is a complete unitary invariant for the class of completely non-coisometric elements. We study the universal model ${\bf S}$, its joint invariant subspaces and the representations of the universal operator algebras it generates: the {\it variety algebra} $\cA(\cV_\cQ)$, the Hardy algebra $F^\infty(\cV_\cQ)$, and   the $C^*$-algebra $C^*(\cV_\cQ)$. Using noncommutative Berezin transforms associated with each variety, we develop  an operator model theory and dilation theory for large classes of   varieties in  noncommutative polydomains. This includes various commutative cases  which are  close connected to the theory of holomorphic functions in several complex variables and algebraic geometry.
\end{abstract}

      \maketitle

\section*{Contents}
{\it

\quad Introduction

\begin{enumerate}
   \item[1.] Noncommutative  varieties in polydomains and   Berezin transforms
   \item[2.] Universal operator  models and joint invariant subspaces
    \item[3.] Noncommutative varieties and multivariable  function theory
       \item[4.] Isomorphisms of universal operator algebras
            \item[5.] Dilation theory on noncommutative varieties in polydomains
           \item[6.] Characteristic functions and operator models
   \end{enumerate}

\quad References

}

\bigskip

\section*{Introduction}

 We denote by $B(\cH)^{n_1}\times_c\cdots \times_c B(\cH)^{n_k}$
   the set of all tuples  ${\bf X}:=({ X}_1,\ldots, { X}_k)$ in $B(\cH)^{n_1}\times\cdots \times B(\cH)^{n_k}$
     with the property that the entries of ${X}_s:=(X_{s,1},\ldots, X_{s,n_s})$  are commuting with the entries of
      ${X}_t:=(X_{t,1},\ldots, X_{t,n_t})$  for any $s,t\in \{1,\ldots, k\}$, $s\neq t$.
 In an attempt to unify the multivariable operator model theory  for the ball-like domains and  commutative polydiscs,  we developed in \cite{Po-Berezin} an operator  model  theory and a theory of free holomorphic functions on  {\it regular polydomains} of the form
$$
{\bf D_q^m}(\cH):=\left\{ {\bf X}=(X_1,\ldots, X_k)\in B(\cH)^{n_1}\times_c\cdots \times_c B(\cH)^{n_k}: \ {\bf \Delta_{q,X}^p}(I)\geq 0 \ \text{ for }\ {\bf 0}\leq {\bf p}\leq {\bf m}\right\},
$$
where ${\bf m}:=(m_1,\ldots, m_k)$ and ${\bf n}:=(n_1,\ldots, n_k)$ are in $\NN^k$,
 the {\it defect mapping} ${\bf \Delta_{q,X}^m}:B(\cH)\to  B(\cH)$ is defined by
$$
{\bf \Delta_{q,X}^m}:=\left(id -\Phi_{q_1, X_1}\right)^{m_1}\circ \cdots \circ\left(id -\Phi_{q_k, X_k}\right)^{m_k},
$$
and ${\bf q}=(q_1,\ldots, q_k)$ is  a $k$-tuple of positive regular polynomials $q_i\in \CC[Z_{i,1},\ldots, Z_{i,n_i}]$, i.e., all the coefficients  of $q_i$ are positive, the constant term is zero, and the coefficients of the linear terms $Z_{i,1},\ldots, Z_{i,n_i}$ are different from zero.  If the polynomial $q_i$ has the form $q_i=\sum_{\alpha} a_{i,\alpha} Z_{i,\alpha}$, the  completely positive linear map
$\Phi_{q_i,X_i}:B(\cH)\to B(\cH)$  is defined by setting $\Phi_{q_i,X_i}(Y):=\sum_{\alpha} a_{i,\alpha} X_{i,\alpha} Y X_{i,\alpha} ^*$ for $Y\in B(\cH)$.

 In this paper, we study noncommutative varieties in the polydomain ${\bf D_q^m}(\cH)$, given by
$$
\cV_\cQ(\cH):=\{{\bf X}\in {\bf D_q^m}(\cH):\ g({\bf X})=0 \text{ for all } g\in \cQ\},
$$
where $\cQ$ is a set of polynomials in noncommutative indeterminates
$Z_{i,j}$, which generates a nontrivial ideal in $\CC[Z_{i,j}]$. The goal is to  understand the structure of this noncommutative variety, determine
 its elements and classify them up to unitary equivalence, for
  large classes of sets  $\cQ\subset \CC[Z_{i,j}]$.
  This study can be seen as  an attempt to initiate   noncommutative algebraic geometry in polydomains.

 To present our results, we need some notation.   Let $H_{n_i}$ be
an $n_i$-dimensional complex  Hilbert space.
  We consider the full Fock space  of $H_{n_i}$ defined by
$$F^2(H_{n_i}):=\bigoplus_{p\geq 0} H_{n_i}^{\otimes p},$$
where $H_{n_i}^{\otimes 0}:=\CC 1$ and $H_{n_i}^{\otimes p}$ is the
(Hilbert) tensor product of $p$ copies of $H_{n_i}$. Let $\FF_{n_i}^+$ be the unital free semigroup on $n_i$ generators
$g_{1}^i,\ldots, g_{n_i}^i$ and the identity $g_{0}^i$. We use the notation $Z_{i,\alpha_i}:=Z_{i,j_1}\cdots Z_{i,j_p}$
  if $\alpha_i\in \FF_{n_i}^+$ and $\alpha_i=g_{j_1}^i\cdots g_{j_p}^i$, and
   $Z_{i,g_0^i}:=1$.
 If $(\alpha):=(\alpha_1,\ldots, \alpha_k)$ is  in $ \FF_{n_1}^+\times \cdots \times\FF_{n_k}^+$, we  set $Z_{(\alpha)}:= Z_{1,\alpha_1}\cdots Z_{k,\alpha_k}$.

   In Section 1, after setting up the notation and recalling some basic results from \cite{Po-Berezin-poly}, we show that the {\it abstract variety} $\cV_\cQ:=\{\cV_\cQ(\cH): \  \cH\text{ \it is a Hilbert space}\}$ has  a universal model ${\bf S}=\{{\bf S}_{i,j}\}$ such that $g({\bf S})=0$, $g\in \cQ$, where  each ${\bf S}_{i,j}$ is
  acting on a subspace $\cN_\cQ$ of a tensor product of full Fock spaces.
  For each  element ${\bf T}\in  \cV_\cQ(\cH)$ we introduce the {\it constrained
noncommutative Berezin transform}  at ${\bf T}$ as the map ${\bf
B_{T, {\cQ}}}:B(\cN_\cQ) \to B(\cH)$ defined by setting
 \begin{equation*}
 {\bf B_{T,{\cQ}}}[\varphi]:= {\bf K_{q,T,{\cQ}}^*} (\varphi\otimes I_\cH){\bf K_{q,T,{\cQ}}},
 \qquad \varphi\in B( \cN_J),
 \end{equation*}
where ${\bf K_{f,T,{\cQ}}}$ is the constrained Berezin kernel.   This Berezin \cite{Ber} type transform will play an important role in this paper.
We  show that the {\it pure} elements of the noncommutative variety
$\cV_{\cQ}(\cH)$ are detected by a class of completely positive linear maps. More precisely,
  given ${\bf T}=\{T_{i,j}\}\in B(\cH)^{n_1}\times \cdots \times B(\cH)^{n_k}$, we prove that  ${\bf T}$ is a pure element of
$\cV_{\cQ}(\cH)$ if and only if there is a unital completely positive and $w^*$-continuous linear map
$$
\Psi: \overline{\text{\rm span}}^{w^*} \{{\bf S}_{(\alpha)} {\bf S}_{(\beta)}^*: \
 (\alpha), (\beta) \in \FF_{n_1}^+\times \cdots \times \FF_{n_k}^+\}\to B(\cH)
 $$
 such that
 $$
 \Psi({\bf S}_{(\alpha)} {\bf S}_{(\beta)}^*)={\bf T}_{(\alpha)} {\bf T}_{(\beta)}^*,\qquad (\alpha), (\beta) \in \FF_{n_1}^+\times \cdots \times \FF_{n_k}^+.
 $$
Every map $\Psi$ with the above-mentioned  properties   is the constrained Berezin transform ${\bf B}_{{\bf T},\cQ}$  at a pure  element ${\bf T}\in \cV_\cQ(\cH)$.
A similar result (see Theorem \ref{C*-charact2}) characterizing the noncommutative variety $\cV_\cQ(\cH)$ is provided under the condition that $\cQ\subset \CC[Z_{i,j}]$ is a left ideal generated by homogeneous polynomials.

  In Section 2, we use the noncommutative Berezin transforms  to show that
a tuple ${\bf T}=\{T_{i,j}\}$ in $B(\cH)^{n_1}\times \cdots \times B(\cH)^{n_k}$ is a pure element in $\cV_\cQ(\cH)$ if and only if it is unitarily equivalent to the compression of a multiple of the universal model  to a co-invariant subspace.  In this case, we have
$$ { \bf T}_{(\alpha)}= {\bf B}_{{\bf T},\cQ}[{\bf S}_{(\alpha)}\otimes  I_\cD], \qquad (\alpha) \in \FF_{n_1}^+\times \cdots \times \FF_{n_k}^+,
    $$
the constrained Berezin kernel ${\bf K_{q,T,\cQ}}$ is an isometry, and the  subspace ${\bf K_{q,T,\cQ}}\cH$ is   co-invariant  under  each operator
${\bf S}_{i,j}\otimes  I_\cD$, where $\cD$ is the closure of the range of the defect operator ${\bf \Delta_{q,T}^m}(I)$.  For a certain class of noncommutative varieties $\cV_\cQ(\cH)$, this leads to a characterization  of  the pure elements ${\bf T}\in \cV_\cQ(\cH)$ with  $\dim \cD=n\in \NN$.  In particular, we obtain the following description and classification of
 the pure elements ${\bf T}\in \cV_\cQ(\cH)$ with $\dim \cD=1$. We show that they have the form
 ${\bf T}=\{P_\cM {\bf S}_{i,j}|_\cM\}$, where $\cM$ is  a co-invariant  subspace under each operator  ${\bf S}_{i,j}$.
Moreover,
if $\cM'$ is another co-invariant subspace under ${\bf S}_{i,j}$, which gives rise to an element ${\bf T}'\in \cV_\cQ(\cH)$, then ${\bf T}$
and ${\bf T}'$ are unitarily equivalent if and only if $\cM=\cM'$. This extends a result of Douglas and Foias \cite{DF} for the Hardy space  $H^2(\DD^n)$ over the polydisc.
We also  obtain  a characterization of the Beurling \cite{Be} type  joint invariant subspaces under the universal model ${\bf S}=\{{\bf S}_{i,j}\}$.   We prove that a subspace $\cM\subset \cN_\cQ\otimes \cH$    has the form  $\cM=M\left(\cN_\cQ\otimes \cE\right)$ for some partially isometric {\it multi-analytic} operator $M:\cN_\cQ\otimes \cE\to
\cN_\cQ\otimes \cH$  with respect to the universal model ${\bf S}$, i.e., $M({\bf S}_{i,j}\otimes I_\cH)=({\bf S}_{i,j}\otimes I_\cK)M$
for all $i,j$,  if and only if
 $$
  {\bf \Delta_{q,S\otimes{\rm I_\cH}}^p}(P_\cM)\geq 0,\qquad \text{ for any }\ {\bf p} \in \ZZ_+^k,  {\bf p}\leq {\bf m},
   $$
 where $P_\cM$ is the orthogonal projection of the Hilbert space
  $\cN_\cQ\otimes \cH$ onto $\cM$.

There is a strong connection between the  noncommutative varieties in polydomains, the theory of functions in several complex variables,  and the classical complex algebraic geometry. Note that the representation of the abstract variety $\cV_\cQ$ on the complex plane $\CC$ is the compact  set
$$
\cV_\cQ(\CC)={\bf D_q}(\CC)\cap \{\lambda\in \CC^n: \ g(\lambda)=0 \text{ for all } g\in \cQ\}
$$
and  ${\bf D^\circ_q}(\CC)=\{\lambda\in \CC^n: \ {\bf \Delta}_{{\bf q},\lambda}(1)> 0\}$ is a Reinhardt domain in $\CC^n$, where $n=n_1+\cdots+ n_k$ is the number of indeterminates in ${\bf q}=(q_1,\ldots, q_k)$.

In Section 3, we determine all the joint  invariant
subspaces of co-dimension one of  the universal model ${\bf S}=\{{\bf S}_{i,j}\}$. We show that the joint eigenvectors for
${\bf S}_{i,j}^*$  are
precisely the noncommutative constrained  Berezin kernels ${\bf K_{q,\lambda,{\cQ}}}$, where $\lambda\in
\cV_\cQ(\CC))\cap {\bf D^\circ_q}(\CC)$.
We introduce the variety algebra $\cA(\cV_\cQ)$  as the norm closed algebra generated by the ${\bf S}_{i,j}$ and the identity,  and the Hardy algebra $F^\infty(\cV_\cQ)$ as the WOT-closed version. We identify  the
$w^*$-continuous and multiplicative linear functionals of the Hardy algebra  $F^\infty(\cV_\cQ)$ as the maps, indexed by $\lambda\in \cV_\cQ(\CC)\cap {\bf D^\circ_q}(\CC)$, defined by
$\Phi_\lambda(A):= {\bf B}_{\lambda,\cQ}[A]$ for  $A\in F^\infty(\cV_\cQ)$.
If  $\cQ\subset \CC[Z_{i,j}]$
  is a left ideal generated by  noncommutative
    homogenous polynomials,
 then we show that  the right joint spectrum
$\sigma_r({\bf S})$ coincides with $\cV_\cQ(\CC)$. On the other hand, it turns out that the variety $\cV_\cQ(\CC)$   is homeomorphic to  the space  $M_{\cA(\cV_\cQ)}$ of all characters of the variety algebra $\cA(\cV_\cQ)$,  via the mapping $\lambda\mapsto \Phi_\lambda$, where $\Phi_\lambda$ is the evaluation functional.

Special attention is given to the commutative case when $\cQ=\cQ_c$, the left ideal generated by the commutators $Z_{i,j} Z_{s,t}-Z_{s,t} Z_{i,j}$ of the indeterminates
in $\CC[Z_{i,j}]$. In this case, the universal model associated with $\cV_{\cQ_c}$, denoted by ${\bf L}=\{{\bf L}_{i,j}\}$,  is acting on the Hilbert space $\cN_{\cQ_c}$ which coincides with  the closed span of all vectors  ${\bf K_{q,\lambda,{\cQ_c}}}$ with $ \lambda\in {\bf D^\circ_q}(\CC)\}$,   and it is identified with a Hilbert space  $H^2({\bf D^\circ_q}(\CC))$ of holomorphic functions on ${\bf D^\circ_q}(\CC)$, namely, the reproducing kernel Hilbert space  with kernel defined by
$$
\kappa_{\bf q}^c(\mu,\lambda):=\frac{1}{\prod_{i=1}^k\left(1- q_i(\mu_i\overline{\lambda}_i )\right)^{m_i}}, \qquad \mu,\lambda\in {\bf D^\circ_q}(\CC).
$$
We prove that the Hardy algebra $F^\infty(\cV_{\cQ_c})$  is reflexive and  coincides with the multiplier
algebra of  the Hilbert space $H^2({\bf D^\circ_q}(\CC))$. Under this identification, ${\bf L}_{i,j}$ is the multiplier by the coordinate function $\lambda_{i,j}$.
 We remark that when $n_1=\cdots =n_k$ and $\cQ_{cc}$ is  the left ideal generated by $\cQ_c$ and the polynomials
 $Z_{i,j}-Z_{p,j}$, the universal model associated with $\cV_{\cQ_{cc}}$ is acting on the Hilbert space $\cN_{\cQ_{cc}}$ which can
 be identified with the  reproducing kernel Hilbert space with kernel
$$
\kappa_{\bf q}^{cc}(z,w):=\frac{1}{\prod_{i=1}^k\left(1- q_i(z\overline{w} )\right)^{m_i}},\qquad z, w\in \cap_{i=1}^k{\bf D}^\circ_{q_i}(\CC).
$$
  In the particular case when $f_1=\cdots =f_k=Z_1+\cdots +Z_n$ and $m_1=\cdots=m_k=1$, we obtain   the reproducing kernel
  $(z,w)\mapsto \frac{1}{(1-\left<z,w\right>)^k}$ on the unit ball $\BB_n$. In this case, the reproducing kernel Hilbert spaces  are the  Hardy-Sobolev spaces (see \cite{BeaBur1}), which include the Drurry-Arveson space (see \cite{Dru}, \cite{Arv1}, \cite{DP}, \cite{Po-poisson}), the  Hardy space of the ball and the Bergman space (see \cite{Ru}). All the results of this paper are true in  these commutative settings.

In Section 4, we show that the isomorphism  problem for the universal polydomain algebras is closed connected to to the biholomorphic  equivalence of Reinhardt domains in several complex variables.
Let ${\bf q}=(q_1,\ldots, q_k)$ and ${\bf g}=(g_1,\ldots, g_{k'})$ be  tuples of positive regular polynomials  with $n$ and $\ell$ indeterminates, respectively, and let ${\bf m}\in \NN^k$ and ${\bf d}\in \NN^{k'}$. We prove that  if   the polydomain algebras $\cA({\bf D_q^m})$ and
 $\cA({\bf D_g^d})$  are   unital completely contractive
  isomorphic, then  the Reinhardt domains ${\bf
D^\circ_{q}}(\CC)$   and ${\bf D^\circ_{g}}(\CC)$  are biholomorphic equivalent and  $n=\ell$.
     A similar result holds for the commutative  variety algebras $\cA({\bf \cV_{q, \cQ_c}^m})$ and
 $\cA({\bf \cV_{g, \cQ_c}^d})$. We remark that  when  ${\bf q}=Z_1+\cdots + Z_n$ and ${\bf g}=(Z_1,\ldots, Z_n)$, the corresponding domain algebras are   the universal algebra of  a commuting row contraction $\cA(\cV_{{\bf q},\cQ_c}^1)$ and  the commutative polydisc algebra  $\cA(\cV_{{\bf g},\cQ_c}^1)$, respectively.
  Since $\BB_n$ and $\DD^n$ are not biholomorphic equivalent domains in $\CC^n$ if $n\geq 2$ (see \cite{Kr}), our result implies that the two algebras  are not isomorphic. The classification problem for polydomain algebras will be pursued in a future paper.

In Section 5,  we develop a dilation theory  for  noncommutative varieties in polydomains.
 For the class of noncommutative varieties   $\cV_\cQ(\cH)$, where $\cQ\subset \CC[Z_{i,j}]$ is an ideal generated by homogeneous  polynomials, the dilation theory is  refined.  In this case, we  obtain  Wold type
decompositions for non-degenerate $*$-representations of the
$C^*$-algebra $C^*( \cV_\cQ)$ generated by the universal model  ${\bf S}_{i,j}$ and the identity, and coisometric    dilations for the elements of  $\cV_\cQ(\cH)$. Under natural conditions, the dilation is unique up to unitary equivalence. In the  particular case when $k={\bf m}=1$, ${\bf q}=Z_1+\cdots + Z_n$, and $\cQ=\cQ_c$, we recover Arveson's results \cite{Arv1} concerning the dilation  theory for commuting row contractions.

In the last section of this paper,
we  provide a characterization for the class of tuples of  operators in  the noncommutative variety $\cV_\cQ(\cH)$ which admit constrained  characteristic functions. In this case,
  the characteristic function is  a complete unitary invariant  for  the  completely non-coisometric tuples. We  also provide operator models  in terms of  the  constrained characteristic functions. These results  extend the corresponding ones from \cite{SzFBK-book}, \cite{Po-charact}, \cite{Po-varieties}, \cite{Po-varieties2}, \cite{BES1}, \cite{BS},  \cite{Po-domains}, and \cite{Po-Berezin-poly}, to varieties in noncommutative polydomains.

We   remark that  the results of this paper are presented in a more general setting,
 when ${\bf q}$ is replaced by a $k$-tuple ${\bf f}=(f_1,\ldots, f_k)$ of
  positive regular free holomorphic functions in a neighborhood of the origin, and $\cQ$ is replaced by a WOT-closed left ideal of the Hardy algebra $F^\infty({\bf D_f^m})$.

  We mention that noncommutative varieties in {\it ball-like domains} were studied   in several papers (see \cite{ArPo2},   \cite{Po-varieties}, \cite{Po-varieties2}, \cite{Po-Berezin}, \cite{Po-unitary}, \cite{Po-domains},  \cite{Po-similarity-domains}, and the references there in).
   The commutative case when $m_1\geq 2$, $n_1\geq 2$, and  $q_1=Z_1+\cdots + Z_n$,  was studied
   by Athavale \cite{At1}, M\" uller \cite{M}, M\" uller-Vasilescu \cite{MVa},
   Vasilescu \cite{Va1}, and Curto-Vasilescu \cite{CV1}.
   Some  of these results  were extended by
   S.~Pott \cite{Pot} when $q_1$ is a positive regular  polynomial in commuting indeterminates (see also \cite{BS}).
The commutative polydisc case, i.e, $k\geq 2$, $n_1=\cdots=n_k=1$, and ${\bf q}=(Z_1,\ldots,Z_n)$,  was first considered by Brehmer \cite{Br} in connection with regular dilations. Motivated by Agler's work \cite{Ag2} on weighted shifts as model operators, Curto and Vasilescu developed a theory of standard operator models in
 the polydisc in \cite{CV2}, \cite{CV3}. Timotin \cite{T}  obtained some of their results from Brehmer's theorem. The polyball case, when $k\geq 2$ and $q_i=Z_1+\cdots +Z_{n_i}$, $i\in \{1,\ldots,k\}$, was considered in \cite{Po-poisson}  and \cite{BeTi2} for the noncommutative and commutative case, respectively.

\bigskip

\section{Noncommutative  Varieties in polydomains  and   Berezin transforms}

In  this section, we consider noncommutative varieties
${\bf \cV_{f,{\it J}}^m}(\cH)\subset {\bf D_f^m}(\cH)$ determined by  left ideals $J$  in either one of the following algebras:  $\CC[Z_{i,j}]$, $\CC[{\bf W}_{i,j}]$, $\cA({\bf D_f^m})$, or $F^\infty({\bf D_f^m})$.   We associate  with each such
a variety a universal model ${\bf S}=({\bf S}_1,\ldots,
{\bf S}_n)\in {\bf \cV_{f,{\it J}}^m}(\cN_J)$, where $\cN_J$ is an appropriate
subspace of  a tensor product of  full Fock  spaces. We introduce a {\it constrained
noncommutative Berezin transform} and use it to characterize  noncommutative varieties in   polydomains.

 We begin by recalling  from \cite{Po-Berezin-poly}  some  definitions  and basic properties  of the universal model associated with the abstract noncommutative polydomain ${\bf D_f^m}$ and of the associated Berezin kernel.

For each $i\in \{1,\ldots, k\}$,
let $\FF_{n_i}^+$ be the unital free semigroup on $n_i$ generators
$g_{1}^i,\ldots, g_{n_i}^i$ and the identity $g_{0}^i$.  The length of $\alpha\in
\FF_{n_i}^+$ is defined by $|\alpha|:=0$ if $\alpha=g_0^i$  and
$|\alpha|:=p$ if
 $\alpha=g_{j_1}^i\cdots g_{j_p}^i$, where $j_1,\ldots, j_p\in \{1,\ldots, n_i\}$.
If $Z_{i,1},\ldots,Z_{i,n_i}$  are  noncommuting indeterminates,   we
denote $Z_{i,\alpha}:= Z_{i,j_1}\cdots Z_{i,j_p}$  and $Z_{i,g_0^i}:=1$.
 Let  $f_i:= \sum_{\alpha\in
\FF_{n_i}^+} a_{i,\alpha} Z_\alpha$, \ $a_{i,\alpha}\in \CC$,  be a formal power series in $n_i$ noncommuting indeterminates $Z_{i,1},\ldots,Z_{i,n_i}$. We say that $f_i$ is
a {\it
positive regular free holomorphic function} if
$a_{i,\alpha}\geq 0$ for
any $\alpha\in \FF_{n_i}^+$, \ $a_{i,g_{0}^i}=0$,
   $a_{i,g_{j}^i}>0$ for  $j\in \{1,\ldots, n_i\}$, and
 $
\limsup_{k\to\infty} \left( \sum_{|\alpha|=k}
|a_{i,\alpha}|^2\right)^{1/2k}<\infty.
 $
 Throughout this paper, we denote by $B(\cH)$ the algebra of bounded
linear operators on a separable Hilbert space $\cH$.
Given $X_i:=(X_{i,1},\ldots, X_{i,n_i})\in B(\cH)^{n_i}$, define the map $\Phi_{f_i,X_i}:B(\cH)\to B(\cH)$  by setting
 $$
\Phi_{f_i,X_i}(Y):=\sum_{k=1}^\infty\sum_{\alpha\in \FF_{n_i}^+,|\alpha|=k} a_{i,\alpha} X_{i,\alpha}
YX_{i,\alpha}^*, \qquad   Y\in B(\cH),$$
 where  the convergence is in the week
operator topology.
Let ${\bf n}:=(n_1,\ldots, n_k)$ and ${\bf m}:=(m_1,\ldots, m_k)$, where $n_i,m_i\in\NN:=\{1,2,\ldots\}$ and $i\in \{1,\ldots, k\}$, and let ${\bf f}:=(f_1,\ldots,f_k)$ be a $k$-tuple of positive regular free holomorphic functions.
 We associate with each  element ${\bf X}=(X_1,\ldots, X_k)\in B(\cH)^{n_1}\times\cdots \times B(\cH)^{n_k}$  and ${\bf p}=(p_1,\ldots, p_k)\in \ZZ_+^k$ the {\it defect mapping} ${\bf \Delta_{f,X}^p}:B(\cH)\to  B(\cH)$ defined by
$$
{\bf \Delta_{f,X}^p}:=\left(id -\Phi_{f_1, X_1}\right)^{p_1}\circ \cdots \circ\left(id -\Phi_{f_k, X_k}\right)^{p_k}.
$$
 We use the convention that $(id-\Phi_{f_i,X_i})^0=id$.
We denote by $B(\cH)^{n_1}\times_c\cdots \times_c B(\cH)^{n_k}$
   the set of all tuples  ${\bf X}=({ X}_1,\ldots, { X}_k)\in B(\cH)^{n_1}\times\cdots \times B(\cH)^{n_k}$, where ${ X}_i:=(X_{i,1},\ldots, X_{i,n_i})\in B(\cH)^{n_i}$, $i\in \{1,\ldots, k\}$,
     with the property that, for any $p,q\in \{1,\ldots, k\}$, $p\neq q$, the entries of ${ X}_p$ are commuting with the entries of ${ X}_q$. In this case we say that ${ X}_p$ and ${ X}_q$ are commuting tuples of operators. Note that, for each $i\in \{1,\ldots,k\}$,  the operators $X_{i,1},\ldots, X_{i,n_i}$ are not necessarily commuting.

In \cite{Po-Berezin-poly},  we developed an operator  model  theory and a theory of free holomorphic functions on  the noncommutative polydomain
$$
{\bf D_f^m}(\cH):=\left\{ {\bf X}=(X_1,\ldots, X_k)\in B(\cH)^{n_1}\times_c\cdots \times_c B(\cH)^{n_k}: \ {\bf \Delta_{f,X}^p}(I)\geq 0 \ \text{ for }\ {\bf 0}\leq {\bf p}\leq {\bf m}\right\}.
$$
 Throughout this paper, we refer to ${\bf D_f^m}:=\{{\bf D_f^m}(\cH): \ \cH \, \text{is a Hilbert space}\}$ as the {\it abstract noncommutative polydomain}, while ${\bf D_f^m}(\cH)$ is its representation on the Hilbert space $\cH$.

 Let $H_{n_i}$ be
an $n_i$-dimensional complex  Hilbert space with orthonormal basis
$e_1^i,\dots,e_{n_i}^i$.
  We consider the full Fock space  of $H_{n_i}$ defined by
$$F^2(H_{n_i}):=\CC 1 \oplus\bigoplus_{p\geq 1} H_{n_i}^{\otimes p},$$
where   $H_{n_i}^{\otimes p}$ is the
(Hilbert) tensor product of $p$ copies of $H_{n_i}$. Set $e_\alpha^i :=
e^i_{j_1}\otimes \cdots \otimes e^i_{j_p}$ if
$\alpha=g^i_{j_1}\cdots g^i_{j_p}\in \FF_{n_i}^+$
 and $e^i_{g^i_0}:= 1\in \CC$.
Note that $\{e^i_\alpha:\alpha\in\FF_{n_i}^+\}$ is an orthonormal
basis of $F^2(H_{n_i})$.
Let $m_i, n_i\in \NN:=\{1,2,\ldots\}$, $i\in\{1,\ldots,k\}$, and  $j\in\{1,\ldots,n_i\}$. We  define
 the {\it weighted left creation  operators} $W_{i,j}:F^2(H_{n_i})\to
F^2(H_{n_i})$,   associated with the abstract noncommutative
domain  ${\bf D}_{f_i}^{m_i} $    by setting
\begin{equation*}
W_{i,j} e_\alpha^i:=\frac {\sqrt{b_{i,\alpha}^{(m_i)}}}{\sqrt{b_{i,g_j
\alpha}^{(m_i)}}} e^i_{g_j \alpha}, \qquad \alpha\in \FF_{n_i}^+,
\end{equation*}
where
\begin{equation} \label{b-coeff}
  b_{i,g_0}^{(m_i)}:=1\quad \text{ and } \quad
 b_{i,\alpha}^{(m_i)}:= \sum_{p=1}^{|\alpha|}
\sum_{{\gamma_1,\ldots,\gamma_p\in \FF_{n_i}^+}\atop{{\gamma_1\cdots \gamma_p=\alpha }\atop {|\gamma_1|\geq
1,\ldots, |\gamma_p|\geq 1}}} a_{i,\gamma_1}\cdots a_{i,\gamma_p}
\left(\begin{matrix} p+m_i-1\\m_i-1
\end{matrix}\right)
\end{equation}
for all  $ \alpha\in \FF_{n_i}^+$ with  $|\alpha|\geq 1$.
For each $i\in \{1,\ldots, k\}$ and $j\in \{1,\ldots, n_i\}$, we
define the operator ${\bf W}_{i,j}$ acting on the tensor Hilbert space
$F^2(H_{n_1})\otimes\cdots\otimes F^2(H_{n_k})$ by setting
$${\bf W}_{i,j}:=\underbrace{I\otimes\cdots\otimes I}_{\text{${i-1}$
times}}\otimes W_{i,j}\otimes \underbrace{I\otimes\cdots\otimes
I}_{\text{${k-i}$ times}}.
$$
 The $k$-tuple ${\bf W}:=({\bf W}_1,\ldots, {\bf W}_k)$, where  ${\bf W}_i:=({\bf W}_{i,1},\ldots,{\bf W}_{i,n_i})$, is  an element  in the
noncommutative polydomain $ {\bf D_f^m}(\otimes_{i=1}^kF^2(H_{n_i}))$ and it
  is  called the {\it universal model} associated
  with the abstract noncommutative
  polydomain ${\bf D_f^m}$.
We say that ${\bf T}=({ T}_1,\ldots, { T}_k)\in {\bf D_f^m}(\cH)$ is
 { \it completely non-coisometric}    if there is no $h\in \cH$, $h\neq 0$ such that
$$
\left< (id-\Phi_{f_1,T_1}^{q_1})\cdots (id-\Phi_{f_k,T_k}^{q_k})(I_\cH)h, h\right>=0$$
for any $(q_1,\ldots, q_k)\in \NN^k$.  The $k$-tuple
${\bf T} $ is called  {\it pure} if
$$
\lim_{{\bf q}=(q_1,\ldots, q_k)\in \ZZ_+^k}(id-\Phi_{f_k,T_k}^{q_k})\cdots (id-\Phi_{f_1,T_1}^{q_1})(I)=I.
$$

The {\it noncommutative Berezin kernel} associated with any element
   ${\bf T}=\{T_{i,j}\}$ in the noncommutative polydomain ${\bf D_f^m}(\cH)$ is the operator
   $${\bf K_{f,T}}: \cH \to F^2(H_{n_1})\otimes \cdots \otimes  F^2(H_{n_k}) \otimes  \overline{{\bf \Delta_{f,T}^m}(I) (\cH)}$$
   defined by
   $$
   {\bf K_{f,T}}h:=\sum_{\beta_i\in \FF_{n_i}^+, i=1,\ldots,k}
   \sqrt{b_{1,\beta_1}^{(m_1)}}\cdots \sqrt{b_{k,\beta_k}^{(m_k)}}
   e^1_{\beta_1}\otimes \cdots \otimes  e^k_{\beta_k}\otimes {\bf \Delta_{f,T}^m}(I)^{1/2} T_{1,\beta_1}^*\cdots T_{k,\beta_k}^*h,
   $$
where the defect operator is defined by
$$
{\bf \Delta_{f,T}^m}(I)  :=(id-\Phi_{f_1,T_1})^{m_1}\cdots (id-\Phi_{f_k,T_k})^{m_k}(I),
$$
and the coefficients  $b_{1,\beta_1}^{(m_1)}, \ldots, b_{k,\beta_k}^{(m_k)}$
are given by relation \eqref{b-coeff}.
The noncommutative Berezin kernel  ${\bf K_{f,T}}$ is a contraction  and
$$
{\bf K_{f,T}^*}{\bf K_{f,T}}=
\lim_{q_k\to\infty}\ldots \lim_{q_1\to\infty}  (id-\Phi_{f_k,T_k}^{q_k})\cdots (id-\Phi_{f_1,T_1}^{q_1})(I),
$$
where the limits are in the weak  operator topology.
 Moreover, for any $i\in \{1,\ldots, k\}$ and $j\in \{1,\ldots, n_i\}$,  $${\bf K_{f,T}} { T}^*_{i,j}= ({\bf W}_{i,j}^*\otimes I)  {\bf K_{f,T}}.
    $$

The {\it  noncommutative Berezin transform}
 at   ${\bf T}\in {\bf D_f^m}(\cH)$ is the mapping
 $ {\bf B_{T}}: B(\otimes_{i=1}^k F^2(H_{n_i}))\to B(\cH)$
 given by

 \begin{equation*}
 \label{def-Be2}
 {\bf B_{T}}[g]:= {\bf K^*_{f,T}} (g\otimes I_\cH) {\bf K_{f,T}},
 \qquad g\in B(\otimes_{i=1}^k F^2(H_{n_i})).
 \end{equation*}
  The polydomain algebra $\cA({\bf D_f^m})$ is the norm closed algebra generated by ${\bf W}_{i,j}$ and the identity.
Let
  $$\cS:=\overline{\text{\rm  span}} \{{\bf W}_{(\alpha)} {\bf W}_{(\beta)}^*:\
(\alpha), (\beta) \in \FF_{n_1}^+\times \cdots \times \FF_{n_k}^+\},
$$
where the closure is in the operator norm.
   We proved in \cite{Po-Berezin-poly} that  there is
    a unital completely contractive linear map
$
\Psi_{{\bf f,T}}:\cS \to B(\cH)
$
such that
\begin{equation*}
\Psi_{{\bf f,T}}(g)=\lim_{r\to 1} {\bf B}_{r{\bf T}}[g],\qquad g\in \cS,
\end{equation*}
where the limit exists in the norm topology of $B(\cH)$, and
$$
\Psi_{\bf f,T}  ({\bf W}_{(\alpha)} {\bf W}_{(\beta)}^*)=  {\bf T}_{(\alpha)} {\bf T}_{(\beta)}^*,\qquad
(\alpha), (\beta) \in \FF_{n_1}^+\times \cdots \times \FF_{n_k}^+,
$$
where  ${\bf W}_{(\alpha)}:= {\bf W}_{1,\alpha_1}\cdots {\bf W}_{k,\alpha_k}$ for $(\alpha):=(\alpha_1,\ldots, \alpha_k)$.
 In particular, the restriction of
$\Psi_{\bf f,T}$ to the polydomain algebra $\cA({\bf D_f^m})$ is a
completely contractive homomorphism. For information on completely bounded (resp. positive) maps, we refer to \cite{Pa-book}.

The noncommutative Hardy  algebra $F^\infty({\bf D^m_f})$ is the sequential  SOT-(resp.~WOT-, $w^*$-) closure of all polynomials in ${\bf W}_{i,j}$   and
the identity, where   $i\in \{1,\ldots,k\}$, $j\in \{1,\ldots, n_k\}$.
Each elemeny    $\varphi({\bf W}_{i,j})$ in  $ F^\infty({\bf D_f^m})$ has a unique  Fourier type representation
$$\varphi({\bf W}_{i,j})=\sum\limits_{(\beta) \in \FF_{n_1}^+\times \cdots \times \FF_{n_k}^+} c_{(\beta)}
{\bf W}_{(\beta)},\qquad c_{(\beta)}\in \CC, $$
and $\varphi({\bf W}_{i,j})
     =\text{\rm SOT-}\lim_{r\to 1} \varphi(r{\bf W}_{i,j}),
     $
 where $\varphi(r{\bf W}_{i,j})$ is in the   polydomain algebra $\cA({\bf D_f^m})$.
We recall \cite{Po-Berezin-poly} the following result concerning the   $F^\infty({\bf D_f^m})$--{\it functional calculus}
 for the completely non-coisometric  part of the noncommutative polydomain ${\bf D^m_f}(\cH)$.
Let ${\bf T}=({ T}_1,\ldots, { T}_k)$ be  a completely
non-coisometric $k$-tuple
      in  the noncommutative polydomain ${\bf D_f^m}(\cH)$.  Then
 $$\Psi_T(\varphi):=\text{\rm SOT-}\lim_{r\to 1} \varphi(rT_{i,j}), \qquad
  \Psi_T=\varphi({\bf W}_{i,j})\in F^\infty({\bf D_f^m}),
 $$
 exists in the strong operator topology   and defines a map
 $\Psi_T:F^\infty({\bf D_f^m})\to B(\cH)$ with the following
 properties:
\begin{enumerate}
\item[(i)]
$\Psi_T(\varphi)=\text{\rm SOT-}\lim\limits_{r\to 1}{\bf B}_{r{\bf T}}[\varphi]$, where ${\bf
B}_{r{\bf T}}$ is the    Berezin transform  at  $r{\bf T}\in {\bf D_f^m}(\cH)$;
\item[(ii)] $\Psi_T$ is    WOT-continuous (resp.
SOT-continuous)  on bounded sets;
\item[(iii)]
$\Psi_T$ is a unital completely contractive homomorphism and
$$\Psi_{\bf T}({\bf W}_{(\beta)})={T}_{(\beta)}, \qquad (\beta) \in \FF_{n_1}^+\times \cdots \times \FF_{n_k}^+
$$
\end{enumerate}
If ${\bf T}$ is a pure $k$-tuple, then $\Psi_T(\varphi)={\bf B_T}[\varphi]$.

For each $i\in\{1,\ldots, k\}$, let $Z_i:=(Z_{i,1},\ldots, Z_{i,n_i})$ be
an  $n_i$-tuple of noncommuting indeterminates and assume that, for any
$s,t\in \{1,\ldots, k\}$, $s\neq t$, the entries in $Z_s$ are commuting
 with the entries in $Z_t$.  The algebra of all polynomials  in indeterminates $Z_{i,j}$ is denoted by $\CC[Z_{i,j}]$.

Let  ${\bf W}:=\{{\bf W}_{i,j}\}$  be   the {\it universal model} associated
  with the abstract noncommutative
  polydomain ${\bf D_f^m}$. If $\cQ$ is a left ideal of     polynomials  in $\CC[Z_{i,j}]$, we let $\cQ_{\bf W}:=\{q({\bf W}_{i,j}):\ q\in \cQ\}$ be the  corresponding ideal
  in the algebra $\CC[{\bf W}_{i,j}]$ of all polynomials in ${\bf W}_{i,j}$ and the identity.
   Using the $\cA({\bf D_f^m})$-functional calculus, one can easily show that the norm-closed left ideal generated by  $\cQ_{\bf W}$ in the polydomain algebra $\cA({\bf D_f^m})$ coincides with the norm closure ${\overline \cQ}_{\bf W}$. Similarly, using the $F^\infty({\bf D_f^m})$--functional calculus, one can prove that
  the   WOT-closed left ideal  generated by  $\cQ_{\bf W}$ in the Hardy algebra $F^\infty({\bf D_f^m})$ coincides with   ${\overline \cQ}_{\bf W}^{wot}$.
If $J$ is a left ideal in $\CC[{\bf W}_{i,j}]$, $\cA({\bf D_f^m})$, or $F^\infty({\bf D_f^m})$, we introduce the subspace $\cM_J$ to be the closed image  of $J$ in $\otimes_{i=1}^k F^2(H_{n_i})$, i.e.,
  $\cM_J:=\overline{J(\otimes_{i=1}^k F^2(H_{n_i}))}$.  We also introduce the space
  $$
  \cN_J:= [\otimes_{i=1}^k F^2(H_{n_i})]\ominus \cM_J.
  $$
  When $\cQ$ is a  left ideal of     polynomials  in $\CC[Z_{i,j}]$, we set $\cM_\cQ:= \cM_{\cQ_{\bf W}}$ and $\cN_\cQ:= [\otimes_{i=1}^k F^2(H_{n_i})]\ominus \cM_\cQ$. We remark that in this case we have
  $$
  \cN_\cQ=\cN_{{\overline \cQ}_{\bf W}}=\cN_{{\overline \cQ}_{\bf W}^{wot}}.
  $$

   To simplify our notation, throughout this paper, unless otherwise specified, we consider $J$ to denote a left ideal in either one of the following algebras:  $\CC[Z_{i,j}]$, $\CC[{\bf W}_{i,j}]$, $\cA({\bf D_f^m})$, or $F^\infty({\bf D_f^m})$.
 We always  assume that $\cN_J\neq
 \{0\}$. It is easy to see that $\cN_J$ is invariant under each
 operator ${\bf W}_{i,j}^* $  for $i\in \{1,\ldots, k\}$, $j\in \{1,\ldots, n_i\}$.  Define ${\bf S}_{i,j}:=P_{\cN_J}
{\bf W}_{i,j}|_{\cN_J}$, where
 $P_{\cN_J}$ is the orthogonal projection of $\otimes_{i=1}^k F^2(H_{n_i})$ onto $\cN_J$.
 Using the properties of  the universal model ${\bf W}=\{{\bf W}_{i,j}\}$  and the fact that  $\cN_J$ is invariant under each
 operator ${\bf W}_{i,j}^*$, one can obtain  the following result.

 \begin{lemma} \label{univ-mod-var} Let $J $ be a left ideal in either one of the following algebras:  $\CC[Z_{i,j}]$, $\CC[{\bf W}_{i,j}]$, $\cA({\bf D_f^m})$, or $F^\infty({\bf D_f^m})$. The $k$-tuple ${\bf S}:=({\bf S}_1,\ldots, {\bf S}_k)$, where ${\bf S}_i:=({\bf S}_{i,1}\ldots, {\bf S}_{i,n_i})$  and
  ${\bf S}_{i,j}:=P_{\cN_J}
{\bf W}_{i,j}|_{\cN_J}$ has the following properties.
\begin{enumerate}
\item[(i)] ${\bf S}$ is a pure tuple in the polydomain ${\bf D_f^m}(\cN_J)$.
\item[(ii)] Under  the $F^\infty({\bf D_f^m})$-functional calculus,
 $$
 g({\bf S}_1,\ldots, {\bf S}_k)=0,\qquad g\in \overline{J}^{wot}.
 $$
\item[(iii)] If \,${\bf P}_\CC$ denotes  the
 orthogonal projection from $\otimes_{i=1}^k F^2(H_{n_i})$ onto $\CC 1$,  then
 $$(id-\Phi_{f_1,{\bf S}_1})^{m_1}\cdots (id-\Phi_{f_k,{\bf S}_k})^{m_k}(I_{\cN_J})=P_{\cN_J}{\bf P}_\CC|_{\cN_J}.
 $$
\end{enumerate}
 \end{lemma}
\begin{proof} Since $\cN_J$ is invariant under each operator  ${\bf W}_{i,j}^*$, we have
$\Phi_{f_i, {\bf S}_i}^{q_i}(I)=P_{\cN_J}\Phi_{f_i, {\bf W}_i}^{q_i}(I)|_{\cN_J}$.
Taking into account that ${\bf W}$ is a pure  element in ${\bf D_f^m}(\otimes_{i=1}^k F^2(H_n))$, we deduce that
SOT-$\lim_{q_i\to \infty}\Phi_{f_i, {\bf W}_i}^{q_i}(I)=0$, which implies that ${\bf S}$ is a pure tuple in the polydomain ${\bf D_f^m}(\cN_J)$.
To prove part (ii), note that if $g({\bf W}_{i,j})\in \overline{J}^{wot}$, then
the range of $g({\bf W}_{i,j})$ is in $ \cN_J$. Using the $F^\infty({\bf D_f^m})$-functional calculus, we deduce that
$$
g({\bf S}_1,\ldots, {\bf S}_k)=\text{\rm SOT-}\lim_{r\to 1}g(r{\bf S}_{i,j})=
\text{\rm SOT-}\lim_{r\to 1}P_{\cN_J}g(r{\bf W}_{i,j})|_{\cN_J}=P_{\cN_J}g({\bf W}_{i,j})|_{\cN_J}=0.
$$
Part (iii) follows from the fact that ${\bf \Delta_{f,W}^m}(I)={\bf P}_\CC$ and $\cN_J$ is invariant under each operator  ${\bf W}_{i,j}^*$. Indeed, we have
${\bf \Delta_{f,S}^m}(I)=P_{\cN_J}{\bf \Delta_{f,W}^m}(I)|_{\cN_J}=P_{\cN_J}{\bf P}_\CC|_{\cN_J}$.
\end{proof}

  We define the noncommutative variety ${\bf \cV_{f,{\it J}}^m}(\cH)$ in the polydomain  $ {\bf D_f^m}(\cH)$ by setting
 $$
{\bf \cV_{f,{\it J}}^m}(\cH):= \left\{{\bf X}=\{X_{i,j}\} \in {\bf D_f^m}(\cH):\
g({\bf X} )=0\  \text{ for any }\  g\in J\right\}.
$$
We remark that  this variety   is well-defined if  $J$ is a left ideal in    $\CC[Z_{i,j}]$, $\CC[{\bf W}_{i,j}]$, or $\cA({\bf D_f^m})$. In the case when $J$ is a  WOT-closed left ideal in  $F^\infty({\bf D_f^m})$, we can use the $F^\infty({\bf D_f^m})$--functional calculus to define
the   variety ${\bf \cV_{f,{\it J,cnc}}^m}(\cH)$ of all completely non coisometric  (c.n.c.)  tuples ${\bf X}\in {\bf D_f^m}(\cH)$  satisfying the equation $g({\bf X})=0$ for any $g\in J$.

According to Lemma \ref{univ-mod-var}, the $k$-tuple ${\bf S}:=({\bf S}_1,\ldots, {\bf S}_k)$ is in the noncommutative variety
 $ {\bf \cV_{f,{\it J}}^m}(\cN_J)$. We remark that ${\bf S}$ will play the role of {\it universal model}
   for the {\it abstract  noncommutative variety}
   $${\bf \cV_{f,{\it J}}^m}:=\{{\bf \cV_{f,{\it J}}^m}(\cH):\ \cH \text { is a Hilbert space}\}.
   $$

We introduce the {\it constrained noncommutative  Berezin kernel}
associated with  ${\bf T}\in {\bf \cV_{f,{\it J}}^m}(\cH)$  as   the
bounded operator  \ ${\bf K_{f,T,{\it J}}}:\cH\to \cN_J\otimes
\overline{{\bf \Delta_{f,T}^m}(I) (\cH)}$ defined by
$${\bf K_{f,T, {\it J }}}:=\left(P_{\cN_J}\otimes I_{\overline{{\bf \Delta_{f,T}^m}(I) (\cH)}}\right){\bf K_{f,T}},
$$
where ${\bf K_{f,T}}$ is the noncommutative Berezin kernel associated with ${\bf T}\in
{\bf D_f^m}(\cH)$. The next result shows that the main properties of the noncommutative Berezin kernel  remain true for the constrained Berezin kernel associated with the elements of  the  noncommutative variety ${\bf \cV_{f,{\it J}}^m}(\cH)$.
\begin{proposition} \label{Ber-const}
Let ${\bf T}=({ T}_1,\ldots, { T}_k)$, with  $T_i:=(T_{i,1},\ldots, T_{i,n_i})$, be in the noncommutative variety $ {\bf \cV_{f,{\it J}}^m}(\cH)$, where $J$ is a  left ideal in $\CC[Z_{i,j}]$, $\CC[{\bf W}_{i,j}]$, or $\cA({\bf D_f^m})$.  The constrained noncommutative  Berezin kernel associated with
 ${\bf T} $ has the following properties.
\begin{enumerate}
\item[(i)] ${\bf K_{f,T, {\it J}}}$ is a contraction  and
$$
{\bf K_{f,T,{\it J}}^*}{\bf K_{f,T, {\it J }}}=
\lim_{q_k\to\infty}\ldots \lim_{q_1\to\infty}  (id-\Phi_{f_k,T_k}^{q_k})\cdots (id-\Phi_{f_1,T_1}^{q_1})(I),
$$
where the limits are in the weak  operator topology.
\item[(ii)]  For any $i\in \{1,\ldots, k\}$ and $j\in \{1,\ldots, n_i\}$,  $${\bf K_{f,T, {\it J }}} { T}^*_{i,j}= ({\bf S}_{i,j}^*\otimes I) {\bf K_{f,T, {\it J }}}.
    $$
    \item[(iii)]  If ${\bf T}$ is pure,   then
$${\bf K_{f,T, {\it J}}^*}{\bf K_{f,T, {\it J }}}=I_\cH. $$
\end{enumerate}
If $J$ is a  WOT-closed left ideal in  $F^\infty({\bf D_f^m})$ and  ${\bf T}\in{\bf \cV_{f,{\it J,cnc}}^m}(\cH)$, all the properties above remain true.
\end{proposition}
\begin{proof}
Since     ${\bf K_{f,T}} { T}^*_{i,j}= ({\bf W}_{i,j}^*\otimes I)  {\bf K_{f,T}}$
for any $i\in \{1,\ldots, k\}$ and $j\in \{1,\ldots, n_i\}$, we deduce that
\begin{equation} \label{K-eq}
\left<  {\bf K_{f,T}}x,  q({\bf W}_{i,j}){\bf W}_{(\alpha)}
(1)\otimes y\right>=\left<x, q(T_{i,j})T_{(\alpha)}
{\bf K_{f,T}^*}(1\otimes y\right>=\left<x, q(T_{i,j})T_{(\alpha)}
{\bf \Delta_f^m}(I)^{1/2} y\right>
\end{equation}
for any $x\in \cH$, $y\in \overline{{\bf \Delta_{f,T}^m}(I)\cH}$, $(\alpha) \in \FF_{n_1}^+\otimes \cdots \otimes \FF_{n_k}^+$, and any polynomial $q({\bf W}_{i,j})\in \CC[{\bf W}_{i,j}]$.  Consequently, if $J$ is a  left ideal in $\CC[Z_{i,j}]$ or  $\CC[{\bf W}_{i,j}]$, then $q(T_{i,j})=0$ for any $q\in J$ and therefore
\begin{equation}\label{range}
\text{\rm range}\,{\bf K_{f,T}}\subseteq \cN_J\otimes
\overline{{\bf \Delta_{f,T}^m}(I)\cH}.
\end{equation}
 Assume that $J$ is a norm-closed left ideal of $\cA({\bf D_f^m})$ and let $g({\bf W}_{i,j})\in J$. Choose a sequence of polynomials  $q_n({\bf W}_{i,j})$ which converges in norm to $g({\bf W}_{i,j})$.  This implies that $q_n(T_{i,j})$ converges in norm to
$g({T}_{i,j})$.  Using equation \eqref{K-eq}, we deduce a similar one where $q({\bf W}_{i,j})$ is replaced by $g({\bf W}_{i,j})$. As above, we deduce that  relation \eqref{range} remains true in this case. Now, we consider the case when $J$ is a  WOT-closed left ideal in  $F^\infty({\bf D_f^m})$ and  ${\bf T}\in{\bf \cV_{f,{\it J,cnc}}^m}(\cH)$. Let  $\varphi({\bf W}_{i,j})$  be in  $J\subset F^\infty({\bf D_f^m})$  with    Fourier  representation
$$\varphi({\bf W}_{i,j})=\sum\limits_{(\beta) \in \FF_{n_1}^+\times \cdots \times \FF_{n_k}^+} c_{(\beta)}
{\bf W}_{(\beta)}.$$
Then $\varphi({\bf W}_{i,j})
     =\text{\rm SOT-}\lim_{r\to 1} \varphi(r{\bf W}_{i,j}),
     $
 and  $\varphi(r{\bf W}_{i,j})$ is in the   polydomain algebra $\cA({\bf D_f^m})$.
 Relation \eqref{K-eq} implies
 \begin{equation*}
 \left<  {\bf K_{f,T}}x,  \varphi(r{\bf W}_{i,j}){\bf W}_{(\alpha)}
(1)\otimes y\right> =\left<x, \varphi(rT_{i,j})T_{(\alpha)}
{\bf \Delta_f^m}(I)^{1/2} y\right>
\end{equation*}
for any $r\in [0,1)$,  $x\in \cH$,   $y\in \overline{{\bf \Delta_{f,T}^m}(I)\cH}$, and $(\alpha)  \in \FF_{n_1}^+\times \cdots \times \FF_{n_k}^+$.
Due to the $F^\infty({\bf D_f^m})$--functional calculus, we have
$0=\varphi({T}_{i,j})
     =\text{\rm SOT-}\lim_{r\to 1} \varphi(r{T}_{i,j})$. Consequently.
$\left<  {\bf K_{f,T}}x,  \varphi({\bf W}_{i,j}){\bf W}_{(\alpha)}
(1)\otimes y\right>  =0$ for any $\varphi({\bf W}_{i,j})\in J$,
 $y \in \overline{{\bf \Delta_{f,T}^m}(I)\cH}$, and $(\alpha) \in \FF_{n_1}^+\times \cdots \times \FF_{n_k}^+$. Therefore, relation  \eqref{range} holds also in this case. It is clear that  due to relation \eqref{range}, we have ${\bf K_{f,T,{\it J}}^*}{\bf K_{f,T, {\it J }}}={\bf K_{f,T}^*}{\bf K_{f,T}}$.
 Now, one can easily complete the proof using the appropriate properties of the noncommutative Berezin kernel ${\bf K_{f,T}}$ and the definition of the constrained Berezin kernel.
\end{proof}

For each  $n$-tuple ${\bf T}:=\{T_{i,j}\}\in {\bf \cV_{f,{\it J}}^m}(\cH)$, we introduce the {\it constrained
noncommutative Berezin transform}  at ${\bf T}$ as the map ${\bf
B_{T, {\it J}}}:B(\cN_J) \to B(\cH)$ defined by setting
 \begin{equation*}
 {\bf B_{T,{\it J}}}[g]:= {\bf K_{f,T,{\it J}}^*} (g\otimes I_\cH){\bf K_{f,T,{\it J}}},
 \quad g\in B( \cN_J),
 \end{equation*}
where $J$ is a left ideal in    $\CC[Z_{i,j}]$, $\CC[{\bf W}_{i,j}]$,  $\cA({\bf D_f^m})$, or  $F^\infty({\bf D_f^m})$. Note that ${\bf B_{T,{\it J}}}$ is a completely contractive, completely positive,  and $w^*$-continuous linear map. Consequently,  ${\bf B_{T,{\it J}}}$ is WOT-continuous (resp. SOT-continuous) on bounded sets. Note that ${\bf T}$ is pure if and only if  ${\bf B_{T,{\it J}}}(I)=I$.

\begin{theorem}\label{vN1-variety} Let ${\bf T}=(T_1,\ldots, T_k)\in B(\cH)^{n_1}\times \cdots \times B(\cH)^{n_k}$ and let $J$ be a $w^*$-closed left ideal of $F^\infty({\bf D_f^m})$. Then ${\bf T}$ is a pure element of the noncommutative variety
$\cV_{{\bf f}, J}^{\bf m}(\cH)$ if and only if there is a unital completely positive and $w^*$-continuous linear map
$$
{\bf \Psi}: \overline{\text{\rm span}}^{w^*} \{{\bf S}_{(\alpha)} {\bf S}_{(\beta)}^*: \
 (\alpha), (\beta) \in \FF_{n_1}^+\times \cdots \times \FF_{n_k}^+\}\to B(\cH)
 $$
 such that
 $$
 {\bf \Psi}({\bf S}_{(\alpha)} {\bf S}_{(\beta)}^*)={\bf T}_{(\alpha)} {\bf T}_{(\beta)}^*,\qquad (\alpha), (\beta) \in \FF_{n_1}^+\times \cdots \times \FF_{n_k}^+.
 $$
\end{theorem}
\begin{proof} Due to Proposition \ref{Ber-const},
 if ${\bf T}:=(T_1,\ldots, T_k)$ is a pure
tuple  in the noncommutative variety
${\bf \cV_{f,{\it J}}^m}(\cH)$,
   then   ${\bf K_{f,T, {\it J}}}$ is an isometry and the constrained
noncommutative Berezin transform
 is
    a unital completely contractive and $w^*$-continuos linear map
such that
   $$
{\bf B_{T,{\it J}}}[{\bf S}_{(\alpha)} {\bf S}_{(\beta)}^*]={\bf K_{f,T,{\it J}}^*}[{\bf S}_{(\alpha)} {\bf S}_{(\beta)}^*\otimes
I_\cH]{\bf K_{f,T, {\it J}}} =T_\alpha T_\beta^*
$$
for any $(\alpha), (\beta) \in \FF_{n_1}^+\times \cdots \times \FF_{n_k}^+$.
To prove the converse, assume that ${\bf \Psi}$ has the required properties. Since $({\bf S}_1,\ldots, {\bf S}_k)$ is a commuting tuple and ${\bf \Psi}$ is a homomorphism when restricted to $\CC[{\bf S}_{i,j}]$, we deduce that  $(T_1,\ldots, T_k)$ is a commuting tuple. Taking into account that  $\Phi_{f_i, {\bf S}_i}$ is a $w^*$-continuous map, and
${\bf \Delta}_{{\bf f},{\bf S}}^{\bf p}$ is a linear combination  of products of the form
 $\Phi_{f_1, {\bf S}_1}^{q_1}\cdots \Phi_{f_k, {\bf S}_k}^{q_k}$, where $(q_1,\ldots, q_k)\in \ZZ_+^k$, we deduce that ${\bf \Delta}_{{\bf f},{\bf S}}^{\bf p}$ is a
 $w^*$-continuous map. Since $\Psi$ is  a completely positive $w^*$-continuous linear map such that
 $
 {\bf \Psi}({\bf S}_{(\alpha)} {\bf S}_{(\beta)}^*)={\bf T}_{(\alpha)} {\bf T}_{(\beta)}^*$
 for any $ (\alpha), (\beta) \in \FF_{n_1}^+\times \cdots \times \FF_{n_k}^+,
 $
we obtain
$$
{\bf \Delta}_{{\bf f},{\bf S}}^{\bf p}(I)={\bf \Psi}({\bf \Delta}_{{\bf f},{\bf S}}^{\bf p}(I))\geq 0
$$
 for any ${\bf p}=(p_1,\ldots, p_k)\in \ZZ_+^k$ with ${\bf p}\leq {\bf m}$. Therefore,
 ${\bf T}\in {\bf D_f^m}(\cH)$. On the other hand, for each $i\in \{1,\ldots, k\}$, we have
 $$
 \lim_{q_i\to\infty} \Phi_{f_i, T_i}^{q_i}(I)={\bf \Psi}(\lim_{q_i\to\infty}  \Phi_{f_i, {\bf S}_i}^{q_i}(I))={\bf \Psi}(0)=0,
 $$
 which shows that ${\bf T}$ is a pure tuple in  the polydomain ${\bf D_f^m}(\cH)$.
 To prove that ${\bf T}$ is in   the noncommutative variety
$\cV_{{\bf f}, J}^{\bf m}(\cH)$, fix  $g\in J$ and recall that $g({\bf W}_{i,j})
     =\text{\rm SOT-}\lim_{r\to 1} g(r{\bf W}_{i,j}),
     $
     where
   $g(r{\bf W}_{i,j})$ is in the   polydomain algebra $\cA({\bf D_f^m})$, and $\|g(r{\bf W}_{i,j})\|\leq \|g({\bf W}_{i,j})\|$ for any $r\in [0,1)$.
 Using the the $F^\infty({\bf D_f^m})$--functional calculus for pure elements in ${\bf D_f^m}(\cH)$ and the fact that  WOT and $w^*$-topology coincide on bounded sets, we deduce that
 \begin{equation*}
 \begin{split}
 g({T}_{i,j})&=\text{\rm WOT-}\lim_{r\to 1} g(r{T}_{i,j})
 =\text{\rm WOT-}\lim_{r\to 1}{\bf \Psi}( g(r{\bf S}_{i,j}))\\
 &={\bf \Psi}(\text{\rm WOT-}\lim_{r\to 1} g(r{\bf S}_{i,j}))={\bf \Psi}(g({\bf S}_{i,j}))={\bf \Psi}(0)=0.
 \end{split}
 \end{equation*}
 Therefore, ${\bf T}$ is in   the noncommutative variety
$\cV_{{\bf f}, J}^{\bf m}(\cH)$. The proof is complete.
\end{proof}

\begin{theorem}\label{C*-charact2} Let  $\cQ\subset \CC[Z_{i,j}]$
  be a left ideal generated by  noncommutative
    homogenous polynomials  and let
$${\bf T}:=(T_1,\ldots, T_k)\in   B(\cH)^{n_1}\times\cdots \times B(\cH)^{n_k}.$$
 Then  ${\bf T}$ is in the
noncommutative variety ${\bf \cV_{q,\cQ}^m}(\cH)$, where  ${\bf q}=(q_1,\ldots, q_k)$ is a $k$-tuple of   positive regular  noncommutative polynomials, if and only if  there is a unital completely positive linear map
 $
 {\bf \Psi}: \overline{\cS}\to B(\cH)$, where
 $\overline{\cS}:= \overline{\text{\rm span}} \{{\bf S}_{(\alpha)} {\bf S}_{(\beta)}^*: \
 (\alpha), (\beta) \in \FF_{n_1}^+\times \cdots \times \FF_{n_k}^+\}$,
  such that
   $$
{\bf \Psi}({\bf S}_{(\alpha)} {\bf S}_{(\beta)}^*)=  {\bf T}_{(\alpha)} {\bf T}_{(\beta)}^*,\qquad
(\alpha), (\beta) \in \FF_{n_1}^+\times \cdots \times \FF_{n_k}^+, $$
where ${\bf S}:=\{{\bf S}_{i,j}\}$  is the  universal model associated with the abstract  noncommutative variety ${\bf \cV_{q,\cQ}^m}$ .
\end{theorem}
\begin{proof} Assume that ${\bf T}\in {\bf \cV_{q,\cQ}^m}(\cH)$. Since ${\bf D_q^m}(\cH)$ is a radial domain \cite{Po-Berezin-poly}, $r{\bf T}\in {\bf D_q^m}(\cH)$ for any $r\in [0,1)$. Note that, due to the fact that  $\cQ\subset \CC[Z_{i,j}]$
  is a left ideal generated by  noncommutative
    homogenous polynomials, if $g\in \cQ$, then $g({\bf T}_{i,j})=0$ and  $g(r{\bf T}_{i,j})=0$. Thus $r{\bf T}\in {\bf \cV_{q,\cQ}^m}(\cH)$ and,
   as in the proof of Theorem \ref{vN1-variety}, one  can show that
 $\text{\rm range}\, {\bf K_{q,{\it r}T}}\subseteq \cN_\cQ\otimes \cH$
for any $r\in [0,1)$, where ${\bf K_{q,{\it r}T}}$ is the Berezin
kernel associated with $r{\bf T}\in{\bf D^m_q}(\cH)$. Moreover,
$$
{\bf K_{q,{\it r}T,\cQ}} (r^{|\alpha|+|\beta|} {\bf T}_{(\alpha)}{\bf T}_{(\beta)}^*)=({\bf S}_{(\alpha)}{\bf S}_{(\beta)}^*\otimes
I_\cH){\bf K_{q,{\it r}T,\cQ}},\qquad (\alpha), (\beta)  \in \FF_{n_1}^+\times \cdots \times \FF_{n_k}^+.
$$
  Since $r{\bf T}$ is
 pure, ${\bf K_{q,{\it r}T,\cQ}}$ is an isometry.
Consequently,  for any $n\times n$ matrix with entries $\psi_{st}({\bf S}_{i,j})$ in the linear span  $\cS$ of all products ${\bf S}_{(\alpha)} {\bf S}_{(\beta)}^*$,  where $(\alpha), (\beta) \in \FF_{n_1}^+\times \cdots \times \FF_{n_k}^+$, we have the von Neumann type inequality
$$
\|[\psi_{st}(r{\bf T}_{i,j})]_{n\times n}\|\leq \|[\psi_{st}({\bf S}_{i,j})]_{n\times n}\|, \qquad r\in [0,1).
$$
Taking $r\to 1$, we deduce that $
\|[\psi_{st}({\bf T}_{i,j})]_{n\times n}\|\leq \|[\psi_{st}({\bf S}_{i,j})]_{n\times n}\|.
$
  We define
the unital
completely contractive linear map ${\bf \Psi_{f,q,\cQ}}:\cS\to B(\cH) $
 by setting ${\bf \Psi_{q,T,\cQ}}({\bf S}_{(\alpha)} {\bf S}_{(\beta)}^*):= {\bf T}_{(\alpha)} {\bf T}_{(\beta)}^*$, for all $ (\alpha), (\beta)$  in $\FF_{n_1}^+\times \cdots \times \FF_{n_k}^+$. Now, it is clear that ${\bf \Psi}$ has a unique extension to a unital
completely contractive linear map on $\overline{\cS}$.

  To prove the converse, assume that ${\bf \Psi}$ has the required properties and   note that, due to Lemma \ref{univ-mod-var} and the fact that $1\in \cN_\cQ$, we have
$$ (I-\Phi_{q_1,{T}_1})^{p_1}\cdots (I-\Phi_{q_k,{T}_k})^{p_k}(I)={\bf \Psi}\left[ (I-\Phi_{q_1,{\bf S}_1})^{p_1}\cdots (I-\Phi_{q_k,{\bf S}_k})^{p_k}(I_{\cN_\cQ})\right]\geq 0
$$
for any $p_i\in \{0,1,\ldots, m_i\}$ and
       $i\in \{1,\ldots,k\}$. Since $({\bf S}_1,\ldots, {\bf S}_k)$ is a commuting tuple and $\Psi$ is a homomorphism when restricted to $\CC[{\bf S}_{i,j}]$, we deduce that  $(T_1,\ldots, T_k)$ is a commuting tuple. Therefore, ${\bf T}\in {\bf D_f^m}(\cH)$. On the other hand, since $g({\bf S}_{i,j})=0$ for any $g\in \cQ$, we have
        $g({\bf T}_{i,j})=\Psi(g({\bf S}_{i,j}))=0$, which shows that  ${\bf T}\in {\bf \cV_{q,\cQ}^m}(\cH)$.
         The proof is complete.
\end{proof}

 \begin{proposition}\label{vN2-variety}
 Let  $\cQ\subset \CC[Z_{i,j}]$
  be a left ideal generated by  noncommutative
    homogenous polynomials, and let ${\bf T}:=(T_1,\ldots, T_n)$  be
  in the noncommutative variety
${\bf \cV_{f,\cQ}^m}(\cH)$, where ${\bf f}=(f_1,\ldots, f_k)$ is a $k$-tuple of  positive regular free holomorphic functions.
   Then there is
    a unital completely contractive linear map
$ {\bf \Psi_{f,T,\cQ}}:  \overline{\cS}\to B(\cH),
 $
 where $\overline{\cS}:= \overline{\text{\rm span}} \{{\bf S}_{(\alpha)} {\bf S}_{(\beta)}^*: \
 (\alpha), (\beta) \in \FF_{n_1}^+\times \cdots \times \FF_{n_k}^+\}$,
  such that
\begin{equation*}
{\bf \Psi_{f,T,\cQ}}(g)=\lim_{r\to 1} {\bf B}_{r{\bf T}, \cQ}[g],\qquad g\in \overline{\cS},
\end{equation*}
where the limit exists in the norm topology of $B(\cH)$, and
$$
{\bf \Psi_{f,T,\cQ}}({\bf S}_{(\alpha)} {\bf S}_{(\beta)}^*)={\bf T}_{(\alpha)} {\bf T}_{(\beta)}^*, \qquad (\alpha), (\beta) \in \FF_{n_1}^+\times \cdots \times \FF_{n_k}^+.
$$
In particular, the restriction of ${\bf \Psi_{f,T,\cQ}}$ to the variety algebra $\cA({\bf \cV_{f,\cQ}^m})$ is a unital completely contractive homomorphism.
 If, in addition,
 ${\bf T}$ is a  pure $k$-tuple of operators,
then  $$\lim_{r\to 1} {\bf B}_{r{\bf T},\cQ}[g]= {\bf
B}_{\bf T,\cQ}[g], \qquad g\in \overline{\cS},
$$
 where the limit exists in the norm
topology of $B(\cH)$.
\end{proposition}
\begin{proof}   Following  the proof   of the direct implication of
 Theorem \ref{C*-charact2}, we can show that
the linear map ${\bf \Psi_{f,T,\cQ}}:\cS\to B(\cH) $
  defined by ${\bf \Psi_{f,T,\cQ}}({\bf S}_{(\alpha)} {\bf S}_{(\beta)}^*):= {\bf T}_{(\alpha)} {\bf T}_{(\beta)}^*$, for all $ (\alpha), (\beta) \in \FF_{n_1}^+\times \cdots \times \FF_{n_k}^+$, is  unital and
completely contractive.  Given $g=g({\bf S}_{i,j})\in \overline{\cS}$, we define ${\bf \Psi_{f,T,\cQ}}(g):=\lim_{n\to\infty} {\bf \Psi_{f,T,\cQ}}(g_n)$, where $g_n\in \cS$ with $\|g-g_n\|\to 0$, as $n\to\infty$. Note that ${\bf \Psi_{f,T,\cQ}}(g)$ does not depend on the choice of the sequence $\{g_n\}$ and
\begin{equation*}
\begin{split}
\|{\bf \Psi_{f,T,\cQ}}(g)-{\bf B}_{r{\bf T}, \cQ}[g]\|&\leq
\|{\bf \Psi_{f,T,\cQ}}(g)-{\bf \Psi_{f,T,\cQ}}(g_n)\|
 +\|{\bf \Psi_{f,T,\cQ}}(g_n)-{\bf B}_{r{\bf T}, \cQ}[g_n]\| +\|{\bf B}_{r{\bf T}, \cQ}[g_n-g]\|\\
 &\leq 2\|g-g_n\| +\|{\bf \Psi_{f,T,\cQ}}(g_n)-{\bf B}_{r{\bf T}, \cQ}[g_n]\|.
\end{split}
\end{equation*}
Hence, we deduce that ${\bf \Psi_{f,T,\cQ}}(g)=\lim_{r\to 1} {\bf B}_{r{\bf T}, \cQ}[g]$ for any $ g\in \overline{\cS}$. Now, we assume that  ${\bf T}$ is a pure $k$-tuple in  ${\bf \cV_{f,\cQ}^m}(\cH)$.
Since
$$
{\bf B_{T,\cQ}}[g_n]:= {\bf K^*_{f,T,\cQ}} (g_n\otimes I_\cH) {\bf K_{f,T,\cQ}}=g_n({ T}_{i,j})
$$
and taking into account that   $g_n\in \cS$ with $\|g-g_n\|\to 0$, as $n\to\infty$,  we conclude that
$
{\bf B_{T,\cQ}}[g]=\Psi_{\bf f,T,\cQ}(g)$ for any  $g\in \cS$.
This
completes the proof.
\end{proof}

\bigskip

\section{Universal operator  models and joint invariant subspaces}

In this section,  we  obtain  a characterization of the Beurling \cite{Be} type  joint invariant subspaces under the universal model ${\bf S}=\{{\bf S}_{i,j}\}$ of $\cV_{\bf f, \it J}^{\bf m}$, and a characterization of  the joint reducing subspaces of ${\bf S}\otimes I$. We use noncommutative Berezin transforms to characterize the pure elements in  noncommutative varieties  $\cV_{\bf f, \it J}^{\bf m}$ and obtain a classification result for the  pure elements of rank one.

  Denote by  $C^*({\bf S}_{i,j})$ the $C^*$-algebra generated by  the operators ${\bf S}_{i,j}$,  where $i\in \{1,\ldots, k\}$,  $j\in \{1,\ldots, n_i\}$, and the identity.

\begin{theorem}\label{compact} Let ${\bf q}=(q_1,\ldots, q_k)$ be a $k$-tuple of   positive regular  noncommutative polynomials
 and let ${\bf S}=({\bf S}_1,\ldots, {\bf S}_k)$ be the universal model
associated with the abstract noncommutative variety   $\cV_{\bf q, {\it J}}^{\bf m}$, where $J$ is a WOT-closed two sided ideal of $F^\infty({\bf D_q^m})$ such that $1\in \cN_J$.
  Then all
the compact operators in $B(\cN_J)$ are contained in the operator
space
$$\overline{\cS}:=\overline{\text{\rm  span}} \{ {\bf S}_{(\alpha)} {\bf S}_{(\beta)}^*:\
(\alpha), (\beta) \in \FF_{n_1}^+\times \cdots \times \FF_{n_k}^+\}.
$$
Moreover,   the $C^*$-algebra  $C^*({\bf S}_{i,j})$ is
irreducible.
\end{theorem}

\begin{proof}
   Since $1\in \cN_J$,   Lemma \ref{univ-mod-var} implies
\begin{equation}\label{proj2}
  (I-\Phi_{q_1,{\bf S}_1})^{m_1}\cdots (I-\Phi_{q_k,{\bf S}_k})^{m_k}(I_{\cN_J})=P_{\cN_J}{\bf P}_\CC|_{\cN_J}={\bf P}_\CC^{\cN_J},
\end{equation}
where $ P^{\cN_J}_\CC$ is the orthogonal projection of $\cN_J$
onto $\CC$.
Fix a polynomial
$g({\bf W}_{i,j}):= \sum\limits_{{(\beta) \in \FF_{n_1}^+\times \cdots \times \FF_{n_k}^+}\atop{|\beta_1|+\cdots +|\beta_k|\leq n}} d_{(\beta)}
{\bf W}_{(\beta)} $
and  let $\xi:=\sum\limits_{(\beta) \in \FF_{n_1}^+\times \cdots \times \FF_{n_k}^+} c_{(\beta)}e_{(\beta)}$ be in $\cN_J\subset \otimes_{i=1}^k F^2(H_{n_i})$, where  we denote $e_{(\beta)}:=e^1_{\beta_1}\otimes\cdots \otimes  e^k_{\beta_k}$, if $(\beta):=(\beta_1,\ldots, \beta_k)$.
It is easy to see   that
$
{\bf P}_\CC^{\cN_J} g({\bf S}_{i,j})^*\xi= \left< \xi,g({\bf S}_{i,j})(1)\right>.
$
Consequently, we have
\begin{equation}\label{rankone}
\chi({\bf S}_{i,j}){\bf P}_\CC^{\cN_J} g({\bf S}_{i,j})^*\xi= \left<
\xi,g({\bf S}_{i,j})(1)\right>\chi({\bf S}_{i,j})(1)
\end{equation}
for any polynomial $\chi({\bf S}_{i,j})$.  Employing relation \eqref{proj2}, we deduce that the
operator $\chi({\bf S}_{i,j}){\bf P}_\CC^{\cN_J} g({\bf S}_{i,j})^*$
 has rank one  and  it is   in the
operator space $\overline{\cS}$. On the other hand, due to the fact that
the set of all vectors  of the form
 $  \sum\limits_{{(\beta) \in \FF_{n_1}^+\times \cdots \times \FF_{n_k}^+}\atop{|\beta_1|+\cdots +|\beta_k|\leq n}} d_{(\beta)}
{\bf S}_{(\beta)}(1)$ with  $n\in \NN$, $d_{(\beta)}\in \CC$, is  dense in
$\cN_J$, relation \eqref{rankone} implies that   all
the compact operators in $B(\cN_J)$ are contained in  $\cS$.

  To prove the last part of this theorem, let $\cE\neq\{0\}$ be a
subspace of $\cN_J\subset \otimes_{i=1}^k F^2(H_{n_i})$, which is jointly reducing
for the operators ${\bf S}_{i,j}$,    $i\in \{1,\ldots,k\}$ and $j\in\{1,\ldots,n_i\}$. Let $\varphi\in \cE$,
$\varphi\neq 0$, and assume that
$
\varphi=\sum\limits_{(\beta) \in \FF_{n_1}^+\times \cdots \times \FF_{n_k}^+} a_{(\beta)}e_{(\beta)}.
$
 If
 $a_{(\beta)}$  is a nonzero coefficient of $\varphi$,
then we have
 $$
{\bf P}_\CC {\bf S}_{1,\beta_1}^*\cdots {\bf S}_{k,\beta_k}^*\varphi= {\bf P}_\CC {\bf W}_{1,\beta_1}^*\cdots {\bf W}_{k,\beta_k}^*\varphi = \frac{1}{\sqrt{b_{1,\beta_1}^{(m_1)}}}\cdots \frac{1}{\sqrt{b_{k,\beta_k}^{(m_k)}}}  a_{(\beta)}.
$$
 Due to relation \eqref{proj2} and  using the fact that
 $\cE$ is reducing for each  ${\bf S}_{i,j}$,
we deduce that $a_{(\beta)}\in \cE$, so $1\in \cE$. Using  again
that $\cE$ is invariant under the operators ${\bf S}_{i,j}$, we
deduce that  $\cE= \cN_J$.
This completes the proof.
\end{proof}

 Let ${\bf T}=({ T}_1,\ldots, { T}_k)\in {\bf D_f^m}(\cH)$  and  ${\bf T}'=({ T}_1',\ldots, { T}_k')\in {\bf D_f^m}(\cH')$ be $k$-tuples with $T_i:=(T_{i,1},\ldots, T_{i,n_i})$ and $T_i':=(T_{i,1}',\ldots, T_{i,n_i}')$. We say that  ${\bf T}$ is unitarily equivalent to ${\bf T}'$ if there is a unitary operator $U:\cH\to \cH'$ such that
$
T_{i,j}=U^* T_{i,j}'U
$
for all $i\in\{1,\ldots,k\}$ and $j\in\{1,\ldots, n_i\}$.

\begin{corollary}\label{eq-mult}
Let ${\bf S}=\{{\bf S}_{i,j}\} $ be the universal model
associated with the abstract noncommutative variety   $\cV_{\bf q, {\it J}}^{\bf m}$, where $J$ is a WOT-closed left ideal of $F^\infty({\bf D_q^m})$ such that $1\in \cN_J$. If $\cH$, $\cK$ are Hilbert spaces, then $\{{\bf S}_{i,j}\otimes I_\cH\}$ is unitarily equivalent to $\{{\bf S}_{i,j}\otimes I_\cK\}$ if and only if $\dim \cH=\dim \cK$.
\end{corollary}
\begin{proof}
  Let   $U:\cN_J\otimes \cH\to
\cN_J\otimes \cK$ be a unitary operator such that
$
U({\bf S}_{i,j}\otimes I_\cH)=({\bf S}_{i,j}\otimes I_{\cK})U
$
for all $i\in \{1,\ldots, k\}$ and $j\in \{1,\ldots, n_i\}$.
 Then
$
U({\bf S}_{i,j}^*\otimes I_\cH)=({\bf S}_{i,j}^*\otimes I_{\cK})U
$
and, due to the fact that  the $C^*$-algebra $C^*({\bf S}_{i,j})$ is irreducible, we
must have $U=I_{\cN_J}\otimes A$, where $A\in B(\cH,\cK)$ is a
unitary operator. Therefore, $\dim \cH=\dim \cK$. The proof is
complete.
 \end{proof}

We recall that a   subspace $\cH\subseteq \cK$ is
called co-invariant under $\Lambda\subset B(\cK)$ if $X^*\cH\subseteq
\cH$ for any $X\in \Lambda$.

\begin{theorem}\label{cyclic} Let ${\bf S}=\{{\bf S}_{i,j}\}$   be the universal model
associated with the abstract noncommutative variety $\cV_{\bf f, {\it J}}^{\bf m}$, where $J$ is a WOT-closed two sided ideal of $F^\infty({\bf D_f^m})$ such that $1\in \cN_J$.
  If $\cK$  be a Hilbert space and
$\cM\subseteq \cN_J\otimes \cK$ is a co-invariant subspace under each operator
${\bf S}_{i,j}\otimes I_\cK$ for  $i\in \{1,\ldots, k\}$, $j\in \{1,\ldots, n_i\}$, then  there exists a subspace
$\cE\subseteq \cK$ such that
\begin{equation*}
\overline{\text{\rm span}}\,\left\{\left({\bf S}_{(\beta)}\otimes
I_\cK\right)\cM:\ (\beta) \in \FF_{n_1}^+\times \cdots \times \FF_{n_k}^+\right\}=\cN_J\otimes \cE.
\end{equation*}
\end{theorem}

\begin{proof}
   Set
 $\cE:=({\bf P}_\CC\otimes I_\cK)\cM\subseteq \cK$, where ${\bf P}_\CC$ is the
 orthogonal projection from $\cN_J$ onto $\CC 1\subset \cN_J$ and let $\varphi$ be a
nonzero element of $\cM$ with   representation
$$\varphi=\sum\limits_{(\beta) \in \FF_{n_1}^+\times \cdots \times \FF_{n_k}^+} e_{(\beta)} \otimes h_{(\beta)} \in \cM\subset \cN_J,
$$
where $h_{(\beta)}\in \cK$ and $\sum\limits_{(\beta) \in \FF_{n_1}^+\times \cdots \times \FF_{n_k}^+}\|h_{(\beta)}\|^2<\infty$.
  Assume that $h_{(\sigma)}\neq 0$ for some $\sigma=(\sigma_1,\ldots, \sigma_k)$ in $\FF_{n_1}^+\times \cdots \times \FF_{n_k}^+$  and note that
\begin{equation*} \begin{split}
({\bf P}_\CC \otimes I_\cK) ({\bf S}_{1,\sigma_1}^*\cdots {\bf S}_{k,\sigma_k}^*\otimes I_\cK)\varphi&= ({\bf P}_\CC \otimes I_\cK) ({\bf W}_{1,\sigma_1}^*\cdots {\bf W}_{k,\sigma_k}^*\otimes I_\cK)\varphi\\
&= 1\otimes\frac{1}{\sqrt{b_{1,\sigma_1}^{(m_1)}}}\cdots \frac{1}{\sqrt{b_{k,\sigma_k}^{(m_k)}}}  h_{(\sigma)}.
\end{split}
\end{equation*}
 Consequently, since $\cM$ is  a
co-invariant subspace under each operator ${\bf S}_{i,j}\otimes I_\cK$,
we must have  $ h_{(\sigma)}\in \cE$.  Since $1\in \cN_J$, we deduce that
$$
 ({\bf S}_{1,\sigma_1}\cdots {\bf S}_{k,\sigma_k}\otimes I_\cK)(1\otimes
h_{(\sigma)})=\frac{1}{\sqrt{b_{1,\sigma_1}^{(m_1)}}}\cdots \frac{1}{\sqrt{b_{k,\sigma_k}^{(m_k)}}}P_{\cN_J}(e^1_{\sigma_1}\otimes\cdots \otimes  e^k_{\sigma_k})\otimes h_{(\sigma)}
$$
is a vector in $\cN_J\otimes \cE$.
 Therefore,
\begin{equation}\label{phi-series}
\varphi= \lim_{n\to\infty}\sum_{q=0}^n \sum\limits_{{(\beta) \in \FF_{n_1}^+\times \cdots \times \FF_{n_k}^+}\atop{|\beta_1|+\cdots +|\beta_k|=q}} P_{\cN_J}(e^1_{\beta_1}\otimes\cdots \otimes  e^k_{\beta_k})\otimes h_{(\beta) }
\end{equation}
is in $\cN_J\otimes \cE$. Hence,
 $\cM\subset\cN_J\otimes \cE$ and
$$
\cY:= \overline{\text{\rm span}}\,\left\{({\bf S}_{(\sigma)} \otimes I_\cK)\cM:\ (\sigma) \in \FF_{n_1}^+\times \cdots \times \FF_{n_k}^+\right\}\subset \cN_J\otimes \cE.
$$

Now, we prove  the reverse inclusion.
 If
$h_0\in \cE$, $h_0\neq 0$, then there exists $\xi\in \cM\subset
\cN_J\otimes \cE$ such that $$\xi=1\otimes
h_0+\sum\limits_{{(\beta) \in \FF_{n_1}^+\times \cdots \times \FF_{n_k}^+}\atop{|\beta_1|+\cdots +|\beta_k|\geq 1}} e^1_{\beta_1}\otimes\cdots \otimes  e^k_{\beta_k}\otimes h_{(\beta)}
$$
and $1\otimes h_0=({\bf P}_\CC\otimes I_\cK) \xi$.
Consequently, due to Lemma \ref{univ-mod-var}, we have
\begin{equation*}
\begin{split}
1\otimes h_0=({\bf P}_\CC\otimes I_\cK) \xi
=(id-\Phi_{f_1,{\bf S}_1\otimes I_\cK})^{m_1}\cdots (id-\Phi_{f_k,{\bf S}_k\otimes I_\cK})^{m_k}(I_{\cN_J}\otimes I_\cK)\xi.
\end{split}
\end{equation*}
 Taking into account that $\cM$ is co-invariant under each operator ${\bf S}_{i,j}\otimes I_\cK$,
we deduce that $h_0\in\cY$ for any $h_0\in \cE$. Therefore,  $\cE\subset
\cY$. This   inclusion  shows that $({\bf S}_{(\sigma)} \otimes I_\cK)(1\otimes \cE)\subset \cY$ for any $(\sigma) \in \FF_{n_1}^+\times \cdots \times \FF_{n_k}^+$, which
implies
\begin{equation*}\label{PN}
\frac{1}{\sqrt{b_{1,\sigma_1}^{(m_1)}}}\cdots \frac{1}{\sqrt{b_{k,\sigma_k}^{(m_k)}}}P_{\cN_J}(e^1_{\sigma_1}\otimes\cdots \otimes  e^k_{\sigma_k})\otimes \cE\subset
\cY.
\end{equation*}
Consequently, if $\varphi\in \cN_J\otimes \cE$
has   the  representation \eqref{phi-series},   we conclude  that
$\varphi \in \cY.$
Therefore, $ \cN_J\otimes \cE\subseteq \cY$.  The proof is complete.
\end{proof}

Now, we can easily deduce the following result.

\begin{corollary}  Let   ${\bf S}:=({\bf S}_1,\ldots, {\bf S}_k)$ be the universal model associated to the abstract  noncommutative variety $\cV_{{\bf f}, J}^{\bf m}$, where $J$ is   a WOT-closed two sided ideal of $F^\infty({\bf D_f^m})$ such that $1\in \cN_J$. If $\cK$  is a Hilbert space, then    a  subspace
$\cM\subseteq \cN_J\otimes \cK$ is reducing under each operator
${\bf S}_{i,j}\otimes I_\cK$ for  $i\in \{1,\ldots, k\}$, $j\in \{1,\ldots, n_i\}$,
if and only if   there exists a subspace $\cE\subseteq \cK$ such
that
\begin{equation*}
 \cM=\cN_J\otimes \cE.
\end{equation*}
\end{corollary}

Let ${\bf S}:=\{{\bf S}_{i,j}\}$ be the universal model associated to the abstract  noncommutative variety ${\bf {\cV}_{f,{\it J}}^m}$.  An operator $M:\cN_J\otimes \cH\to
\cN_J\otimes \cK$  is called {\it  multi-analytic}  with respect to ${\bf S}$
 if $$M({\bf S}_{i,j}\otimes I_\cH)=({\bf S}_{i,j}\otimes I_\cK)M$$
for any $i\in \{1,\ldots,k\}$ and  $j\in \{1,\ldots, n_i\}$. In case $M$ is a partial isometry, we call
it {\it inner multi-analytic} operator.

The next result is an extension of Theorem 5.2 from \cite{Po-Berezin-poly} to varieties in noncommutative polydomains. The constructions from the proof are needed in  a forthcoming section  to define  characteristic functions associated with  noncommutative varieties.

\begin{theorem}\label{Beur-fact} Let ${\bf S}:=({\bf S}_1,\ldots, {\bf S}_k)$ be the universal model associated to the abstract  noncommutative variety $\cV_{{\bf f}, J}^{\bf m}$ and  let ${\bf S}_i\otimes I_\cH:=({\bf S}_{i,1}\otimes I_\cH,\ldots, {\bf S}_{i,n_i}\otimes I_\cH)$ for  $i\in \{1,\ldots,k\}$, where $\cH$ is a Hilbert space.
If $G\in B(\cN_J\otimes \cH)$ then the following statements  are equivalent.
\begin{enumerate}
\item[(i)]
There is  a multi-analytic operator
 $\Gamma:\cN_J\otimes \cE \to
 \cN_J\otimes \cH$ with respect to ${\bf S}$,
 where  $\cE$ is a Hilbert space, such that
$$G=\Gamma \Gamma^*.$$
\item[(ii)] For any ${\bf p}:=(p_1,\ldots, p_k)\in \ZZ_+^k$
 such that ${\bf p}\leq {\bf m}$, ${\bf p}\neq 0$,
     $$
   ({\bf \Delta}_{{\bf f,S}\otimes I_\cH}^{\bf p}
(G)\geq 0.
   $$
\end{enumerate}
\end{theorem}
\begin{proof}  Assume that item (i) holds. Then we have
$$
{\bf \Delta}_{{\bf f,S}\otimes I_\cH}^{\bf p}
(G)=(id-\Phi_{f_1,{\bf S}_1\otimes I_\cH})^{p_1}\cdots
 (id-\Phi_{f_k,{\bf S}_k\otimes I_\cH})^{p_k}(G)=\Gamma {\bf \Delta}_{{\bf f,S}\otimes I_\cE}^{\bf p}(I)\Gamma^*\geq 0
$$
for any ${\bf p}:=(p_1,\ldots, p_k)\in \ZZ_+^k$ such that ${\bf p}\leq {\bf m}$,
 ${\bf p}\neq 0$.

Now,  assume that (ii) holds. In particular, we have
$\Phi_{f_1,{\bf S}_1\otimes I_\cH}({\bf \Delta}_{{\bf f,S}\otimes I_\cH}^{\bf m'}(G))
\leq {\bf \Delta}_{{\bf f,S}\otimes I_\cH}^{\bf m'}(G)$, where ${\bf m}'
=(m_1-1,m_2,\ldots,m_k)$, which implies
 $\Phi_{f_1,{\bf S}_1\otimes I_\cH}^n({\bf \Delta}_{{\bf f,S}\otimes I_\cH}^{\bf m'}(G))
 \leq {\bf \Delta}_{{\bf f,S}\otimes I_\cH}^{\bf m'}(G)$ for any $n\in \NN$. Since
$ {\bf S}:=({\bf S}_1,\ldots, {\bf S}_k)$ is a pure $k$-tuple, we have
SOT-$\lim_{n\to\infty}\Phi_{f_1,{\bf S}_1
\otimes I_\cH}^n({\bf \Delta}_{{\bf f,S}\otimes I_\cH}^{\bf m'}(G))=0$. Consequently,
 ${\bf \Delta}_{{\bf f,S}\otimes I_\cH}^{\bf m'}(G)\geq 0$.
Continuing this process, we deduce that $G\geq 0$.

 Let $\cG:=\overline{\text{\rm range}\, G^{1/2}}$ and define
\begin{equation}\label{ai}
A_{i,j}(G^{1/2} x):=G^{1/2} ({\bf S}_{i,j}^*\otimes I_\cH)x,\qquad x\in
\cN_J\otimes \cH,
\end{equation}
 for any $i\in \{1,\ldots, k\}$ and $j\in \{1,\ldots, n_i\}$. Taking into account that  $\Phi_{f_i,{\bf S}_i\otimes I}(G)\leq G$, we have
\begin{equation*}
\begin{split}
\sum_{\alpha\in \FF_{n_i}^+, |\alpha|\geq 1}
 a_{i,\alpha}\|A_{i,{\tilde \alpha}}G^{1/2}x\|^2
=\left< \Phi_{f_i,{\bf S}_i\otimes I_\cH}(G)x,x\right>\leq \|G^{1/2} x\|^2
\end{split}
\end{equation*}
for any $x\in \cN_J\otimes \cH$, where $\widetilde \alpha=g_{j_p}^i\cdots g_{j_1}^i$ denotes the reverse of
$\alpha=g_{j_1}^i\cdots g_{j_p}^i\in \FF_{n_i}^+$.  Consequently,
$a_{i,g_j^i}\|A_{i,j}G^{1/2} x\|^2\leq \|G^{1/2} x\|^2$, for any $x\in
\cN_J\otimes \cH$.  Since $a_{i,g_j^i}\neq 0$ each  $A_{i,j}$ can be uniquely
be extended to a bounded operator (also denoted by $A_{i,j}$) on the
subspace $\cG$. Set $X_{i,j}:=A_{i,j}^*$ for  $i\in \{1,\ldots, k\}$, $j\in\{1,\ldots, n_i\}$.    An
approximation argument shows that $\Phi_{f_i,X_i}(I_\cG)\leq I_\cG$
and relation \eqref{ai} implies
\begin{equation}
\label{int**}
X_{i,j}^*(G^{1/2} x)=G^{1/2} ({\bf S}_{i,j}^*\otimes I_\cH)x,\qquad x\in
\cN_J\otimes \cH.
\end{equation}
   This implies
 $
 G^{1/2} {\bf \Delta_{f,X}^p}(I_\cM) G^{1/2}
 ={\bf \Delta}_{{\bf f,S}\otimes I_\cH}^{\bf p}(G)\geq 0
 $
for any ${\bf p}:=(p_1,\ldots, p_k)\in \ZZ_+^k$ such that ${\bf p}\leq {\bf m}$, ${\bf p}\neq 0$.
 Now, note that
\begin{equation*}
\begin{split}
\left< \Phi_{f_i, X_i}^n(I_\cG) Y^{1/2}x, G^{1/2} x\right> &= \left<
\Phi_{f_i,
{\bf S}_i\otimes I_\cH}^n(G)x,  x\right>
\leq \|G\| \left< \Phi_{f_i, {\bf S}_i\otimes I_\cH}^n(I)x,  x\right>
\end{split}
\end{equation*}
for any $ x\in
\cN_J\otimes \cH$ and $n\in \NN$. Since \
SOT-$\lim\limits_{n\to\infty}\Phi_{f_i,{\bf S}_i\otimes I_\cH}^n(I)=0$,
we have \
SOT-$\lim\limits_{m\to\infty}\Phi_{f_i,X_i}^n(I_\cG)=0$. Therefore,
${\bf X}:=(X_1,\ldots, X_k)$ is a pure $k$-tuple in the noncommutative variety
${\bf D_f^m}(\cM)$.  Due to the $F^\infty({\bf D_f^m})$--functional calculus, relation \eqref{int**} implies
$$
G^{1/2} g(X_{i,j})=g({\bf S}_{i,j})G^{1/2}=0, \qquad g\in J.
$$
  Consequently,
$g(X_{i,j})=0$ for any $g\in J$. This shows that ${\bf X}:=(X_1,\ldots, X_k)$ is a pure $k$-tuple in the noncommutative variety $\cV_{{\bf f},J}^{\bf m}(\cG)$.
According to Proposition \ref{Ber-const}, the noncommutative Berezin
kernel ${\bf K_{f,X, {\it J}}}:\cG\to \cN_J\otimes \cE$
 is an isometry with the property that
 $
  X_{i,j}{\bf K_{f,X,{\it J}}^*}={\bf K_{f,X,{\it J}}^*} ({\bf S}_{i,j}\otimes I_\cE)$.
 Set $\cE:=\overline{{\bf \Delta_{f,X}^m}(I_\cG)(\cG)}$ and define the bounded linear  operator  $\Gamma:=G^{1/2} {\bf K_{f,X,{\it J}}^*}: \cN_J\otimes \cE\to \cN_J\otimes
\cH$. Note that
\begin{equation*}
\begin{split}
\Gamma({\bf S}_{i,j}\otimes I_\cE)&=G^{1/2}{\bf K_{f,X,{\it J}}^*}
({\bf S}_{i,j}\otimes I_\cE)=G^{1/2}
X_{i,j} {\bf K_{f,X, {\it J}}^*}\\
&=({\bf S}_{i,j}\otimes I_\cH) G^{1/2} {\bf K_{f,X, {\it J}}^*}
 =({\bf S}_{i,j}\otimes I_\cH) \Gamma
\end{split}
\end{equation*}
for any $i\in\{1,\ldots, k\}$ and $j\in \{1,\ldots, n_i\}$, which proves that $\Gamma$ is a multi-analytic operator with respect to  the universal model ${\bf S}=\{{\bf S}_{i,j}\}$.  Note that $\Gamma \Gamma^*=G^{1/2}
{\bf K_{f,X, {\it J}}^*} {\bf  K_{f,X,{\it J}}} G^{1/2} =G$. The proof is complete.
\end{proof}

Following the classical case \cite{Be}, we say that $\cM\subset \cN_J\otimes \cH$ is
a Beurling type invariant subspace under the operators
 ${\bf S}_{i,j}\otimes I_\cH$ for  $i\in \{1,\ldots, k\}$, $j\in \{1,\ldots, n_i\}$,
  if there is an inner multi-analytic operator  with respect to ${\bf S}=\{{\bf S}_{i,j}\}$,
  $$\Psi:\cN_J\otimes
   \cE \to \cN_J\otimes \cH,$$
   such that
   $\cM=\Psi\left(\cN_J\otimes \cE\right)$.

\begin{corollary}\label{Beurling}
Let  $\cM\subset \cN_J\otimes \cH$
 be an invariant subspace under the operators
  ${\bf S}_{i,j}\otimes I_\cH$ for any $i\in \{1,\ldots, k\}$,
  $j\in \{1,\ldots, n_i\}$. Then $\cM$ is  Beurling type  invariant subspace if and only if
  $$
  {\bf \Delta_{f,S\otimes{\rm I_\cH}}^p}(P_\cM)\geq 0,\qquad \text{ for any }\ {\bf p} \in \ZZ_+^k,  {\bf p}\leq {\bf m},
   $$
  where $P_\cM$ is the orthogonal projection of the Hilbert space
    $ \cN_J\otimes \cH$ onto $\cM$.
 \end{corollary}
\begin{proof}
If
$M:\cN_J\otimes
\cE \to \cN_J\otimes \cH$ is a
 inner multi-analytic operator
 and
 $\cM=M\left(\cN_J\otimes \cE\right)$,
then $P_\cM=M M^*$. Taking into account
Lemma \ref{univ-mod-var}, we deduce that
$$
   {\bf \Delta_{f,S\otimes{\rm I_\cH}}^p}(P_\cM)
    =\Psi (P_{\bf C}\otimes I_\cE) \Psi^*\geq 0
   $$
   for any ${\bf p}:=(p_1,\ldots, p_k)\in \ZZ_+^k$ such that ${\bf p}\leq {\bf m}$.
   The converse is a consequence of Theorem \ref{Beur-fact}, when we take $G=P_\cM$.
 The proof is complete.
 \end{proof}

We remark that in the particular case when ${\bf m}=(1,\ldots,1)$, the condition in Corollary \ref{Beurling} is satisfied when ${\bf S}\otimes I_\cH|_\cM:=\{{\bf S}_{i,j}\otimes I_\cH|_\cM\}$ is  doubly commuting. The proof is very similar to that of  the corresponding result from \cite{Po-Berezin-poly}.

\begin{theorem}\label{dil3} Let ${\bf S}=\{{\bf S}_{i,j}\}$   be the universal model
associated with the abstract noncommutative variety $\cV_{\bf f, {\it J}}^{\bf m}$, where $J$ is a WOT-closed left ideal of $F^\infty({\bf D_f^m})$,
  and let ${\bf T}=\{{T}_{i,j}\}$ be  a
 pure element in the noncommutative variety $\cV_{{\bf f}, J}^{\bf m}(\cH)$. If
  $${\bf K_{f,T,{\it J}}}: \cH \to \cN_J \otimes  \overline{{\bf \Delta_{f,T}^m}(I)(\cH)}$$
  is the noncommutative constrained  Berezin kernel,
 then the  subspace ${\bf K_{f,T,{\it J}}}\cH$ is   co-invariant  under  each operator
${\bf S}_{i,j}\otimes  I_{\overline{{\bf \Delta_{f,T}^m}\cH}}$ for   any $i\in \{1,\ldots, k\}$,  $  j\in \{1,\ldots, n_i\}$. If $1\in \cN_J$, then
    the dilation
  provided by   the relation
    $$ { \bf T}_{(\alpha)}= {\bf K_{f,T,{\it J}}^*}({\bf S}_{(\alpha)}\otimes  I_{\overline{{\bf\Delta_{f,T}^m}(I)(\cH)}})  {\bf K_{f,T,{\it J}}}, \qquad (\alpha) \in \FF_{n_1}^+\times \cdots \times \FF_{n_k}^+,
    $$
  is minimal.
If, in addition,  \ ${\bf f}= {\bf q}=(q_1,\ldots, q_k)$ is a $k$-tuple of   positive regular  noncommutative polynomials and
\begin{equation*}
\overline{\text{\rm span}}\,\{{\bf S}_{(\alpha)}  {\bf S}_{(\beta)}^*) :\
 (\alpha), (\beta)\in \FF_{n_1}^+\times \cdots \times \FF_{n_k}^+\}=C^*({\bf S}_{i,j}),
\end{equation*}
then  the   minimal dilation  of ${\bf T}$
 is unique up to an isomorphism.
\end{theorem}
 \begin{proof} According to Proposition \ref{Ber-const},
  $${\bf K_{f,T,{\it J}}} { T}^*_{i,j}= ({\bf S}_{i,j}^*\otimes I)  {\bf K_{f,T, {\it J}}},\qquad i\in \{1,\ldots, k\}, j\in \{1,\ldots, n_i\},
   $$
   and the noncommutative Berezin kernel ${\bf K_{f,T, {\it J}}}$  is an isometry.
Due to  the
definition of the constrained Berezin kernel ${\bf K_{f,T, {\it J}}}$, we obtain
$
 ({\bf P}_\CC \otimes I_\cD) K_{\bf f,T, {\it J}}\cH= \cD,
$
where $\cD:=\overline{{\bf\Delta_{f,T}^m}(I)(\cH)}$.
   Now, using
Theorem \ref{cyclic} in the particular case when $\cM:=K_{\bf f,T, {\it J}}\cH$
and $\cE:=\cD$, we deduce that the subspace
${\bf K_{f,T, {\it J}}}\cH$ is cyclic for the operators  ${\bf S}_{i,j}\otimes I_\cE$, where $i\in \{1,\ldots, k\}$ and $j\in \{1,\ldots, n_i\}$. This implies
  the minimality of the  dilation, i.e.,
\begin{equation}\label{minimal1}
\cN_J\otimes \cD=\bigvee_{(\alpha)\in \FF_{n_1}^+\times \cdots \times \FF_{n_k}^+}
({\bf S}_{(\alpha)}\otimes I_{\cD}) {\bf K_{f,T,{\it J}}}\cH.
\end{equation}
Now, assume that
${\bf f}= {\bf q}=(q_1,\ldots, q_k)$ is a $k$-tuple of   positive
 regular  noncommutative polynomials  and that  the relation in the theorem  holds. Consider
another minimal   dilation of ${\bf T}$, i.e.,
\begin{equation}
\label{another} { \bf T}_{(\alpha)}=V^* ({ \bf S}_{(\alpha)}\otimes I_{\cD'})V, \qquad (\alpha) \in \FF_{n_1}^+\times \cdots \times \FF_{n_k}^+,
\end{equation}
where $V:\cH\to \cN_J\otimes \cD'$ is an isometry, $V\cH$ is
co-invariant under each operator ${\bf S}_{i,j}\otimes I_{\cD'}$, and
\begin{equation}\label{minimal2}
\cN_J\otimes \cD'=\bigvee_{(\alpha)\in \FF_{n_1}^+\times \cdots \times \FF_{n_k}^+} ({\bf S}_{(\alpha)}\otimes
I_{\cD'}) V\cH.
\end{equation}
According  to Theorem \ref{vN1-variety}, there exists a unique unital completely positive
linear map
${\bf \Psi}: C^*({\bf S}_{i,j})\to B(\cH)$
 with the property that
$$
{\bf \Psi}\left( {\bf S}_{(\alpha)} {\bf S}_{(\beta)}^*\right)= {\bf T}_{(\alpha)} {\bf T}_{(\beta)}^*, \qquad (\alpha), (\beta)\in \FF_{n_1}^+\times \cdots \times \FF_{n_k}^+.
$$
   Now, we consider the $*$-representations
$
\pi_1:C^*({\bf S}_{i,j})\to B(\cN_J\otimes \cD)$, $ \pi_1(X)
:= X\otimes I_{\cD},
$
  and
  $
\pi_2:C^*({\bf S}_{i,j})\to B(\cN_J)\otimes \cD')$, $ \pi_2(X):=
X\otimes I_{\cD'}.
$
Since the   subspaces ${\bf K_{q,T,{\it J}}}\cH$  and $V\cH$  are     co-invariant  for each operator
${\bf S}_{i,j}\otimes  I_{\cD}$,
 relation    \eqref{another}  implies
$$
{\bf \Psi}(X)={\bf K_{q,T, {\it J}}^*}\pi_1(X){\bf K_{q,T, {\it J}}}
=V^*\pi_2(X)V,\qquad
\  X\in C^*({\bf S}_{i,j}).
$$
Relations \eqref{minimal1} and
\eqref{minimal2} show  that $\pi_1$
and $\pi_2$ are  minimal Stinespring dilations of the completely
 positive linear map
 ${\bf \Psi}$.  Since
these representations are unique up to an isomorphism, there exists a unitary operator $U:\cN_J\otimes
\cD\to \cN_J\otimes \cD'$ such that
 $
U({\bf S}_{i,j}\otimes I_{\cD})=({\bf S}_{i,j}\otimes
I_{\cD'})U
$
for $i\in \{1,\ldots, k\}, j\in \{1,\ldots, n_i\}$,
and $U{\bf K_{q,T, {\it J}}}=V$.  Taking into account that $U$ is unitary, we deduce that
$
U({\bf S}_{i,j}^*\otimes I_{\cD})=({\bf S}_{i,j}^*\otimes
I_{\cD'})U.
$
Since the $C^*$-algebra $C^*({\bf S}_{i,j})$ is irreducible, due to Theorem \ref{compact},  we must have $U=I\otimes W$, where
$W\in B(\cD,\cD')$ is a unitary operator.
This implies that $\dim \cD=\dim\cD'$ and
$U{\bf K_{q,T, {\it J}}}\cH=V\cH$. Consequently, the two dilations are
unitarily equivalent.
 The proof is complete.
\end{proof}

\begin{proposition}\label{rank-n} Let ${\bf S}=\{{\bf S}_{i,j}\} $   be the universal model
associated with the abstract noncommutative variety $\cV_{\bf q, {\it J}}^{\bf m}$, where $J$ is a WOT-closed left ideal of $F^\infty({\bf D_q^m})$ such that $1\in \cN_J$,
  and   ${\bf q}=(q_1,\ldots, q_k)$ is  a $k$-tuple of   positive regular  noncommutative polynomials such that
  \begin{equation*}
\overline{\text{\rm span}}\,\{{\bf S}_{(\alpha)}  {\bf S}_{(\beta)}^*) :\
 (\alpha), (\beta)\in \FF_{n_1}^+\times \cdots \times \FF_{n_k}^+\}=C^*({\bf S}_{i,j}).
\end{equation*}
A
 pure element  ${\bf T}=\{{T}_{i,j}\}\in {\bf {\cV}_q^m}(\cH)$
    has
    $$\rank {\bf \Delta_{q,T}^m}(I)=n,\qquad  n=1,2,\ldots, \infty,
     $$
     if and only if it is
unitarily equivalent to one obtained by compressing $\{{\bf S}_{i,j}\otimes
I_{\CC^n}\}$ to a co-invariant subspace
$\cM\subset  \cN_J\otimes \CC^n$  under each operator
 ${\bf S}_{i,j}\otimes
I_{\CC^n}$
 with the property that $\dim [({\bf P}_\CC\otimes I_{\CC^n})\cM]=n$,
where ${\bf P}_\CC$ is the orthogonal projection from $\cN_J $
onto  $\CC 1$.
\end{proposition}
\begin{proof}
  Note that the direct implication  is a consequence
of Theorem \ref{dil3}. We prove the converse.  Assume that
$$
{ \bf T}_{(\alpha)}=P_\cH ({ \bf S}_{(\alpha)}\otimes I_{\CC^n})|_\cH,\qquad  (\alpha) \in \FF_{n_1}^+\times \cdots \times \FF_{n_k}^+
$$
where $\cH\subset \cN_J \otimes \CC^n$ is a co-invariant subspace
under each operator
${\bf S}_{i,j}\otimes  I_{\CC^n}$  such
that $\dim ({\bf P}_\CC\otimes I_{\CC^n})\cH=n$.  It is clear that ${\bf T}$  is  a pure element in the noncommutative variety ${\bf {\cV}_q^m}(\cH)$. First, we consider   the case when
$n<\infty$. Since $({\bf P}_\CC\otimes I_{\CC^n})\cH\subseteq \CC^n$ and $\dim ({\bf P}_\CC\otimes I_{\CC^n})\cH=n$, we
must have $({\bf P}_\CC\otimes I_{\CC^n})\cH=\CC^n$. The later
condition is equivalent  to the equality
 $
  \cH^\perp\cap \CC^n=\{0\}.
$
Since
 ${\bf \Delta_{q,S}^m}(I)={\bf P}_\CC$,
  we have
 $
 {\bf \Delta_{q,T}^m}(I)= P_\cH \left[  {\bf P}_\CC \otimes I_{\CC^n}
\right]|_\cH= P_\cH \CC^n.
 $
Consequently,    $\rank  {\bf \Delta_{q,T}^m}(I)=\dim P_\cH \CC^n$.
 If we assume that $\rank{\bf \Delta_{q,T}^m}(I)<n$, then there exists  $h\in
\CC^n$, $h\neq 0$, with $P_\cH h=0$, which contradicts the relation
$\cH^\perp\cap \CC^n=\{0\}$. Therefore, we must have $\rank {\bf \Delta_{q,T}^m}(I)=n$.

Now, assume that  $n=\infty$. According to Theorem
\ref{cyclic} and its proof,   we have
$$
\cN_J\otimes \cE=\bigvee_{(\alpha)\in \FF_{n_1}^+\times \cdots \times \FF_{n_k}^+} ({\bf S}_{(\alpha)}\otimes
I_{\CC^n}) \cH
$$
where $\cE:=({\bf P}_\CC\otimes I_{\CC^n})\cH$.  Since $\cN_J\otimes \cE$  is reducing for each operator
${\bf S}_{i,j}\otimes I_{\CC^m}$,   we deduce that
$
{ \bf T}_{(\alpha)}=P_\cH ({ \bf S}_{(\alpha)}\otimes I_\cE)|_\cH,$ for all
$ (\alpha) \in \FF_{n_1}^+\times \cdots \times \FF_{n_k}^+.
$
Due to  Theorem \ref{dil3},  the  minimal   dilation of
${\bf T}$ is unique. Consequently,  we have
$\dim \overline{ {\bf \Delta_{q,T}^m(I)}\cH}=\dim\cE=\infty.
$
This completes the proof.
\end{proof}

In what follows, we
characterize   the pure  elements of  rank one in the noncommutative variety $\cV_{\bf q, {\it J}}^{\bf m} (\cH)$  and obtain a classification result.

\begin{corollary}\label{rank1} Under the hypothesis  of Proposition \ref{rank-n}, the following statements hold.
 \begin{enumerate}
 \item[(i)]
If $\cM\subset \cN_J$ is  a co-invariant  subspace under each operator  ${\bf S}_{i,j}$,  then
$
{\bf T}:= \{P_\cM {\bf S}_{i,j}|_\cM\}
$
is a pure  element   in the  noncommutative variety $\cV_{\bf q, {\it J}}^{\bf m}(\cM)$  and
$\rank {\bf \Delta_{q,T}^m}=1$.
\item[(ii)]
If $\cM'$ is another co-invariant subspace under   each operator ${\bf S}_{i,j}$, which gives rise to    ${\bf T}'$, then ${\bf T}$
and ${\bf T}'$ are unitarily equivalent if and only if $\cM=\cM'$.
\end{enumerate}
\end{corollary}
\begin{proof} To prove (i), note that
 ${\bf \Delta_{q,T}^m}(I)= P_\cM  {\bf P}_\CC|_\cM$
and, consequently,  $\rank  {\bf \Delta_{q,T}^m}(I) \leq 1$.
Since  ${\bf S}$ is pure (see Lemma \ref{univ-mod-var}) and $\cM\subset \cN_J$ is  a co-invariant  subspace under each operator  ${\bf S}_{i,j}$, we deduce that ${\bf T}$ is pure. Hence,  ${\bf \Delta_{q,T}^m}(I)\neq 0$, so
 $\rank {\bf \Delta_{q,T}^m}(I)\geq 1$. Therefore,  $\rank {\bf \Delta_{q,T}^m}(I)=1$.

To prove  (ii), note that,  as in the proof of  Theorem
\ref{dil3}, one can show that ${\bf T}$ and ${\bf T}'$ are unitarily
equivalent if and only if there exists a unitary operator
$\Lambda:\cN_J\to \cN_J$ such that
$
\Lambda {\bf S}_{i,j}={\bf S}_{i,j} \Lambda$ for all $i,j$,
    and $\Lambda \cM=\cM'$.
Since  $\Lambda {\bf S}_{i,j}^*={\bf S}_{i,j}^* \Lambda$ and
$C^*({\bf S}_{i,j})$ is irreducible,
$\Lambda$ must be a scalar multiple of the identity. Therefore, we must have
$\cM=\Lambda \cM=\cM'$. The proof is complete.
\end{proof}

\bigskip

\section{Noncommutative varieties and multivariable  function theory}
 \label{Symmetric}

In  this section,  we find all the joint eigenvectors  for ${\bf S}_{i,j}^*$, where ${\bf S}=\{{\bf S}_{i,j}\}$  is the  universal model
associated with the noncommutative variety ${\bf {\cV}^m_{f,{\it J}}}$  and $J$ is a WOT-closed left ideal of the Hardy space $F^\infty({\bf D_f^m})$. As
consequences, we
  determine the joint right spectrum of  ${\bf S}$ and identify  the character space of the noncommutative variety algebra $\cA({\bf {\cV}^m_{f,{\it J}}})$. When $J_c$ is the commutator ideal of $F^\infty({\bf D_f^m})$, we show that the WOT-closed algebra  $F^\infty({\bf {\cV}^m_{f,{\it J_c}}})$ generated by ${\bf S}_{i,j}$ and the identity coincides with the multiplier algebra of a reproducing kernel Hilbert space of holomorphic functions on  a certain  polydomain in $\CC^n$.
The results of this section show that there is a strong connection between the study of noncommutative varieties in polydomains and  the  analytic function theory  in $\CC^n$.

 Let  ${\bf f}:=(f_1,\ldots,f_k)$ be a $k$-tuple of positive regular free holomorphic functions and define the   polydomain
$${\bf D_{f,>}^m}(\CC):=\left\{ {\bf z}=(z_1,\ldots z_k)\in
\CC^{n_1}\times \cdots \times \CC^{n_k}:\ {\bf \Delta_{f,z}^m}(1)>0 \right\}.
$$
Note that ${\bf D_{f,>}^m}(\CC)={\bf D}_{f_1, >}^1(\CC)\times \cdots \times {\bf D}_{f_k, >}^1(\CC)$, where $f_i:= \sum_{\alpha\in
\FF_{n_i}^+} a_{i,\alpha} Z_\alpha$ and
$$
{\bf D}_{f_i,>}^1(\CC):=\{z_i=(z_{i,1},\ldots, z_{i,n_i})\in \CC^{n_i}: \ \sum_{\alpha\in \FF_{n_i}^+} a_{i,\alpha}|z_{i,\alpha}|^2<1\}.
$$
Let  $J$ be a WOT-closed left ideal of the Hardy space $F^\infty({\bf D_f^m})$. We consider the set
$$
{\bf {\cV}^m_{f,{\it J}, >}}(\CC):=\{{\bf z}=(z_1,\ldots z_k)\in {\bf D_{f,>}^m}(\CC): \ g(z_1,\ldots, z_k)=0 \ \text{ for } \ g\in J\}\subset \CC^n,
$$
where $n=n_1+\cdots +n_k$ is the number of indeterminates in ${\bf f}:=(f_1,\ldots,f_k)$.

\begin{theorem}\label{eigenvectors} Let   ${\bf S}=\{{\bf S}_{i,j}\}$  be the  universal model
associated with the noncommutative variety ${\bf {\cV}^m_{f,{\it J}}}$, where $J$ is a WOT-closed left ideal of the Hardy space $F^\infty({\bf D_f^m})$.
 The joint eigenvectors for
${\bf S}_{i,j}^*$  are
precisely the noncommutative constrained  Berezin kernels
$$
\Gamma_\lambda:={\bf \Delta_{f,\lambda}^m}(1)^{1/2} \sum_{\beta_i\in \FF_{n_i}^+, i=1,\ldots,k}
   \overline{\lambda}_{1,\beta_1}\cdots \overline{\lambda}_{k,\beta_k}\sqrt{b_{1,\beta_1}^{(m_1)}}\cdots \sqrt{b_{k,\beta_k}^{(m_k)}}
   e^1_{\beta_1}\otimes \cdots \otimes  e^k_{\beta_k}
$$
for $\mathbf{\lambda}=(\lambda_1,\ldots, \lambda_n)\in {\bf {\cV}^m_{f,{\it J}, >}}(\CC)$, where ${\bf \Delta_{f,\lambda}^m}(1):=(1-\Phi_{f_1,\lambda_1}(1))^{m_1}\cdots (1-\Phi_{f_k,\lambda_k}(1))^{m_k}$.
 They satisfy the equations
$$
{\bf S}_{i,j}^*\Gamma_\lambda=\overline{\lambda}_{i,j} \Gamma_\lambda
 \qquad \text{
for } \ i\in\{1,\ldots,k\}, j\in \{1,\ldots, n_i\},
$$
 where $\lambda_i=(\lambda_{i,1},\ldots, \lambda_{i,n_i})$.
\end{theorem}

\begin{proof}
First, note that if $\lambda=(\lambda_1,\ldots, \lambda_n)\in {\bf {\cV}^m_{f,{\it J}, >}}(\CC)$, then $\lambda$ is  a pure element.
   The noncommutative constrained Berezin kernel  at $\lambda$
is ${\bf K_{f,\lambda, {\it J}}}: \CC \to \cN_J\otimes \CC   $
   defined by
$$
{\bf K_{f,\lambda, {\it J}}} (w)={\bf \Delta_{f,\lambda}^m}(1)^{1/2} \sum_{\beta_i\in \FF_{n_i}^+, i=1,\ldots,k}
   \sqrt{b_{1,\beta_1}^{(m_1)}}\cdots \sqrt{b_{k,\beta_k}^{(m_k)}}
   e^1_{\beta_1}\otimes \cdots \otimes  e^k_{\beta_k}\otimes \overline{\lambda}_{1,\beta_1}\cdots \overline{\lambda}_{k,\beta_k}w,\qquad w\in \CC.
   $$
According to Proposition \ref{Ber-const}, we have $({\bf S}_{i,j}^*\otimes I_\CC)
{\bf K_{f,\lambda, {\it J}}}={\bf K_{f,\lambda, {\it J}}}(\overline{\lambda}_{i,j} I_\CC)$ for
$i\in\{1,\ldots,k\}, j\in \{1,\ldots, n_i\}$. Identifying $\cN_J\otimes \CC$ with $\cN_J$, we have ${\bf K_{f,\lambda, {\it J}}}=\Gamma_\lambda$ and  ${\bf S}_{i,j}^*\Gamma_\lambda=\overline{\lambda}_{i,j} \Gamma_\lambda$.

Conversely, let $h=\sum\limits_{\beta_1\in \FF_{n_1}^+,\ldots, \beta_k\in \FF_{n_k}^+} c_{\beta_1,\ldots, \beta_k}
e_{\gamma_1}^1\otimes\cdots \otimes e_{\gamma_k}^k$ be  a non-zero vector in $\cN_J\subset \otimes_{i=1}^kF^2(H_{n_i})$ and assume that  there exists $(\lambda_1,\ldots, \lambda_n)\in
\CC^{n_1}\times \cdots \times \CC^{n_k}$, where $\lambda_i=(\lambda_{i,1},\ldots, \lambda_{i,n_i})$, such that
${\bf S}_{i,j}^*h=\overline{\lambda}_{i,j} h$
  for any $i\in\{1,\ldots,k\}, j\in \{1,\ldots, n_i\}$.  Since $\cN_J$ is invariant under ${\bf W}_{i,j}^*$, we also have ${\bf W}_{i,j}^*h=\overline{\lambda}_{i,j} h$.  Using the definition of the
operators ${\bf W}_{i,j}$ (see Section 1), we deduce that
\begin{equation*}
\begin{split}
c_{\beta_1,\ldots, \beta_k}&=\left< h,e^1_{\beta_1}\otimes \cdots \otimes  e^k_{\beta_k}\right>=\left< h, \sqrt{b_{1,\beta_1}^{(m_1)}}\cdots \sqrt{b_{k,\beta_k}^{(m_k)}}
{\bf W}_{1,\beta_1}\cdots {\bf W}_{k,\beta_k}(1)\right>\\
&= \sqrt{b_{1,\beta_1}^{(m_1)}}\cdots \sqrt{b_{k,\beta_k}^{(m_k)}}\left< {\bf W}_{1,\beta_1}^*\cdots {\bf W}_{k,\beta_k}^* h,1\right>=
\sqrt{b_{1,\beta_1}^{(m_1)}}\cdots \sqrt{b_{k,\beta_k}^{(m_k)}}
     \overline{\lambda}_{1,\beta_1}\cdots \overline{\lambda}_{k,\beta_k}\left< h,1\right>\\
&=c_0\sqrt{b_{1,\beta_1}^{(m_1)}}\cdots \sqrt{b_{k,\beta_k}^{(m_k)}}
     \overline{\lambda}_{1,\beta_1}\cdots \overline{\lambda}_{k,\beta_k}
\end{split}
\end{equation*}
 for any $\beta_1\in \FF_{n_1}^+,\ldots, \beta_k\in \FF_{n_k}^+$. Hence, we  obtain
  $$
  h= c_0 \sum_{\beta_i\in \FF_{n_i}^+, i=1,\ldots,k}
   \overline{\lambda}_{1,\beta_1}\cdots \overline{\lambda}_{k,\beta_k}\sqrt{b_{1,\beta_1}^{(m_1)}}\cdots \sqrt{b_{k,\beta_k}^{(m_k)}}
   e^1_{\beta_1}\otimes \cdots \otimes  e^k_{\beta_k}.
   $$
 Since $h\in \otimes_{i=1}^kF^2(H_{n_i})$, we must have $\sum\limits_{\beta_1\in \FF_{n_1}^+,\ldots, \beta_k\in \FF_{n_k}^+}
    |\lambda_{1,\beta_1}|^2\cdots|\lambda_{k,\beta_k}|^2 {b_{1,\beta_1}^{(m_1)}\cdots
  b_{k,\beta_k}^{(m_k)}}<\infty.$ On the other hand, relation \eqref{b-coeff}
implies
$$
\prod_{i=1}^k\left(\sum_{s=0}^{p_i} \left(\sum_{|\alpha_i|\geq 1} a_{i,\alpha_i}
|\lambda_{i,\alpha_i}|^2\right)^s\right)^{m_i}\leq \sum\limits_{\beta_1\in \FF_{n_1}^+,\ldots, \beta_k\in \FF_{n_k}^+}
    |\lambda_{1,\beta_1}|^2\cdots|\lambda_{k,\beta_k}|^2 {b_{1,\beta_1}^{(m_1)}\cdots
  b_{k,\beta_k}^{(m_k)}}<\infty
$$
for any $p_1,\ldots, p_k\in \NN$. Letting $p_i\to\infty$ in the relation above, we
must have $\sum_{|\alpha_i|\geq 1} a_{i,\alpha_i}
|\lambda_{i,\alpha_i}|^2<1$,
 for each $i\in \{1,\ldots,k\}$. Therefore, $\lambda=(\lambda_1,\ldots, \lambda_n)\in {\bf D_{f,>}^m}(\CC)$. On the other hand,  if $g\in J$, then relation ${\bf S}_{i,j}^*h=\overline{\lambda}_{i,j} h$  and an approximation argument in the norm topology imply
 $$
 \left<h, g(r{\bf S}_{i,j})h\right>=\left<g(r{\bf S}_{i,j})^*h, h\right>=\overline{g(r\lambda_{i,j})} \|h\|^2.
 $$
Using the $F^\infty({\bf D_f^m})$-functional calculus  for pure elements and taking the limit as $r\to 1$ in the relation above, we obtain
$$
 \left<h, g({\bf S}_{i,j})h\right>= \overline{g(\lambda_{i,j})} \|h\|^2.
 $$
Since, due to Lemma \ref{univ-mod-var}, $g({\bf S}_{i,j})=0$ and $h\neq 0$, we conclude that $g(\lambda_{i,j})=0$, which shows that $\lambda \in {\bf {\cV}^m_{f,{\it J}, >}}(\CC)$.
The proof is complete.
\end{proof}

Let   ${\bf S}=\{{\bf S}_{i,j}\}$  be the  universal model
associated with the noncommutative variety ${\bf {\cV}^m_{f,{\it J}}}$, where $J$ is a WOT-closed left ideal of the Hardy algebra $F^\infty({\bf D_f^m})$. We introduce the Hardy algebra   $F^\infty({\bf {\cV}^m_{f,{\it J}}})$  as  the WOT-closed algebra generated by ${\bf S}_{i,j}$ and the identity.

\begin{theorem} \label{w*} Let $J$ be  a WOT-closed left ideal of the Hardy algebra $F^\infty({\bf D_f^m})$ such that $1\in \cN_J$.  Then
 $\Phi:F^\infty({\bf {\cV}^m_{f,{\it J}}})\to \CC$ is a $w^*$-continuous and multiplicative linear functional  if and only if there exists $\lambda\in {\bf {\cV}^m_{f,{\it J},>}}(\CC)$ such that
$$
\Phi(A)=\left<A(1), u_\lambda\right>,\qquad A\in
 F^\infty({\bf {\cV}^m_{f,{\it J}}}),
 $$
where $ u_\lambda:=\frac{1}{{\bf \Delta_{f,\lambda}^m}(1)^{1/2} }\Gamma_\lambda
$
and  $\Gamma_\lambda$ is given by Theorem \ref{eigenvectors}. Moreover, in this case, $A^*u_\lambda=\overline{\Phi(A)} u_\lambda$ and
$$
\Phi(A)=\left<A\Gamma_\lambda, \Gamma_\lambda\right>,\qquad A\in
 F^\infty({\bf {\cV}^m_{f,{\it J}}}).
 $$
\end{theorem}
\begin{proof}
For each $\lambda\in {\bf {\cV}^m_{f,{\it J},>}}(\CC)$, let
$\Phi_\lambda:F^\infty({\bf {\cV}^m_{f,{\it J}}})\to \CC$ be given by
$\Phi_\lambda(A)=\left<A(1), u_\lambda\right>$. It is clear that $\Phi_\lambda$ is $w^*$-continuous. To prove that $\Phi_\lambda$ is multiplicative, let $\varphi,\psi\in F^\infty({\bf {\cV}^m_{f,{\it J}}})$ and let
 $\{p_\iota({\bf S}_{i,j})\}$ and $\{q_\kappa({\bf S}_{i,j})\}$ be nets of polynomials such that $p_\iota({\bf S}_{i,j})\to \varphi$ and $q_\kappa({\bf S}_{i,j})\to \psi$ in the weak operator topology.
 Note that, due to Theorem \ref{eigenvectors}, we have
 $p_\iota(\lambda)=\left<p_\iota({\bf W}_{i,j})1,u_\lambda\right>=\left<p_\iota({\bf S}_{i,j})1,u_\lambda\right>
 $
and, consequently, $\lim_\iota p_\iota(\lambda)=\left<\varphi(1),u_\lambda\right>$. Similarly, we obtain
$\lim_\kappa q_\kappa(\lambda)=\left<\psi(1),u_\lambda\right>$.  Hence, it is easy to see that
\begin{equation*}
\begin{split}
\Phi_\lambda(\varphi \psi)&=\left<\psi\psi(1),u_\lambda\right>
=\lim_\kappa \left<q_\kappa (1), \varphi^*(u_\lambda)\right>\\
&=\lim_\kappa \lim_\iota\left<p_\iota({\bf S}_{i,j})q_\kappa({\bf S}_{i,j})(1),u_\lambda\right>
=\lim_\kappa \lim_\iota p_\iota(\lambda) q_\kappa(\lambda)\\
&= \left<\varphi(1),u_\lambda\right>\lim_\kappa q_\kappa(\lambda)
=\left<\varphi(1),u_\lambda\right>\left<\psi(1),u_\lambda\right>
=\Phi_\lambda(\varphi) \Phi_\lambda(\psi).
\end{split}
\end{equation*}
Note that, due to Theorem \ref{eigenvectors}, we have
\begin{equation*}
\begin{split}
 p_\iota({\bf S}_{i,j})^* u_\lambda =   \overline {p_\iota(\lambda )}u_\lambda =
 \overline{\left<p_\iota({\bf S}_{i,j})1,u_\lambda\right>}  u_\lambda.
\end{split}
\end{equation*}
Since $p_\iota({\bf S}_{i,j})\to \varphi$ in the weak operator topology,   we deduce that
$
 \varphi^* u_\lambda=\overline {
\left<\varphi(1),u_\lambda\right>} u_\lambda.
$
Hence, we deduce that
\begin{equation*}
\begin{split}
\left< \varphi \Gamma_\lambda, \Gamma_\lambda\right>&=
{\bf \Delta_{f,\lambda}^m}(1)\left< u_\lambda,\varphi ^*
u_\lambda\right>
 = \varphi(\lambda)=\Phi_\lambda(\varphi).
\end{split}
\end{equation*}
Now, assume that $\Phi:F^\infty({\bf {\cV}^m_{f,{\it J}}})\to \CC$ is a $w^*$-continuous and multiplicative linear functional and let $\cX:=\ker \Phi$.
Then $\cX$ is a $w^*$-closed two-sided ideal of $F^\infty({\bf {\cV}^m_{f,{\it J}}})$ of codimension one. We claim that   $\cM_\cX:=\overline{\cX \cN_J}$ is a subspace in $\cN_J$ of codimension one and $\cM_\cX+\CC 1=\cN_J$. By contradiction, assume that there is a vector $y\in \cN_J$ which is perpendicular to $\cM_\cX+\CC 1$ and $\|y\|=1$. Since
$$\overline{\text{\rm span}}\{p({\bf W}_{i,j})(1): \ p\in \CC[Z_{i,j}]\}=\otimes_{i=1}^k F^2(H_{n_i})
$$ and  taking the projection onto $\cN_J$, we deduce that
$\overline{\text{\rm span}}\{p({\bf S}_{i,j})(1): \ p\in \CC[Z_{i,j}]\}=\cN_J$.
Consequently, we can choose a polynomial  $p({\bf S}_{i,j})\in F^\infty({\bf {\cV}^m_{f,{\it J}}})$ such that $\|p({\bf S}_{i,j})(1)-y\|<1$. On the other hand, since $p({\bf S}_{i,j})-\Phi(p({\bf S}_{i,j})) I_{\cN_J}$ is in $\cX=\ker \Phi$ and $1\in \cN_J$, we have $p({\bf S}_{i,j})(1)-\Phi(p({\bf S}_{i,j}))\in \cM_\cX$. Taking into account that $y$ is perpendicular to $\cM_\cX+\CC 1$, we have
\begin{equation*}
\begin{split}
\|y\|&=\left<y-\Phi(p({\bf S}_{i,j})),y\right>\\
&\leq |\left<y-p({\bf S}_{i,j})(1),y\right>| + |\left<p({\bf S}_{i,j})(1)-\Phi(p({\bf S}_{i,j})),y\right>|\\
&=|\left<y-p({\bf S}_{i,j})(1),y\right>|\leq \|y-p({\bf S}_{i,j})(1)\| \|y\|<1,
\end{split}
\end{equation*}
which contradicts the fact that $\|y\|=1$ and proves our assertion.
Therefore, $\cM_\cX \subset \cN_J$ has codimension one and it is invariant under each operator ${\bf S}_{i,j}$ for $i\in\{1,\ldots,k\}, j\in \{1,\ldots, n_i\}$. According to Theorem \ref{eigenvectors}, there exists  $\lambda\in {\bf {\cV}^m_{f,{\it J},>}}(\CC)$ such that $\cM_\cX=\{u_\lambda\}^\perp$. As shown in the first part of the proof,  $\Phi_\lambda$ is  a $w^*$-continuous and multiplicative linear functional. Note that, if $A\in \cX:=\ker \Phi$, then $A(1)\in \cM_\cX=\{u_\lambda\}^\perp$, which implies  $\left<A(1), u_\lambda\right>=0$. Hence, $A\in \ker \Phi_\lambda$ and, therefore, $\ker \Phi\subset \ker \Phi_\lambda$. Since $\ker \Phi$ and $\ker \Phi_\lambda$ are $w^*$-closed two sided maximal ideals of $F^\infty({\bf {\cV}^m_{f,{\it J}}})$ of codimension one, we must have
$\ker \Phi=\ker \Phi_\lambda$. Therefore, $\Phi=\Phi_\lambda$. This completes the proof.
\end{proof}

 We make a few remarks concerning the particular case when $J=\{0\}$.
 First, we note that if $\lambda=(\lambda_1,\ldots, \lambda_n)\in {\bf D_{f,>}^m}(\CC)$ and
$
\varphi({\bf W}_{i,j})=\sum\limits_{\beta_1\in \FF_{n_1}^+,\ldots, \beta_k\in \FF_{n_k}^+} c_{\beta_1,\ldots, \beta_k}
{\bf W}_{1,\beta_1}\cdots {\bf W}_{k,\beta_k}
$
is in the Hardy algebra $F^\infty({\bf
D^m_f})$, then $\sum\limits_{\beta_1\in \FF_{n_1}^+,\ldots, \beta_k\in \FF_{n_k}^+} |c_{\beta_1,\ldots, \beta_k}|
|\lambda_{1,\beta_1}|\cdots |\lambda_{k,\beta_k}|<\infty$.
 Indeed, since
  $\varphi({\bf W}_{i,j})(1)\in \otimes_{i=1}^k F^2(H_{n_i})$, we  have
$$K_1:=\sum\limits_{\beta_1\in \FF_{n_1}^+,\ldots, \beta_k\in \FF_{n_k}^+}
 |c_{\beta_1,\ldots, \beta_k}|^2 \frac{1}{b_{1,\beta_1}^{(m_1)}\cdots
  b_{k,\beta_k}^{(m_k)}}<\infty.
$$
  On the other hand, since
$\lambda=(\lambda_1,\ldots, \lambda_n)\in {\bf D_{f,>}^m}(\CC)$,
we deduce that  $$K_2:=\sum\limits_{\beta_1\in \FF_{n_1}^+,\ldots, \beta_k\in \FF_{n_k}^+}
    |\lambda_{1,\beta_1}|^2\cdots|\lambda_{k,\beta_k}|^2 {b_{1,\beta_1}^{(m_1)}\cdots
  b_{k,\beta_k}^{(m_k)}}<\infty.
  $$ Applying Cauchy's inequality, we obtain
\begin{equation*}
\begin{split}
&\sum\limits_{\beta_1\in \FF_{n_1}^+,\ldots, \beta_k\in \FF_{n_k}^+} |c_{\beta_1,\ldots, \beta_k}|
|\lambda_{1,\beta_1}|\cdots |\lambda_{k,\beta_k}|\leq (K_1 K_2)^{1/2}.
\end{split}
\end{equation*}
 We note that the $w^*$-continuous and multiplicative map
$\Phi_\lambda:F^\infty({\bf D^m_f})\to \CC$ satisfies the equation
 $\Phi_\lambda(\varphi({\bf W}_{i,j})):=\varphi(\lambda)$.
Indeed, in this case we have
\begin{equation*}
\begin{split}
\left<\varphi({\bf W}_{i,j})1, u_\lambda\right>&= \left<
\sum\limits_{\beta_1\in \FF_{n_1}^+,\ldots, \beta_k\in \FF_{n_k}^+} c_{{\beta}_1,\ldots, {\beta}_k}\frac{1}{\sqrt{b_{1, \beta_1}^{(m_1)}}}\cdots \frac{1}{\sqrt{b_{k, \beta_k}^{(m_k)}}}
e_{ \beta_1}^1\otimes \cdots \otimes e_{ \beta_k}^k,  u_\lambda\right>\\
&=\sum\limits_{\beta_1\in \FF_{n_1}^+,\ldots, \beta_k\in \FF_{n_k}^+} c_{\beta_1,\ldots, \beta_k}
\lambda_{1,\beta_1}\cdots \lambda_{k,\beta_k}
=\varphi(\lambda).
\end{split}
\end{equation*}

 We recall   that the joint  right spectrum
  $\sigma_r(T_1,\ldots, T_n)$ of an   $n$-tuple
 $(T_1,\ldots, T_n)$ of operators
   in $B(\cH)$ is the set of all $n$-tuples
    $(\mu_1,\ldots, \mu_n)$  of complex numbers such that the
     right ideal of $B(\cH)$  generated by the operators
     $\mu_1I-T_1,\ldots, \mu_nI-T_n$ does
      not contain the identity operator.
We recall \cite{Po-unitary} that
  $(\mu_1,\ldots, \mu_n)\notin \sigma_r(T_1,\ldots, T_n)$
  if and only if  there exists $\delta>0$ such that
  $\sum\limits_{i=1}^n (\mu_iI-T_i)
  (\overline{\mu}_iI-T_i^*)\geq \delta I$.

\begin{proposition}\label{right-spec}  Let $J$ be  a WOT-closed left ideal of the Hardy space $F^\infty({\bf D_f^m})$ and  let  ${\bf S}=\{{\bf S}_{i,j}\}$  be the  universal model
associated with the abstract noncommutative variety ${\bf {\cV}^m_{f,{\it J}}}$. If  the set
${\bf {\cV}^m_{f,{\it J},>}}(\CC)$ is dense in ${\bf {\cV}^m_{f,{\it J}}}(\CC)$,
then the right joint spectrum
$\sigma_r({\bf S})$ coincide with ${\bf {\cV}^m_{f,{\it J}}}(\CC)$.

In particular, if  $\cQ\subset \CC[Z_{i,j}]$
  is a left ideal generated by  noncommutative
    homogenous polynomials,
 then the right joint spectrum
$\sigma_r({\bf S}) ={\bf {\cV}^m_{f,{\cQ}}}(\CC)$.
\end{proposition}
\begin{proof}
Let $\lambda=\{\lambda_{i,j}\}\in \sigma_r({\bf S})$.
 Since the left ideal of $B(\cN_\cQ)$ generated by the operators
 ${\bf S}_{i,j}^*-\overline{\lambda}_{i,j} I$ does not contain the identity,
 there is a pure
  state
 $\varphi$ on $B(\cN_\cQ)$ such that
 $\varphi(X({\bf S}_{i,j}^*-\overline{\lambda}_{i,j} I))=0$
  for
 any $X\in B(\cN_\cQ)$ and
  $i\in\{1,\ldots, k\}$, $j\in\{1,\ldots, n_i\}$. In particular, we have
 $\varphi({\bf S}_{i,j})= \lambda_{i,j}=\overline{
 \varphi({\bf S}_{i,j}^*)}$ and
 $$
 \varphi({\bf S}_{(\alpha)} {\bf S}_{(\alpha)}^*)=\overline{\lambda}_{(\alpha)}
  \varphi({\bf S}_{(\alpha)})=|\lambda_{(\alpha)}|^2,\qquad (\alpha)=(\alpha_1,\ldots, \alpha_k)\in \FF_{n_1}^+\times\cdots \times \FF_{n_k}^+.
 $$
 Hence, we deduce that
 \begin{equation*}
 \begin{split}
   \sum_{\alpha\in \FF_{n_i}^+, 1\leq |\alpha|\leq m} a_{i,\alpha}|\lambda_{i,\alpha}|^2  &=
   \varphi\left(\sum_{\alpha\in \FF_{n_i}^+, 1\leq |\alpha|\leq m} a_{i,\alpha} {\bf S}_{i,\alpha}
   {\bf S}_{i,\alpha}^*\right)
   \leq \left\|\sum_{\alpha\in \FF_{n_i}^+, 1\leq |\alpha|\leq n} a_{i,\alpha} {\bf S}_{i,\alpha}
   {\bf S}_{i,\alpha}^*\right\|\leq 1
\end{split}
 \end{equation*}
 for any $n\in \NN$.  Therefore,  $\sum_{\alpha\in \FF_{n_i}^+} a_{i,\alpha}|\lambda_{i,\alpha}|^2 \leq 1$, which proves  that
 $(\lambda_{i,1},\ldots, \lambda_{i,n_i})\in {\bf D}_{f_i}^1(\CC)$. Hence, we deduce that $\lambda:=\{\lambda_{i,j}\}\in {\bf D^m_f}(\CC)$. On the other hand, if $g\in \cQ$, then $g({\bf S}_{i,j})=0$ and, consequently, we obtain
 $g({\lambda}_{i,j})=\varphi(g({\bf S}_{i,j}))=0$. Therefore, $\lambda\in {\bf {\cV}^m_{f,{\cQ}}}(\CC)$.
Now, let  $\mu:=\{\mu_{i,j}\} \in {\bf {\cV}^m_{f,{\cQ}}}(\CC)$  and assume
that
 there is $\delta>0$ such that
 $$\sum_{i=1}^n \sum_{j=1}^{n_i}  \|({\bf S}_{i,j}-\mu_{i,j}I)^*h\|^2\geq \delta \|h\|^2 \quad
 \text{ for all } \ h\in \cN_\cQ.
 $$
 Take
 $$
 h=
\Gamma_\lambda:={\bf \Delta_{f,\lambda}^m}(1)^{1/2} \sum_{\beta_i\in \FF_{n_i}^+, i=1,\ldots,k}
   \overline{\lambda}_{1,\beta_1}\cdots \overline{\lambda}_{k,\beta_k}\sqrt{b_{1,\beta_1}^{(m_1)}}\cdots \sqrt{b_{k,\beta_k}^{(m_k)}}
   e^1_{\beta_1}\otimes \cdots \otimes  e^k_{\beta_k}
$$
for $\mathbf{\lambda} \in {\bf {\cV}^m_{f,{\cQ}, >}}(\CC)$  in the
 inequality above.  Due to   Theorem \ref{eigenvectors}, we  have
 $
{\bf S}_{i,j}^*\Gamma_\lambda=\overline{\lambda}_{i,j} \Gamma_\lambda$
 for any $ i\in\{1,\ldots,k\}$ and $ j\in \{1,\ldots, n_i\}$.
 Consequently, we
 deduce that
 \begin{equation*}
 \sum_{i=1}^k \sum_{j=1}^{n_i} |\lambda_{i,j}-\mu_{i,j}|^2\geq \delta\quad \text{ for all }\
 \mathbf{\lambda}= \{\lambda_{i,j}\}\in {\bf {\cV}^m_{f,{\cQ}, >}}(\CC).
 \end{equation*}
 Since  the set
${\bf {\cV}^m_{f,{\it J},>}}(\CC)$ is dense in ${\bf {\cV}^m_{f,{\it J}}}(\CC)$, this leads to  a contradiction.

 Note that  if  $\cQ\subset \CC[Z_{i,j}]$
  is a left ideal generated by  noncommutative
    homogenous polynomials, then   $\{r\mu_{i,j}\} \in {\bf {\cV}^m_{f,{\cQ},>}}(\CC)$ for any  $\{\mu_{i,j}\} \in {\bf {\cV}^m_{f,{\cQ}}}(\CC)$  and  $r\in [0,1)$. Consequently, ${\bf {\cV}^m_{f,{\cQ},>}}(\CC)$ is dense in ${\bf {\cV}^m_{f,{\cQ}}}(\CC)$.
   The proof is complete.
\end{proof}

Let  $\cQ\subset \CC[Z_{i,j}]$
  be  a left ideal generated by  noncommutative
    homogenous polynomials.
We recall that  the variety algebra $\cA(\cV_{\bf f, \cQ}^m)$  is the norm closed algebra generated by the ${\bf S}_{i,j}$ and the identity,  and the Hardy algebra $F^\infty(\cV_{\bf f, \cQ}^m)$ is the WOT-closed version.
In what follows, we identify the characters of the noncommutative
  variety
algebra $\cA({\bf {\cV}^m_{f,{\cQ}}})$.
Due to  Proposition \ref{vN2-variety}, if $\lambda\in {\bf {\cV}^m_{f,{\cQ}}}(\CC)$, then  the
evaluation functional
$$
\Phi_\lambda:\cA({\bf {\cV}^m_{f,{\cQ}}})\to
\CC,\quad\Phi_\lambda(p({\bf S}_{i,j}))=p(\lambda_{i,j}),
 $$is a character of $\cA({\bf {\cV}^m_{f,{\cQ}}})$.
\begin{theorem} Let $\cQ\subset \CC[Z_{i,j}]$
  be  a left ideal generated by  noncommutative
    homogenous polynomials and   let   $M_{\cA({\bf {\cV}^m_{f,{\cQ}}})}$ be the set of all characters of
$\cA({\bf {\cV}^m_{f,{\cQ}}})$. Then the map
$$\Psi: {\bf {\cV}^m_{f,{\cQ}}}(\CC)\to M_{\cA({\bf {\cV}^m_{f,{\cQ}}})},\quad
\Psi(\lambda)=\Phi_\lambda, $$
 is a homeomorphism of ${\bf {\cV}^m_{f,{\cQ}}}(\CC)$
onto
$M_{\cA({\bf {\cV}^m_{f,{\cQ}}})}$.
\end{theorem}

\begin{proof} The injectivity of   $\Psi$ is clear. To   prove that $\Psi$ is surjective
assume that $\Phi:\cA({\bf {\cV}^m_{f,{\cQ}}})\to\CC$ is a character. Setting
$\lambda_{i,j}:=\Phi({\bf S}_{i,j})$ for $i\in\{1,\ldots,k\}, j\in \{1,\ldots, n_i\}$, we  deduce that
$
\Phi(p({\bf S}_{i,j}))=p(\lambda_{i,j})
$
for any polynomial $p({\bf S}_{i,j})$ in $\cA({\bf {\cV}^m_{f,{\cQ}}})$.
Since $\Phi$ is a character, it   is completely
contractive. Consequently,  $(\lambda_{i,1},\ldots, \lambda_{i,n_i})\in {\bf D}_{f_i}^1(\CC)$ for each $i\in \{1,\ldots, k\}$, which implies
$(\lambda_{i,j} I_{\CC})\in {\bf D_f^m}(\CC)$.  On the other hand, if $g\in \cQ$, then $g({\bf S}_{i,j})=0$ and, consequently,
 $g({\lambda}_{i,j})=\Phi(g({\bf S}_{i,j}))=0$. Therefore, $\{\lambda_{i,j}\}\in {\bf {\cV}^m_{f,{\cQ}}}(\CC)$.
Since
$$
\Phi(p({\bf S}_{i,j}))=p(\lambda_{i,j})=
\Phi_\lambda(p({\bf S}_{i,j}))
$$
for any polynomial $p({\bf S}_{i,j})$ in $\cA({\bf {\cV}^m_{f,{\cQ}}})$,
   we must have $\Phi=\Phi_\lambda$.
To prove that $\Psi$ is a homeomorphism,  let
$\lambda^\alpha:=(\lambda_{i,j}^\alpha)$, $
\alpha\in \Lambda$, be
a net in ${\bf {\cV}^m_{f,{\cQ}}}(\CC)$ such that
$\lim_{\alpha\in
\Lambda}\lambda^\alpha=\lambda:=(\lambda_{i,j})$.
It is clear  that
$$
\lim_{\alpha\in \Lambda}\Phi_{\lambda^\alpha}(p({\bf S}_{i,j}))=
\lim_{\alpha\in \Lambda} p(\lambda^\alpha)=
p(\lambda)=\Phi_\lambda(p({\bf S}_{i,j})).
$$
 Since the set of all
polynomials $p({\bf S}_{i,j})$  is  dense
in $\cA({\bf {\cV}^m_{f,{\cQ}}})$  and $\sup_{\alpha\in \Lambda}
\|\Phi_{\lambda^\alpha}\|\leq1$,
 it follows that $\Psi$
is continuous. On the other hand, since  both ${\bf {\cV}^m_{f,{\cQ}}}(\CC)$
and
$M_{\cA({\bf {\cV}^m_{f,{\cQ}}})}$ are compact Hausdorff
spaces and $\Psi$ is a bijection,  the result follows.
The proof is complete.
\end{proof}

 Let ${\bf W}=\{{\bf W}_{i,j}\}$  be the  universal model
associated with the abstract noncommutative polydomain ${\bf D^m_f}$ and let $\cQ_c$ be the left ideal generated  by all polynomials of the form
$$Z_{i, j_1}Z_{i,j_2}-Z_{i,j_2} Z_{i,j_1},\qquad i\in\{1,\ldots,k\} \text{ and }  j_1, j_2\in \{1,\ldots, n_i\}.
$$
The universal model associated with the  abstract variety $\cV_{\bf f, \cQ_c}^m$ is the tuple ${\bf L}=({\bf L}_1,\ldots, {\bf L}_k)$ with ${\bf L}_i:=({\bf L}_{i,1},\ldots, {\bf L}_{i,n_i})$, where
the operators ${\bf L}_{i,j}$ are defined
on  $\cN_{\cQ_c}$ by setting
$${\bf L}_{i,j}:=P_{ \cN_{\cQ_c}} {\bf W}_{i,j}|_{\cN_{\cQ_c}}.
$$
 We recall that $\cN_{\cQ_c}:=(\otimes_{i=1}^k F^2(H_{n_i}))\ominus \cM_\cQ$, where   the subspace $\cM_{\cQ_c}$ of $\otimes_{i=1}^k F^2(H_{n_i})$ is defined  by
setting
$$ \cM_{\cQ_c}:=\overline{\text{\rm span}}\{{\bf W}_{(\alpha)} q({\bf W}_{i,j}) {\bf W}_{(\beta)}(1): \
 (\alpha), (\beta) \in \FF_{n_1}^+\times \cdots \times \FF_{n_k}^+, q\in \cQ_c\}.
 $$

In what follows, we will  identify the space $\cN_{\cQ_c}$ with a reproducing kernel Hilbert space of holomorphic functions in several complex variables and the Hardy algebra $F^\infty(\cV_{\bf f, \cQ_c}^m)$ is identified  with the corresponding multiplier algebra.

 Let  ${\bf f}:=(f_1,\ldots,f_k)$ be a $k$-tuple of positive regular free holomorphic functions with $f_i:= \sum_{\alpha\in
\FF_{n_i}^+} a_{i,\alpha} Z_\alpha$.
  For
each $\lambda_i=(\lambda_{i,1},\ldots, \lambda_{i,n_i})\in \CC^{n_i}$ and each $n_i$-tuple
${\bf k}_i:=(k_{i,1},\ldots, k_{i,n_i})\in \NN_0^{n_i}$, where $\NN_0:=\{0,1,\ldots
\}$, let $\lambda_i^{\bf k_{\it i}}:=\lambda_{i,1}^{k_{i,1}}\cdots \lambda_{i,n}^{k_{i,n_i}}$.
If ${\bf k}_i\in \NN_0^{n_i}$, we denote
$$
\Lambda_{\bf k_{\it i}}:=\{\alpha_i\in \FF_{n_i}^+: \ \lambda_{i,\alpha_i} =\lambda_i^{\bf
k_{\it i}} \text{ for all } \lambda_i\in \CC^{n_i}\}
$$
and define the vector
$$
w_i^{\bf k_{\it i}}:=\frac{1}{\gamma^{(m_i)}_{\bf k_{\it i}}} \sum_{\alpha_i \in
\Lambda_{\bf k_{\it i}}} \sqrt{b_{i,\alpha_i}^{(m_i)}} e^i_{\alpha_i}\in F^2(H_{n_i}), \quad
\text{ where } \ \gamma^{(m_i)}_{\bf k_{\it i}}:=\sum_{\alpha_i\in \Lambda_{\bf
k_{\it i}}} b^{(m_i)}_{i, \alpha_i}
$$
 and the
coefficients $b^{(m_i)}_{i,\alpha_i}$, $\alpha_i\in \FF_{n_i}^+$, are defined by
relation \eqref{b-coeff}. It is easy to see  that the set  $\{w_1^{\bf k_{\it 1}}\otimes \cdots \otimes w_k^{\bf k_{\it k}}:\ {\bf k}_i\in
\NN_0^{n_i}, i\in \{1,\ldots, k\}\}$ consists  of orthogonal vectors in $\otimes_{i=1}^kF^2(H_{n_i})$ and
$$\|w_1^{\bf k_{\it 1}}\otimes \cdots \otimes w_k^{\bf k_{\it k}}\|=\frac{1}{\sqrt{\gamma^{(m_1)}_{\bf k_{\it 1}}}}\cdots \frac{1}{\sqrt{\gamma^{(m_k)}_{\bf k_{\it k}}}}.
$$
 Let
$F_s^2({\bf D^m_f})$  be the closed span of these vectors. The Hilbert space  $F_s^2({\bf D^m_f})\subset\otimes_{i=1}^k F^2(H_{n_i})$ is  called
the symmetric  tensor product Fock space associated with the abstract
noncommutative domain ${\bf D^m_f}$.

  For $z=(z_1,\ldots, z_n)$ and $w:=(w_1,\ldots, w_n)$ in $\CC^n$, we use the notation $z\overline{w}:=(z_1\overline{w}_1,\ldots, z_n\overline{w}_n)$.

\begin{theorem}\label{symm-Fock} Let ${\bf W}=\{{\bf W}_{i,j}\}$  be the  universal model
associated with the noncommutative polydomain ${\bf D^m_f}$,  and  let
$\cQ_c$ be the left ideal generated  by all polynomials of the form
$$Z_{i, j_1}Z_{i,j_2}-Z_{i,j_2} Z_{i,j_1},\qquad i\in\{1,\ldots,k\} \text{ and }  j_1, j_2\in \{1,\ldots, n_i\}.
$$
 Then the following statements hold.
 \begin{enumerate}
 \item[(i)]
 $
 F_s^2({\bf D^m_f})=\overline{\text{\rm span}}\{\Gamma_\lambda: \ \lambda\in
{\bf D_{f,>}^m}(\CC)\}=\cN_{\cQ_c}:=(\otimes_{i=1}^k F^2(H_{n_i}))\ominus
\cM_{\cQ_c}$.
\item[(ii)] The  space $F_s^2({\bf D^m_f})$ can be
identified with the Hilbert space $H^2({\bf D_{f,>}^m}(\CC))$ of
all functions $\varphi:{\bf D_{f,>}^m}(\CC)\to \CC$ which admit
a power series representation
$$\varphi(\lambda_{i,j})=\sum_{{\bf k}_1\in
\NN_0^{n_1},\ldots, {\bf k}_k\in \NN_0^{n_k}} c_{{\bf k}_1,\ldots, {\bf k}_k} \lambda_1^{\bf k_{\it 1}}\cdots \lambda_k^{\bf k_{\it k}}
$$
 with
$$
\|\varphi\|_2^2=\sum_{{\bf k}_1\in
\NN_0^{n_1},\ldots, {\bf k}_k\in \NN_0^{n_k}}|c_{{\bf k}_1,\ldots, {\bf k}_k}|^2\frac{1}{\gamma^{(m_1)}_{\bf k_{\it 1}}}\cdots \frac{1}{\gamma^{(m_k)}_{\bf k_{\it k}}}<\infty.
$$
More precisely, every  element  $\varphi=\sum_{{\bf k}_1\in
\NN_0^{n_1},\ldots, {\bf k}_k\in \NN_0^{n_k}} c_{{\bf k}_1,\ldots, {\bf k}_k} w_1^{\bf k_{\it 1}}\otimes \cdots \otimes w_k^{\bf k_{\it k}}$ in $F_s^2({\bf D^m_f})$  has a functional
representation on ${\bf D_{f,>}^m}(\CC)$ given by
$$
\varphi(\lambda):=\left<\varphi, u_\lambda\right>=\sum_{{\bf k}_1\in
\NN_0^{n_1},\ldots, {\bf k}_k\in \NN_0^{n_k}} c_{{\bf k}_1,\ldots, {\bf k}_k} \lambda_1^{\bf k_{\it 1}}\cdots \lambda_k^{\bf k_{\it k}}, \qquad \lambda=(\lambda_1,\ldots,
\lambda_k)\in {\bf D_{f,>}^m}(\CC),
$$
and
$$
|\varphi(\lambda)|\leq
\frac{\|\varphi\|_2}{\sqrt{{\bf \Delta_{f,\lambda}^m}(1)}},\qquad \lambda=(\lambda_1,\ldots,
\lambda_k)\in {\bf D_{f,>}^m}(\CC),
$$
where ${\bf \Delta_{f,\lambda}^m}(1)=(1-\Phi_{f_1,\lambda_1}(1))^{m_1}\cdots (1-\Phi_{f_k,\lambda_k}(1))^{m_k}$ and  $u_\lambda:=\frac{1}{{\bf \Delta_{f,\lambda}^m}(1)^{1/2} }\Gamma_\lambda$.
\item[(iii)]
 The mapping $\kappa^c_{\bf f}:{\bf D_{f,>}^m}(\CC)\times {\bf D_{f,>}^m}(\CC)\to \CC$ defined by
$$
\kappa^c_{\bf f}(\mu,\lambda):=\frac{1}{\prod_{i=1}^k\left(1- f_i(\mu_i\overline{\lambda}_i )\right)^{m_i}},
$$
  where
$ \lambda=(\lambda_1,\ldots,
\lambda_k)$ and  $\mu=(\mu_1,\ldots,\mu_k)$  are in ${\bf D_{f,>}^m}(\CC)$,
is positive definite and
$$\kappa^c_{\bf f}(\mu,\lambda)= \left<u_\lambda,
u_\mu\right>.
$$
\end{enumerate}
\end{theorem}

\begin{proof} We prove that
$$\overline{\text{\rm span}}\{\Gamma_\lambda: \ \lambda\in
{\bf D_{f,>}^m}(\CC)\}\subseteq F_s^2({\bf D^m_f})\subseteq
\cN_{\cQ_c}.
$$
Note that the first inclusion is due to the fact that
\begin{equation}
\label{ulamda}
u_\lambda=\sum_{{\bf k}_1\in
\NN_0^{n_1},\ldots, {\bf k}_k\in \NN_0^{n_k}}\lambda_1^{\bf k_{\it 1}}\cdots \lambda_k^{\bf k_{\it k}}\gamma^{(m_1)}_{\bf k_{\it 1}}\cdots \gamma^{(m_k)}_{\bf k_{\it k}} w_1^{\bf k_{\it 1}}\otimes \cdots \otimes w_k^{\bf k_{\it k}}
\end{equation}
for $ \lambda=(\lambda_1,\ldots,
\lambda_n)\in {\bf D_{f,>}^m}(\CC)$. To prove the second inclusion, note that, due
to
  the definition of the universal model ${\bf W}=\{{\bf W}_{i,j}\}$ , we have
\begin{equation*}
 \begin{split}
 &\left<w_i^{\bf k_{\it i}},
{\bf W}_{i,\gamma_i}({\bf W}_{i,j_1}{\bf W}_{i, j_2}-{\bf W}_{i,j_2}{\bf W}_{i,j_1}){\bf W}_{i,\beta_i}(1)\right>\\
&\qquad = \frac{1}{\gamma^{(m_i)}_{\bf
k_{\it i}}}\left<\sum_{\alpha_i \in \Lambda_{\bf k_{\it i}}} \sqrt{b^{(m_i)}_{i,\alpha_i}}
e^i_{\alpha_i}, \frac{1}{\sqrt{b^{(m_i)}_{i,\gamma_i g_{j_1}g_{j_2}\beta_i}}} e^i_{\gamma_i
g_{j_1}g_{j_2}\beta_i}- \frac{1}{\sqrt{b^{(m_i)}_{i,\gamma_i g_{j_2}g_{j_1}\beta_i}}}
e^i_{\gamma_i g_{j_2}g_{j_1}\beta_i}\right>=0
\end{split}
\end{equation*}
for any ${\bf k_{\it i}}\in \NN_0^{n_i}$, $\gamma_i, \beta_i \in \FF_{n_i}^+$,
$i\in \{1,\ldots, k\}$. This implies  that $w_1^{\bf k_{\it 1}}\otimes \cdots \otimes w_k^{\bf k_{\it k}}\in \cN_{\cQ_c}$ and
proves our assertion. To complete the proof of part (i), it is
enough to show that
$$
\cN_{\cQ_c}\subseteq \overline{\text{\rm span}}\{\Gamma_\lambda: \
\lambda\in {\bf D_{f,>}^m}(\CC)\}.
$$
To this end, assume that there is a vector $ x:=\sum\limits_{\beta_1\in \FF_{n_1}^+,\ldots, \beta_k\in \FF_{n_k}^+} c_{\beta_1,\ldots, \beta_k}
e_{\beta_1}^1\otimes\cdots \otimes e_{\beta_k}^k\in \cN_{\cQ_c}$ and $x\perp u_\lambda$ for
all $\lambda\in {\bf D_{f,>}^m}(\CC)$. Then, using relation \eqref{ulamda}, we obtain
\begin{equation*}
\begin{split}
&\left<\sum\limits_{\beta_1\in \FF_{n_1}^+,\ldots, \beta_k\in \FF_{n_k}^+} c_{\beta_1,\ldots, \beta_k}
e_{\beta_1}^1\otimes\cdots \otimes e_{\beta_k}^k, u_\lambda\right>\\
&\qquad
=\sum_{{\bf k}_1\in
\NN_0^{n_1},\ldots, {\bf k}_k\in \NN_0^{n_k}}\left(\sum_{\beta_i\in\Lambda_{\bf k_{\it i}}, i\in \{1,\ldots, k\}}
c_{\beta_1,\ldots, \beta_k} \sqrt{b_{1,\beta_1}^{(m_1)}}\cdots \sqrt{b_{k,\beta_k}^{(m_k)}}\right)\lambda_1^{\bf k_{\it 1}}\cdots \lambda_k^{\bf k_{\it k}}=0
\end{split}
\end{equation*}
for any $\lambda\in {\bf D_{f,>}^m}(\CC)$. Since ${\bf D_{f,>}^m}(\CC)$ contains an open polydisc in $\CC^{n_1+\cdots + n_k}$, we deduce
that
\begin{equation}\label{sigma=0}
\sum_{\beta_i\in\Lambda_{\bf k_{\it i}}, i\in \{1,\ldots, k\}}
c_{\beta_1,\ldots, \beta_k} \sqrt{b_{1,\beta_1}^{(m_1)}}\cdots \sqrt{b_{k,\beta_k}^{(m_k)}}=0 \quad
\text{ for all } \ {\bf k}_i\in \NN_0^{n_i}, i\in \{1,\ldots, k\}.
\end{equation}
For each $\gamma_i\in \FF_{n_i}^+$ and  $i\in \{1,\ldots, k\}$,    set $\Omega(\gamma_1,\ldots, \gamma_k):=\frac{c_{\gamma_1,\ldots, \gamma_k}}{ \sqrt{b_{1,\gamma_1}^{(m_1)}}\cdots \sqrt{b_{k,\gamma_k}^{(m_k)}}}$.
Fix $\beta_i^0\in \Lambda_{\bf k_{\it i}}$ and let $\beta_i\in \Lambda_{\bf k_{\it i}}$
be such that $\beta_i$ is obtained from $\beta_i^0$ by transposing just
two generators. We can assume that $\beta_i^0=\gamma_i g^i_{j_1} g^i_{j_2}\omega_i$
and $\beta_i=\gamma_i g^i_{j_2} g^i_{j_1}\omega_i$ for some $\gamma_i,\omega_i\in \FF_{n_i}^+$
and $j_1\neq j_2$, $j_1,j_2\in \{1,\ldots,n_i\}$. Since $x\in
\cN_{\cQ_c}=\otimes_{i=1}^k F^2(H_{n_i})\ominus \cM_{\cQ_c}$, we must have
$$
\left<x,\otimes_{i=1}^k[{\bf W}_{i,\gamma_i}({\bf W}_{i,j_1}{\bf W}_{i, j_2}-{\bf W}_{i,j_2}{\bf W}_{i,j_1})W_{i,\omega_i}(1)]\right>=0,
$$
which implies $
\Omega(\beta_1^0,\ldots, \beta_k^0) =\Omega(\beta_1,\ldots, \beta_k)$.

Since any element $\gamma_i\in \Lambda_{\bf k_{\it i}}$ can be obtained from
$\beta_i^0$  by successive  transpositions, repeating the above
argument, we deduce that
$\Omega(\beta_1^0,\ldots, \beta_k^0) =\Omega(\gamma_1,\ldots, \gamma_k)$.
Setting $t:= \Omega(\beta_1^0,\ldots, \beta_k^0)$, we have
$c_{\gamma_1,\ldots, \gamma_k}=t\sqrt{b_{1,\gamma_1}^{(m_1)}}\cdots \sqrt{b_{k,\gamma_k}^{(m_k)}}$, $\gamma_i\in \Lambda_{\bf k_{\it i}}$, and
relation  \eqref{sigma=0} implies  $t=0$. Therefore, $c_{\gamma_1,\ldots, \gamma_k}=0$ for any $\gamma_i\in
\Lambda_{\bf k_{\it i}}$ and ${\bf k_{\it i}}\in \NN_0^{n_i}$, which implies $x=0$. Therefore,
we have $\cN_{\cQ_c}= \overline{\text{\rm span}}\{\Gamma_\lambda: \
\lambda\in {\bf D_{f,>}^m}(\CC)\}$.

Now, we prove part (ii) of the theorem.  Any element
$\varphi\in F_s^2({\bf D^m_f})$ has a unique representation
$\varphi=\sum_{{\bf k}_1\in
\NN_0^{n_1},\ldots, {\bf k}_k\in \NN_0^{n_k}} c_{{\bf k}_1,\ldots, {\bf k}_k} w_1^{\bf k_{\it 1}}\otimes \cdots \otimes w_k^{\bf k_{\it k}}$ with
 $$
\|\varphi\|_2^2=\sum_{{\bf k}_1\in
\NN_0^{n_1},\ldots, {\bf k}_k\in \NN_0^{n_k}}|c_{{\bf k}_1,\ldots, {\bf k}_k}|^2\frac{1}{\gamma^{(m_1)}_{\bf k_{\it 1}}}\cdots \frac{1}{\gamma^{(m_k)}_{\bf k_{\it k}}}<\infty.
$$
 It is easy to see that
\begin{equation*}
\left<w_1^{\bf k_{\it 1}}\otimes \cdots \otimes w_k^{\bf k_{\it k}},u_\lambda\right>= \lambda_1^{\bf k_{\it 1}}\cdots \lambda_k^{\bf k_{\it k}}
\end{equation*}
for any $\lambda\in {\bf D_{f,>}^m}(\CC)$ and ${\bf k_{\it i}}\in
\NN_0^{n_i}$, $i\in \{1,\ldots, k\}$. Consequently,   $\varphi $  has a functional
representation on ${\bf D_{f,>}^m}(\CC)$ given by
$$
\varphi(\lambda):=\left<\varphi, u_\lambda\right>=\sum_{{\bf k}_1\in
\NN_0^{n_1},\ldots, {\bf k}_k\in \NN_0^{n_k}} c_{{\bf k}_1,\ldots, {\bf k}_k} \lambda_1^{\bf k_{\it 1}}\cdots \lambda_k^{\bf k_{\it k}}, \qquad \lambda=(\lambda_1,\ldots,
\lambda_n)\in {\bf D_{f,>}^m}(\CC),
$$
and
$$|\varphi(\lambda)|\leq \frac{\|\varphi\|_2}{\sqrt{{\bf \Delta_{f,\lambda}^m}(1)}}.
$$
 This shows that  $F_s^2({\bf
D^m_f})$  can be identified with $H^2({\bf D_{f,>}^m}(\CC))$.
Now, we prove part (iii). Note that if $
 (\lambda_1,\ldots, \lambda_k)$ and $\mu=(\mu_1,\ldots, \mu_n)$ are in
${\bf D_{f,>}^m}(\CC)$, then
$$
\left| \sum_{\alpha\in \FF_{n_i}^+} a_{i,\alpha_i}\lambda_{i,\alpha_i} \overline{\mu}_{i,\alpha_i}\right|\leq \left( \sum_{\alpha\in \FF_{n_i}^+} a_{i,\alpha_i}|\lambda_{i,\alpha_i}|^2\right)^{1/2}\left( \sum_{\alpha\in \FF_{n_i}^+} a_{i,\alpha_i}|\mu_{i,\alpha_i}|^2\right)^{1/2}<1.
$$
Using relation \eqref {b-coeff}, we deduce that
\begin{equation*}
\begin{split}
\kappa^c_{\bf f} (\mu,\lambda)&=\prod_{i=1}^k \left( 1-f_i(\mu_i\overline{\lambda}_i )\right)^{-m}=\prod_{i=1}^k \left(1-\sum_{\alpha\in \FF_{n_i}^+} a_{i,\alpha_i}\lambda_{i,\alpha_i} \overline{\mu}_{i,\alpha_i}\right)^{-m_i}\\
&=\sum\limits_{\beta_1\in \FF_{n_1}^+,\ldots, \beta_k\in \FF_{n_k}^+}
     {b_{1,\beta_1}^{(m_1)}\cdots
  b_{k,\beta_k}^{(m_k)}}\lambda_{1,\beta_1}\cdots \lambda_{k,\beta_k}\overline{\mu}_{1,\beta_1}\cdots \overline{\mu}_{k,\beta_k}\\
  &=\left<u_\lambda, u_\mu\right>.
\end{split}
\end{equation*}
    The proof is complete.
\end{proof}

\begin{theorem}\label{multipliers}
The Hardy algebra $F^\infty(\cV_{\bf f, \cQ_c}^m)$  coincides with the
algebra  $H^\infty({\bf D^m_{f,>}}(\CC))$  of all
multipliers of the Hilbert space $H^2({\bf D_{f,>}^m}(\CC))$.
\end{theorem}
\begin{proof}
   Let
$\varphi({\bf W}_{i,j})\in F_n^\infty({\bf D^m_f})$ and set
$M_\varphi:=P_{F_s^2({\bf D^m_f})} \varphi({\bf W}_{i,j})|_{F_s^2({\bf D^m_f})}$.
According to Theorem \ref{eigenvectors}, Proposition \ref{w*},  and Theorem \ref{symm-Fock}, we have $ F_s^2({\bf D^m_f})=\cN_{\cQ_c}$, the vector
 $\Gamma_\lambda$ is in $ F_s^2({\bf D^m_f})$
for $\lambda\in {\bf D^m_{f,>}}(\CC)$, and $\varphi({\bf W}_{i,j})^*\Gamma_\lambda= \overline{\varphi(\lambda)} \Gamma_\lambda$. Consequently,  we obtain
\begin{equation*}
\begin{split}
[M_\varphi \psi](\lambda)&=\left<M_\varphi\psi, u_\lambda\right>
=\left<\varphi({\bf W}_{i,j})\psi, u_\lambda\right>\\
&=\left< \psi,\varphi({\bf W}_{i,j})^*u_\lambda\right>
=\left< \psi,\overline{\varphi(\lambda)} u_\lambda\right>
=\varphi(\lambda)\psi(\lambda)
\end{split}
\end{equation*}
for any $\psi\in F_s^2({\bf D^m_f})$ and $\lambda\in {\bf D^m_{f,>}}(\CC)$. Therefore, $M_\varphi$ is a multiplier of $F_s^2({\bf D^m_f})$.
In particular, the operator ${\bf L}_{i,j}$ is the  multiplier
 by the coordinate function ${\lambda_{i,j}}$.
 Now, we show that $H^\infty({\bf D^m_{f,>}}(\CC))$ is  included in $F^\infty(\cV_{\bf f, \cQ_c}^m)$,  the weakly closed algebra generated by the operators ${\bf L}_{i,j}$ and the identity.
Suppose that $g=\sum_{{\bf k}_1\in
\NN_0^{n_1},\ldots, {\bf k}_k\in \NN_0^{n_k}} c_{{\bf k}_1,\ldots, {\bf k}_k} w_1^{\bf k_{\it 1}}\otimes \cdots \otimes w_k^{\bf k_{\it k}}$ is a
bounded multiplier, i.e., $M_g\in B(F_s^2({\bf D^m_f}))$. As in  \cite{Po-Berezin-poly} (Proposition 3.2),  using Cesaro means,  one can find a sequence $p_n$  of polynomials
 in $w_1^{\bf k_{\it 1}}\otimes \cdots \otimes w_k^{\bf k_{\it k}}$, where
$ {\bf k}_1\in \NN_0^{n_1},\ldots, {\bf k}_k\in \NN_0^{n_k}$, such that $M_{p_n}$ converges
to $M_g$ in the strong operator topology and, consequently, in  the
$WOT$-topology.   Since $M_{p_n}$ is a polynomial in ${\bf L}_{i,j}$ and the identity,  our assertion follows.

Conversely, assume that  the operator  $Y\in B(F_s^2({\bf D}^m_f))$ is in $F^\infty(\cV_{\bf f, \cQ_c}^m)$.    Then $Y$  leaves
invariant all the invariant subspaces under each operator
${\bf L}_{i,j}$. Due to Theorem \ref{eigenvectors}, we have
${\bf L}_{i,j}^*u_\lambda=\overline{\lambda}_{i,j} u_\lambda$ for any $\lambda\in
{\bf D^m_{f,>}}(\CC)$. Therefore,
the vector $u_\lambda$ must be an eigenvector for $Y^*$.
Consequently, there is a function $\varphi:{\bf D^m_{f,>}}(\CC)\to \CC$ such that
$Y^*u_\lambda=\overline{\varphi(\lambda)} u_\lambda$ for any
$\lambda\in {\bf D^m_{f,>}}(\CC)$. Note that, if $f\in
F_s^2({\bf D^m_f})$, then, due to Theorem \ref{symm-Fock}, $Yf$
has the functional representation
$$
(Yf)(\lambda)=\left<Yf,u_\lambda\right>=\left<f,Y^*u_\lambda\right>=
\varphi(\lambda)f(\lambda),\qquad   \lambda\in{\bf D^m_{f,>}}(\CC).
$$
In particular, if $f=1$, then  the the functional representation  of
$Y(1)$  coincide with $\varphi$. Consequently, $\varphi$ admits a
power series representation  on ${\bf D^m_{f,>}}(\CC)$ and can
be identified with $Y(1)\in F_s^2({\bf D^m_f})$. Moreover, the
equality above shows that $\varphi f\in H^2({\bf D^m_{f,>}}(\CC))$ for any $f\in F_s^2({\bf D^m_f})$. The proof is complete.
\end{proof}
 We need to recall some definitions.
 The set of all
invariant subspaces of $A\in B(\cH)$ is denoted by $\text{\rm Lat~}A$. Given  $\cU\subset B(\cH)$, we define
$
\text{\rm Lat~}\cU=\bigcap_{A\in\cU}\text{\rm Lat~}A.
$
 If $\cS$ is any collection of subspaces of $\cH$,
then we define $\text{\rm Alg~}\cS$ by setting $\text{\rm
Alg~}\cS:=\{A\in B(\cH):\ \cS\subset\text{\rm Lat~}A\}.$  The algebra $~\cU\subset B(\cH)~$ is called  reflexive if
$\cU=\text{\rm Alg Lat~}\cU.$

A closser look at the proof of Theorem \ref{multipliers} reveals the following result.

\begin{corollary}\label{reflexivity}
The Hardy algebra $F^\infty(\cV_{\bf f, \cQ_c}^m)$  is reflexive.
\end{corollary}

Now, we make a few remarks in    the particular case when $n_1=\cdots =n_k=n$. Let $\cQ_{cc}$ be the left ideal of $\CC[Z_{i,j}]$ generated by the polynomials
$Z_{i, j_1}Z_{i,j_2}-Z_{i,j_2} Z_{i,j_1}$ and $Z_{i,j}-Z_{p,j}$, where
 $i,p\in\{1,\ldots,k\}$ and  $j_1, j_2, j\in \{1,\ldots, n\}$.
 The universal model associated with the variety $\cV_{\bf f, \cQ_{cc}}^m$ is the $n$-tuple
 $C=(C_1,\ldots,C_n)$, where $C_j:=P_{\cN_{\cQ_{cc}}} {\bf W}_{1,j}|_{\cN_{\cQ_{cc}}}$ for $j\in \{1,\ldots,n\}$. Note that, in this case,
we have  $\cV_{\bf f, \cQ_{cc},>}^m(\CC)=\cap_{i=1}^k{\bf D}^1_{f_i,>}(\CC)$.
Similarly  to  Theorem \ref{symm-Fock}, one can show that the space $\cN_{\cQ_{cc}}$ can be identified with a reproducing kernel Hilbert space with kernel
$$
\kappa_f^{cc}(z,w):=\frac{1}{\prod_{i=1}^k\left(1- f_i(z\overline{w} )\right)^{m_i}},
$$
  where $z=(z_1,\ldots,z_n)$ , $w=(w_1,\ldots,w_n)$ are in the  set
  $\cV_{\bf f, \cQ_{cc},>}^m(\CC)\subset \CC^n$. We remark that in the particular case when $f_1=\cdots =f_k=Z_1+\cdots +Z_n$ and $m_1=\cdots=m_k=1$, we obtain   the reproducing kernel
  $(z,w)\mapsto \frac{1}{(1-\left<z,w\right>)^k}$ on the unit ball $\BB_n$. In this case, the reproducing kernel Hilbert spaces  are the  Hardy-Sobolev spaces (see \cite{BeaBur1}). The case when $k=n$ corresponds to  the Hardy space of the ball, and the case when $k=n+1$  corresponds to the Bergman space.

\bigskip

\section{Isomorphisms of universal operator algebras }

In this section, we show that the isomorphism  problem for the universal polydomain algebras is closed related to to the biholomorphic  equivalence of Reinhardt domains in several complex variables. Our results also show that there are  many non-isomorphic  polydomain algebras.

Given a Hilbert space $\cH$,     the radial polydomain associated with the abstract ${\bf D_f^m}$ is the set
$$
{\bf D_{f,\text{\rm rad}}^m}(\cH):= \bigcup_{0\leq r<1} r{\bf D_{f}^m}(\cH)\subseteq {\bf D_{f}^m}(\cH).
$$
 A formal power series
$
\varphi=\sum_{(\alpha) \in \FF_{n_1}^+\times \cdots \times \FF_{n_k}^+} a_{(\alpha)}  Z_{(\alpha)}$, $ a_{(\alpha)} \in \CC, $
in ideterminates $Z_{i,j}$
is called  {\it free holomorphic function} on the
{\it abstract radial polydomain}
${\bf D_{f,\text{\rm rad}}^m}:=
\{{\bf D_{f,\text{\rm rad}}^m}(\cH):\ \cH \text{ is a Hilbert space}\}$ if the series
$$
\varphi(X_{i,j}):=\sum_{q=0}^\infty \sum_{{(\alpha)\in \FF_{n_1}^+\times \cdots \times\FF_{n_k}^+ }\atop {|\alpha_1|+\cdots +|\alpha_k|=q}} a_{(\alpha)} X_{(\alpha)}
$$
is convergent in the operator norm topology for any $X=\{X_{i,j}\}\in {\bf D_{f, \text{\rm rad}}^m}(\cH)$   and any Hilbert space $\cH$. We denote by $Hol({\bf D_{f, \text{\rm rad}}^m})$ the set of all free holomorphic functions on the abstract radial polydomain ${\bf D_{f,\text{\rm rad}}^m}$. Let
$H^\infty({\bf D_{f,\text{\rm rad}}^m})$  denote the set of  all
elements $\varphi$ in $Hol({\bf D_{f,\text{\rm rad}}^m})$     such
that
$$\|\varphi\|_\infty:= \sup \|\varphi(X_{i,j} )\|<\infty,
$$
where the supremum is taken over all   $ \{X_{i,j}\}\in {\bf D_{f,\text{\rm rad}}^m(\cH)}$ and any Hilbert space
$\cH$. One can  show that $H^\infty({\bf D_{f,\text{\rm
rad}}^m)}$ is a Banach algebra under pointwise multiplication and the
norm $\|\cdot \|_\infty$.
For each $p\in \NN$, we define the norms $\|\cdot
\|_p:M_{p\times p}\left(H^\infty({\bf D_{f,\text{\rm rad}}^m})\right)\to
[0,\infty)$ by setting
$$
\|[\varphi_{st}]_{p\times p}\|_p:= \sup \|[\varphi_{st}(X )]_{p\times p}\|,
$$
where the supremum is taken over all  $ X:=\{X_{i,j}\}\in {\bf D_{f,\text{\rm rad}}^m(\cH)}$ and any Hilbert space
$\cH$.
The norms  $\|\cdot\|_p$, $p\in \NN$,
determine  an operator space structure  on $H^\infty({\bf D_{f,\text{\rm rad}}^m})$,
 in the sense of Ruan (\cite{Pa-book}). Throughout  this section, we assume that  ${\bf D_f^m} (\cH)$  is
closed in the operator norm topology for any Hilbert space $\cH$. Then  we have ${\bf D_{f, \text{\rm rad}}^m}(\cH)^- ={\bf D_f^m} (\cH)$. Note that the
interior  of ${\bf D_f^m} (\cH)$, which we denote by $Int({\bf
D_f^m} (\cH))$, is a subset of ${\bf D_{f, \text{\rm rad}}^m}(\cH)$.
 We remark that if  ${\bf q}=(q_1,\ldots, q_k)$ is a $k$-tuple of   positive regular  noncommutative polynomials, then ${\bf D_q^m} (\cH)$  is
closed in the operator norm topology.

 We  denote by  $A({\bf
D_{f,\text{\rm rad}}^m})$   the set of all  elements $g$
  in $Hol({\bf
D_{f,\text{\rm rad}}^m})$   such that the mapping
$${\bf
D_{f,\text{\rm rad}}^m}(\cH)\ni X\mapsto
g(X)\in B(\cH)$$
 has a continuous extension to  $[{\bf
D_{f,\text{\rm rad}}^m}(\cH)]^-={\bf
D_{f}^m}(\cH)$ for any Hilbert space $\cH$. We remark that  $A({\bf
D_{f,\text{\rm rad}}^m})$ is a  Banach algebra under pointwise
multiplication and the norm $\|\cdot \|_\infty$, and it has an operator space structure under the norms $\|\cdot \|_p$, $p\in \NN$.  Moreover, we can
identify the polydomain algebra $\cA({\bf D^m_f})$ with the subalgebra
 $A({\bf D_{f,\text{\rm rad}}^m})$, as follows.
The map
$ \Phi:A({\bf D_{f,\text{\rm rad}}^m})\to \cA({\bf D^m_f}) $
defined by
$$
\Phi\left(\sum\limits_{(\alpha)} a_{(\alpha)} Z_{(\alpha)}\right):=\sum\limits_{(\alpha)} a_{(\alpha)} {\bf W}_{(\alpha)}
$$
is a completely isometric isomorphism of operator algebras.
If  $g:=\sum\limits_{(\alpha)} a_{(\alpha)} Z_{(\alpha)}$
is  a free holomorphic function on the abstract radial  polydomain ${\bf
D_{f,\text{\rm rad}}^m}$,
then  $g\in A({\bf D_{f,\text{\rm
rad}}^m})$
 if and only if $g(r{\bf W}_{i,j}):=\sum_{q=0}^\infty \sum_{{(\alpha)\in \FF_{n_1}^+\times \cdots \times\FF_{n_k}^+ }\atop {|\alpha_1|+\cdots +|\alpha_k|=q}} r^q  a_{(\alpha)} {\bf W}_{(\alpha)}$ is convergent in the  norm topology  as $r\to 1$.
 In this case, there exists a unique $\varphi\in \cA({\bf D^m_f})$ with $g={\bf
B}[\varphi]$, where ${\bf B}$ is the  noncommutative Berezin
transform  associated with the abstract  polydomain ${\bf D^m_f}$, with the properties
 $$
\Phi(g)=\lim_{r\to 1}g(r{\bf W}_{i,j})  \quad \text{ and } \quad  \Phi^{-1}(\varphi)={\bf B}[\varphi],\quad \varphi\in \cA({\bf D^m_f}).
$$
We proved in \cite{Po-Berezin-poly}(see Proposition 4.4) that if   $p\in \NN$ and
$\varphi$ is  a  free holomorphic function   on the abstract radial polydomain ${\bf D_{f,\text{\rm rad}}^m}$,
  then its representation on $\CC^p$,  i.e., the map $\check \varphi$  defined by
$$
  \CC^{(n_1+\cdots +n_k)p^2}\supset {\bf
D_{f,\text{\rm rad}}^m}(\CC^p)\ni \Lambda\mapsto \varphi(\Lambda)\in M_{ p\times p}(\CC)\subset
\CC^{p^2}
$$
is a  holomorphic function on the interior of ${\bf
D_{f}^m}(\CC^p)$. Moreover,
  if $\varphi\in \cA({\bf D}_{{\bf f},\text{\rm rad}}^{\bf m})$,  then  its representation on $\CC^p$ has a   continuous  extension to ${\bf D_{f}^m}(\CC^p)$ and  it is  holomorphic
 on the interior of ${\bf D_{f}^m}(\CC^p)$. The continuous extension is defined by $\check\varphi(\Lambda):=\lim_{r\to 1} {\bf B}_{r\Lambda}[\varphi]$ for
  $\Lambda\in{\bf  D_f^m}(\CC^p)$.

Let $\Omega_1$, $\Omega_2$ be domains (open and connected sets) in
$\CC^d$. If there exist   holomorphic maps $\zeta:\Omega_1\to
\Omega_2$ and $\psi:\Omega_2\to \Omega_1$ such that
$\zeta\circ\psi=id_{\Omega_2}$ and $\psi\circ
\zeta=id_{\Omega_1}$, then $\Omega_1$ and $\Omega_2$ are called
biholomorphic equivalent  and $\varphi$ and $\psi$ are called
biholomorphic maps.

 \begin{theorem}
 \label{classification}  Let ${\bf f}=(f_1,\ldots, f_k)$ and ${\bf g}=(g_1,\ldots, g_{k'})$ be  tuples of positive regular free
holomorphic functions with $n$ and $\ell$ indeterminates, respectively, and let ${\bf m}:=(m_1,\ldots, m_k)\in \NN^k$ and ${\bf d}:=(d_1,\ldots, d_{k'})\in \NN^{k'}$.
If
    $\widehat\Psi:\cA({\bf D_f^m})
\to \cA({\bf D_g^d})$ is a unital completely contractive
  isomorphism, then the map $\varphi: {\bf
D_{g}^d}(\CC)\to {\bf D_{f}^m}(\CC)$ defined by
$$
\varphi(\lambda):=\left[ \lim_{r\to 1}{\bf B}_{{\bf g},r\lambda}[\widehat{\Psi}({\bf W}^{({\bf f})}_{i,j})]: \ {i\in \{1,\ldots, k\}, j\in \{1,\ldots, n_i\}}\right],\qquad \lambda \in {\bf D_{g}^d}(\CC),
$$
where  ${\bf W}^{(\bf f)}:=\{ {\bf W}^{({\bf f})}_{i,j}\}$ is the universal model of the abstract polydomain ${\bf D_f^m}$ and ${\bf B}_{{\bf g},r\lambda}$ is the Berezin  transform at $r\lambda\in {\bf D_{g,>}^d}(\CC)$,
is a homeomorphism  which is a biholomorphic function from $Int({\bf D_{g}^d}(\CC))$ onto $Int({\bf
D_{f}^m}(\CC))$ and  $n=\ell$.
\end{theorem}

\begin{proof}
Denote
 \begin{equation}
 \label{var-tilde}
 \widetilde\varphi_{i,j}:=\widehat \Psi({\bf W}_{i,j}^{({\bf f})})\in \cA({\bf D_g^d}),
\qquad i\in \{1,\ldots, k\}, j\in \{1,\ldots, n_i\},
 \end{equation}
 where  ${\bf W}^{(\bf f)}:=\{ {\bf W}^{({\bf f})}_{i,j}\}$ is the universal model of the abstract polydomain ${\bf D_f^m}$. Assume that $f_i$ has the representation  $f_i:= \sum_{\alpha\in
\FF_{n_i}^+} a_{i,\alpha} Z_{i,\alpha}$.
 Taking into account that  $0\leq \Phi_{f_i,{\bf W}_i^{({\bf f})}}(I)\leq I$, we deduce that
 $0\leq \sum_{\alpha\in
\FF_{n_i}^+, |\alpha|\leq N} a_{i,\alpha}{\bf W}_{i,\alpha}^{({\bf f})}
({\bf W}_{i,\alpha}^{({\bf f})})^*\leq I
$
for any $N\in \NN$.
Using the fact that $a_{i,\alpha}\geq 0$ and
  $\widehat
 \Psi $ is a  completely contractive homomorphism, one can easily  see that
 $0\leq \Phi_{f_i,\widetilde \varphi_i}(I)\leq I$, where $\widetilde\varphi_i:=(\widetilde\varphi_{i,1}, \ldots, \widetilde\varphi_{i,n_i})$ and  $\widetilde
 \varphi:=(\widetilde
 \varphi_1,\ldots, \widetilde\varphi_{k})$.   Due to the remarks preceding the theorem, for each $i\in \{1,\ldots, k\}$ and $ j\in \{1,\ldots, n_i\}$, the map $\varphi_{i,j}: {\bf D_g^d}(\CC)\to \CC$, given
 by
 $$\varphi_{i,j}(\lambda):=\lim_{r\to 1}{\bf B}_{{\bf g},r\lambda}[\widetilde
 \varphi_{i,j}]
 $$
  is continuous on ${\bf D_g^d}(\CC)$ and holomorphic on $Int({\bf D_{g}^d}(\CC))$. Now, we define the function $\varphi:{\bf D_g^d}(\CC)\to\CC^\ell$ by setting
   $\varphi(\lambda):=(
 \varphi_1(\lambda),\ldots, \varphi_{k}(\lambda))$, where $\varphi_i(\lambda):=(\varphi_{i,1}(\lambda), \ldots, \varphi_{i,n_i}(\lambda))$ for all  $\lambda\in {\bf D_g^d}(\CC)$.
 Since $0\leq \Phi_{f_i,\widetilde \varphi_i}(I)\leq I$, we have
 $0\leq \sum_{\alpha\in
\FF_{n_i}^+, |\alpha|\leq N} a_{i,\alpha}\widetilde\varphi_{i,\alpha}
{\widetilde\varphi}_{i,\alpha}^*\leq I$ for all $N\in \NN$.
Apply the Berezin transform at $r\lambda\in {\bf D_{g,>}^d}(\CC)$, $r\in [0,1)$, we obtain
 $$0\leq \sum_{\alpha\in
\FF_{n_i}^+, |\alpha|\leq N} a_{i,\alpha} \varphi_{i,\alpha}(r\lambda)
\overline{\varphi_{i,\alpha}(r\lambda)} \leq 1,\qquad N\in \NN.
$$
 Taking $r\to 1$ and $N\to\infty$, we deduce that $0\leq \Phi_{f_i, \varphi_i(\lambda)}(1)\leq 1$. Consequently,  $\varphi(\lambda)\in {\bf
D_{f}^m}(\CC)$ for all $\lambda\in {\bf D_g^d}(\CC)$. Moreover, the map  $\varphi:{\bf D_g^d}(\CC)\to {\bf
D_{f}^m}(\CC) $ is continuous on ${\bf D_g^d}(\CC)$ and holomorphic on $Int({\bf D_{g}^d}(\CC))$.
Now, we set
\begin{equation}
 \label{xi}
 \widetilde\xi_{i,j}:=\widehat \Psi^{-1}({\bf W}_{i,j}^{({\bf g})})\in \cA({\bf D_f^m}),
\qquad i\in \{1,\ldots, k'\}, j\in \{1,\ldots, \ell_i\},
 \end{equation}
 where  ${\bf W}^{(\bf g)}:=\{ {\bf W}^{({\bf g})}_{i,j}\}$ is the universal model of the abstract polydomain ${\bf D_g^d}$.
 Since $0\leq \Phi_{g_i,{\bf W}_i^{({\bf g})}}(I)\leq I$ and
  $\widehat
 \Psi^{-1} $ is a  completely contractive homomorphism, we deduce that
 $0\leq \Phi_{g_i,\widetilde \xi_i}(I)\leq I$, where  we set $\widetilde\xi_i:=(\widetilde\xi_{i,1}, \ldots, \widetilde\xi_{i,\ell_i})$ and $\widetilde
 \xi:=(\widetilde
 \xi_1,\ldots, \widetilde\xi_{k'})$.
 As above, for each $i\in \{1,\ldots, k'\}$ and $ j\in \{1,\ldots, \ell_i\}$, the map $\xi_{i,j}: {\bf D_f^m}(\CC)\to \CC$, given
 by
 $$
 \xi_{i,j}(\mu):=\lim_{r\to 1}{\bf B}_{{\bf f},r\mu}[\widetilde
 \xi_{i,j}]
 $$
  is continuous on ${\bf D_f^m}(\CC)$ and holomorphic on $Int({\bf D_{f}^m}(\CC))$.
Set $\xi(\mu):=(
 \xi_1(\mu),\ldots, \xi_{k'}(\mu))$ and $\xi_i(\mu):=(\xi_{i,1}(\mu), \ldots, \xi_{i,\ell_i}(\mu))$ for all $\mu\in {\bf D_f^m}(\CC)$.
 Since $0\leq \Phi_{g_i,\widetilde \xi_i}(I)\leq I$,    we can show that $0\leq \Phi_{g_i, \xi_i(\mu)}(1)\leq 1$. Hence, we deduce that $\xi(\mu)\in {\bf
D_{g}^d}(\CC)$ for all $\mu\in {\bf D_f^m}(\CC)$. Therefore, the map  $\xi:{\bf D_f^m}(\CC)\to {\bf
D_{g}^d}(\CC) $ is continuous on ${\bf D_f^m}(\CC)$ and holomorphic on $Int({\bf D_{f}^m}(\CC))$.

 Now,   each $\widetilde \xi_{i,j}\in
\cA({\bf D_f^m})$, $i\in \{1,\ldots, k'\}$, $j\in \{1,\ldots, \ell_i\}$,  has a unique Fourier
representation $ \sum_{{(\alpha)\in \FF_{n_1}^+\times \cdots \times\FF_{n_k}^+ }}   a_{(\alpha)} {\bf W}_{(\alpha)}^{({\bf f})}$
such that
 $$
 \widetilde \xi_{i,j}=\lim_{r\to 1} \sum_{q=0}^\infty \sum_{{(\alpha)\in \FF_{n_1}^+\times \cdots \times\FF_{n_k}^+ }\atop {|\alpha_1|+\cdots +|\alpha_k|=q}} r^q  a_{(\alpha)} {\bf W}_{(\alpha)}^{({\bf f})},
$$
 where the
limit is in the operator norm topology. Hence, using the
continuity of $\widehat \Psi$ in the operator norm, and relations
\eqref{xi} and \eqref{var-tilde}, we obtain
\begin{equation*}
\begin{split}
{\bf W}_{i,j}^{({\bf g})}&=\widehat \Psi(\widetilde \xi_{i,j})=\widehat
\Psi\left(\lim_{r\to 1} \sum_{q=0}^\infty \sum_{{(\alpha)\in \FF_{n_1}^+\times \cdots \times\FF_{n_k}^+ }\atop {|\alpha_1|+\cdots +|\alpha_k|=q}} r^q  a_{(\alpha)} {\bf W}_{(\alpha)}^{({\bf f})}
\right)\\
&=\lim_{r\to 1} \sum_{q=0}^\infty \sum_{{(\alpha)\in \FF_{n_1}^+\times \cdots \times\FF_{n_k}^+ }\atop {|\alpha_1|+\cdots +|\alpha_k|=q}} r^q  a_{(\alpha)}\widehat\Psi( {\bf W}_{(\alpha)}^{({\bf f})})= \lim_{r\to 1}
\sum_{q=0}^\infty \sum_{{(\alpha)\in \FF_{n_1}^+\times \cdots \times\FF_{n_k}^+ }\atop {|\alpha_1|+\cdots +|\alpha_k|=q}} r^q  a_{(\alpha)}  \widetilde
\varphi_{(\alpha)}.
\end{split}
\end{equation*}
Consequently, using the  continuity in the operator norm  of  the
noncommutative Berezin transform  at $\lambda\in {\bf D_{g,>}^d}(\CC)$ on the
polydomain algebra $\cA({\bf D_g^d})$, and  relations $\varphi_{i,j}(\lambda):={\bf B}_{{\bf g},\lambda}[\widetilde
 \varphi_{i,j}]$ for all $\lambda\in {\bf D_{g,>}^d}(\CC)$,  and $\xi_{i,j}(\mu):=\lim_{r\to 1}{\bf B}_{{\bf f},r\mu}[\widetilde
 \xi_{i,j}]$ for $\mu\in {\bf D_f^m}(\CC)$,  we have
\begin{equation*}
\begin{split}
\lambda_{i,j}&={\bf B}_{{\bf g},\lambda}[{\bf W}_{i,j}^{({\bf g})}]={\bf B}_{{\bf g},\lambda}\left[\lim_{r\to 1}\sum_{q=0}^\infty \sum_{{(\alpha)\in \FF_{n_1}^+\times \cdots \times\FF_{n_k}^+ }\atop {|\alpha_1|+\cdots +|\alpha_k|=q}} r^q  a_{(\alpha)}  \widetilde
\varphi_{(\alpha)}\right]\\
&=\lim_{r\to 1}\sum_{q=0}^\infty \sum_{{(\alpha)\in \FF_{n_1}^+\times \cdots \times\FF_{n_k}^+ }\atop {|\alpha_1|+\cdots +|\alpha_k|=q}} r^q  a_{(\alpha)}
{\bf B}_{{\bf g},\lambda}[\widetilde\varphi_{(\alpha)}]
=\lim_{r\to 1} \sum_{q=0}^\infty \sum_{{(\alpha)\in \FF_{n_1}^+\times \cdots \times\FF_{n_k}^+ }\atop {|\alpha_1|+\cdots +|\alpha_k|=q}} r^q  a_{(\alpha)}
\varphi_{(\alpha)}(\lambda)\\
&=\lim_{r\to 1} {\bf B}_{{\bf f},r\varphi(\lambda)} [\widetilde\xi_{i,j}]=\xi_{i,j}(\varphi(\lambda))
\end{split}
\end{equation*}
for each $i\in \{1,\ldots, k'\}$, $j\in \{1,\ldots, \ell_i\}$, and any $\lambda\in {\bf D_{g,>}^d}(\CC)$. Hence
$(\xi\circ \varphi)(\lambda)=\lambda$ for all  $\lambda\in {\bf D_{g,>}^d}(\CC)$. Now, using the fact that the functions  $\varphi:{\bf D_g^d}(\CC)\to {\bf
D_{f}^m}(\CC) $  and   $\xi:{\bf D_f^m}(\CC)\to {\bf
D_{g}^d}(\CC) $ are  continuous, and ${\bf D_{g,>}^d}(\CC)$ is dense in ${\bf D_g^d}(\CC)$, we conclude that $(\xi\circ \varphi)(\lambda)=\lambda$ for all  $\lambda\in {\bf D_{g}^d}(\CC)$.
Similarly, one can prove that $(\varphi\circ \xi)(\mu)=\mu$ for  $\mu\in
{\bf D_f^m}(\CC)$. Therefore, the map
 $\varphi:{\bf D_g^d}(\CC)\to {\bf D_f^m}(\CC)$ is a homeomorphism  such that
  $\varphi$ and $\varphi^{-1}:=\xi$ are   holomorphic functions on   $Int({\bf D_{g}^d}(\CC))$ and  $Int({\bf
D_{f}^m}(\CC))$, respectively.   Now, a standard argument
using Brouwer's invariance of domain theorem \cite{Bo}  shows that $\varphi$ is a  biholomorphic function from  $Int({\bf D_{g}^d}(\CC))$ onto  $Int({\bf
D_{f}^m}(\CC))$ and $n=\ell$.
The proof is complete.
\end{proof}

\begin{corollary}\label{Sunada} Let ${\bf f}=(f_1,\ldots, f_k)$ and ${\bf g}=(g_1,\ldots, g_{k'})$ be  tuples of positive regular free
holomorphic functions with $n$ and $\ell$ indeterminates, respectively, and let ${\bf m}\in \NN^k$ and ${\bf d}\in \NN^{k'}$.  If   the domain algebras $\cA({\bf D_f^m})$ and
 $\cA({\bf D_g^d})$  are   unital completely contractive
  isomorphic, then    $n=\ell$
    and    there exists a permutation $\sigma$ of the set $\{1,\ldots, n\}$ and
scalars $t_1,\ldots, t_n >0$ such that the map
$$
Int({\bf
D_{f}^m}(\CC))\ni(z_1,\ldots, z_n)\mapsto (t_1z_{\sigma(1)},\ldots, t_n z_{\sigma(n)})\in Int({\bf D_{g}^d}(\CC))
$$
is a biholomorphic map.
\end{corollary}
\begin{proof}
 Note that the sets  $Int({\bf D_{f}^m}(\CC))\subset \CC^n$ and
$Int({\bf D_{g}^d}(\CC))\subset \CC^\ell$ are Reinhardt domains which
contain $0$. Due to Theorem \ref{classification}, there is a  biholomorphic function from  $Int({\bf D_{g}^d}(\CC))$ onto  $Int({\bf
D_{f}^m}(\CC))$ and  $n=\ell$. Using  Sunada's result \cite{Su}, we complete the proof.
\end{proof}

\begin{proposition}\label{continuity}  Let $\cQ\subset \CC[Z_{i,j}]$
  be  a left ideal generated by  noncommutative
    homogenous polynomials and   let $\cA({\bf {\cV}^m_{f,{\cQ}}})$ be   the corresponding  noncommutative
  variety
algebra. If $\varphi\in \cA({\bf {\cV}^m_{f,{\cQ}}})$, then the map
$\check \varphi :{\bf {\cV}^m_{f,{\cQ}}}(\cH) \to B(\cH)$ defined by
$$
\check \varphi(Y):=\lim_{r\to 1} {\bf B}_{rY,\cQ}[\varphi],\qquad  Y\in {\bf {\cV}^m_{f,{\cQ}}}(\cH),
$$
is continuous,
where the convergence is in the  operator norm topology and ${\bf B}_{{\bf f},rY,\cQ}$ is the constrained noncommutative Berezin tranform.
 \end{proposition}
 \begin{proof} First, note that the map $\check \varphi$ is well-defined due to
 Proposition \ref{vN2-variety}. Let $p_n({\bf S}_{i,j})$ be a sequence of polynomials in the variety algebra $\cA({\bf {\cV}^m_{f,{\cQ}}})$ such that $p_n({\bf S}_{i,j})\to \varphi$ in the operator norm. Given $\epsilon >0$, let $N\in \NN$ be such that $\|\varphi- p_N({\bf S}_{i,j})\|<\frac{\epsilon}{4}$. Fix $A\in {\bf {\cV}^m_{f,{\cQ}}}(\cH)$ and and choose $\delta>0$ such that $\|p_N(Y)-p_N(A)\|<\frac{\epsilon}{2}$, whenever $Y\in  {\bf {\cV}^m_{f,{\cQ}}}(\cH)$  and  $\|Y-A\|<\delta$. Now, using again  Proposition \ref{vN2-variety}, we have
 \begin{equation*}
 \begin{split}
 \|\check \varphi(Y)-\check \varphi(A)\|
 &\leq \limsup_{r\to\infty}
  \| {\bf B}_{r{\bf Y},\cQ}[\varphi]- {\bf B}_{r{\bf A},\cQ}[\varphi]\|\\
  &=\limsup_{r\to\infty}\left\{ {\bf B}_{r{\bf Y},\cQ}[\varphi-p_N({\bf S}_{i,j})]\|+\|{\bf B}_{r{\bf Y},\cQ}[p_N({\bf S}_{i,j})]-{\bf B}_{r{\bf A},\cQ}[p_N({\bf S}_{i,j})]\|\right.\\
  &\qquad \qquad \qquad \left.+ \|{\bf B}_{r{\bf A},\cQ}[p_N({\bf S}_{i,j})-\varphi]\|\right\}\\
  &\leq 2\|\varphi- p_N({\bf S}_{i,j})\|+\limsup_{r\to 1} \|p_N(rY)-p_N(rA)\|\\
  &\leq 2\|\varphi- p_N({\bf S}_{i,j})\|+\|p_N(Y)-p_N(A)\|\leq \epsilon
\end{split}
 \end{equation*}
 for any  $Y\in  {\bf {\cV}^m_{f,{\cQ}}}(\cH)$  with  $\|Y-A\|<\delta$.
The proof is complete
 \end{proof}

Consider the particular case when $\cQ=\cQ_c$.  According to Theorem \ref{multipliers}, the Hardy algebra $F^\infty(\cV_{\bf f, \cQ_c}^m)$  coincides with the
algebra  $H^\infty({\bf D^m_{f,>}}(\CC))$  of all
multipliers of the Hilbert space $H^2({\bf D_{f,>}^m}(\CC))$.
We remark that each $\varphi\in \cA({\bf {\cV}^m_{f,{\cQ_c}}})$ can be identified with a multiplier $\xi$ of $H^2({\bf D_{f,>}^m}(\CC))$ which admits a continuous extension to ${\bf D_f^m}(\CC)$. Moreover,
$$\xi(\lambda)=\lim_{r\to 1} {\bf B}_{r\lambda,\cQ_c}[\varphi],\qquad \lambda\in {\bf D}_{{\bf f},>}^{\bf m}(\CC).
$$
Indeed, due to Theorem \ref{multipliers}, $\varphi$ can be identified with a multiplier $\xi$  which is given by the  relation
$\xi(\lambda)=\left<\varphi(1), u_\lambda\right>$ for all $\lambda\in {\bf D_{f,>}^m}(\CC)$. On the other hand, due to Proposition \ref{continuity},
the map $\check \varphi :{\bf {\cV}^m_{f,{\cQ}}}(\CC) \to \CC$ defined by $\check \varphi(\lambda):=\lim_{r\to 1} {\bf B}_{r\lambda,\cQ}[\varphi]$
 is continuous on
 ${\bf {\cV}^m_{f,{\cQ}}}(\CC)={\bf D}_{{\bf f}}^{\bf m}(\CC)$. According to Theorem \ref{w*} and the remarks that follow, we deduce that
 $\xi(\lambda)=\left<\varphi(1), u_\lambda\right>=\check \varphi(\lambda)$ for all $\lambda\in {\bf D_{f,>}^m}(\CC)$, which proves our assertion.

 \begin{theorem}
 \label{classification2}  Let ${\bf f}=(f_1,\ldots, f_k)$ and ${\bf g}=(g_1,\ldots, g_{k'})$ be  tuples of positive regular free
holomorphic functions with $n$ and $\ell$ indeterminates, respectively,  let ${\bf m}:=(m_1,\ldots, m_k)\in \NN^k$ and ${\bf d}:=(d_1,\ldots, d_{k'})\in \NN^{k'}$, and let $\cQ $
  be  a left ideal generated by
    homogenous polynomials in $ \CC[Z_{i,j}]$.
If
    $\widehat\Psi:\cA({\bf {\cV}^m_{f,{\cQ}}})
\to \cA({\bf {\cV}^d_{g,{\cQ}}})$ is a unital completely contractive
  isomorphism, then the map $\varphi: {\bf {\cV}^d_{g,{\cQ}}}(\CC)\to {\bf {\cV}^m_{f,{\cQ}}}(\CC)$ defined by
$$
\varphi(\lambda):=\left[ \lim_{r\to 1}{\bf B}_{{\bf g},r\lambda, \cQ}[\widehat{\Psi}({\bf S}^{({\bf f})}_{i,j}]: \ {i\in \{1,\ldots, k\}, j\in \{1,\ldots, n_i\}}\right],\qquad \lambda \in {\bf {\cV}^d_{g,{\cQ}}}(\CC),
$$
where  ${\bf S}^{(\bf f)}:=\{ {\bf S}^{({\bf f})}_{i,j}\}$ is the universal model of the abstract variety ${\bf {\cV}^m_{f,{\cQ}}}$ and ${\bf B}_{{\bf g}, r\lambda, \cQ}$ is the constrained  Berezin  transform at $r\lambda$,
is a homeomorphism  of ${\bf {\cV}^d_{g,{\cQ}}}(\CC)$  onto ${\bf {\cV}^m_{f,{\cQ}}}(\CC)$.

In the particular case when $\cQ=\cQ_c$, the map $\varphi$ is, in addition,
 a biholomorphic function from $Int({\bf {\cV}^d_{g,{\cQ_c}}}(\CC))$ onto $Int({\bf {\cV}^m_{f,{\cQ_c}}}(\CC))$ and  $n=\ell$.
\end{theorem}
\begin{proof}   We only sketch the proof,  since it is very similar to that  of Theorem \ref{classification}, and point out the differences.
Denote
 \begin{equation}
 \label{var-tilde2}
 \widetilde\varphi_{i,j}:=\widehat \Psi({\bf S}_{i,j}^{({\bf f})})\in\cA({\bf {\cV}^d_{g,{\cQ}}}),
\qquad i\in \{1,\ldots, k\}, j\in \{1,\ldots, n_i\},
 \end{equation}
 where  ${\bf S}^{(\bf f)}:=\{ {\bf S}^{({\bf f})}_{i,j}\}$ is the universal model of the abstract variety  ${\bf {\cV}^m_{f,{\cQ}}}$.   Due to Proposition \ref{continuity}, the map $\varphi_{i,j}: {\bf {\cV}^d_{g,{\cQ}}}(\CC)\to \CC$, given
 by
 $$\varphi_{i,j}(\lambda):=\lim_{r\to 1}{\bf B}_{{\bf g},r\lambda,\cQ}[\widetilde
 \varphi_{i,j}]
 $$
  is well-defined and continuous. Consider  the function $\varphi:{\bf {\cV}^d_{g,{\cQ}}}(\CC)\to\CC^\ell$  given by
   $\varphi(\lambda):=(
 \varphi_1(\lambda),\ldots, \varphi_{k}(\lambda))$, where $\varphi_i(\lambda):=(\varphi_{i,1}(\lambda), \ldots, \varphi_{i,n_i}(\lambda))$ for all  $\lambda\in {\bf {\cV}^d_{g,{\cQ}}}(\CC)$ and note that   $\varphi(\lambda)\in {\bf
D_{f}^m}(\CC)$ for all ${\bf {\cV}^d_{g,{\cQ}}}(\CC)$. On the other hand, since
$q({\bf S}^{(\bf f)})=0$ for any $q\in \cQ$,  and $\widehat\Psi$ is a homomorphism, one can deduce that $q(\widetilde \varphi)=0$. Applying the constrained Berezin transform ${\bf B}_{{\bf g},r\lambda,\cQ}$ and taking the limit as $r\to 1$, we obtain that $q(\varphi(\lambda))=0$ for any $q\in \cQ$. Therefore $\varphi(\lambda)\in {\bf {\cV}^m_{f,{\cQ}}}(\CC)$ and
 the map  $\varphi: {\bf {\cV}^d_{g,{\cQ}}}(\CC)\to {\bf {\cV}^m_{f,{\cQ}}}(\CC)$  is continuous.
Similarly, setting
\begin{equation}
 \label{xi2}
 \widetilde\xi_{i,j}:=\widehat \Psi^{-1}({\bf S}_{i,j}^{({\bf g})})\in \cA({\bf {\cV}^m_{f,{\cQ}}}),
\qquad i\in \{1,\ldots, k'\}, j\in \{1,\ldots, \ell_i\},
 \end{equation}
 where  ${\bf S}^{(\bf g)}:=\{ {\bf S}^{({\bf g})}_{i,j}\}$ is the universal model of the abstract variety ${\bf {\cV}^d_{g,{\cQ}}}$, Proposition \ref{continuity} shows  that
   the map $\xi_{i,j}: {\bf{\cV}^m_{f,{\cQ}}}(\CC)\to \CC$ given
 by $\xi_{i,j}(\mu):=\lim_{r\to 1}{\bf B}_{{\bf f},r\mu, \cQ}[\widetilde
 \xi_{i,j}]$ is well-defined and continuous. Now, one can prove that
   the map  $\xi:{\bf {\cV}^m_{f,{\cQ}}}(\CC)\to {\bf {\cV}^d_{g,{\cQ}}}(\CC)$  defined by
$\xi(\mu):=(
 \xi_1(\mu),\ldots, \xi_{k'}(\mu))$,  where  $\xi_i(\mu):=(\xi_{i,1}(\mu), \ldots, \xi_{i,\ell_i}(\mu))$,
is continuous.

 For each $\widetilde \xi_{i,j}\in
\cA({\bf {\cV}^m_{f,{\cQ}}})$, $i\in \{1,\ldots, k'\}$, $j\in \{1,\ldots, \ell_i\}$,  let $p_s({\bf S}^{({\bf f})})=
  \sum_{{(\alpha)\in \FF_{n_1}^+\times \cdots \times\FF_{n_k}^+ }}   a_{(\alpha)}^{(s)} {\bf S}_{(\alpha)}^{({\bf f})}$, $s\in \NN$, be a sequence of polynomials
such that
 $
 \widetilde \xi_{i,j}=\lim_{s\to \infty} p_s({\bf S}^{({\bf f})})
$
 where the
convergence  is in the operator norm. Using the
continuity of $\widehat \Psi$ in the operator norm, and relations
\eqref{xi2} and \eqref{var-tilde2}, we obtain
\begin{equation*}
\begin{split}
{\bf S}_{i,j}^{({\bf g})}&=\widehat \Psi(\widetilde \xi_{i,j})=\widehat
\Psi\left(\lim_{s\to \infty} p_s({\bf S}^{({\bf f})})
\right) =\lim_{s\to \infty} p_s({\widetilde \varphi} ).
\end{split}
\end{equation*}
Consequently, using the  continuity in the operator norm  of  the constrained
noncommutative Berezin transform  at $\lambda\in {\bf {\cV}^d_{g,{\cQ},>}}(\CC)$ on the
 variety algebra $\cA({\bf {\cV}^d_{g,{\cQ}}})$ and   the relations above, we obtain
\begin{equation*}
\begin{split}
\lambda_{i,j}&={\bf B}_{{\bf g},\lambda, \cQ}[{\bf S}_{i,j}^{({\bf g})}]={\bf B}_{{\bf g},\lambda, \cQ}\left[\lim_{s\to \infty} p_s({\widetilde \varphi} )\right]\\
&
= \lim_{s\to\infty}  p_s({  \varphi(\lambda)} )
=\lim_{s\to\infty} {\bf B}_{{\bf f},\varphi(\lambda),\cQ} [p_s({\bf S}^{({\bf f})}) ]\\
&=\xi_{i,j}(\varphi(\lambda))
\end{split}
\end{equation*}
for each $i\in \{1,\ldots, k'\}$, $j\in \{1,\ldots, \ell_i\}$, and any $\lambda\in {\bf {\cV}^d_{g,{\cQ},>}}(\CC)$. Hence
$(\xi\circ \varphi)(\lambda)=\lambda$ for all $\lambda\in {\bf {\cV}^d_{g,{\cQ},>}}(\CC)$. Now, using the fact that the functions  $\varphi: {\bf {\cV}^d_{g,{\cQ}}}(\CC)\to {\bf {\cV}^m_{f,{\cQ}}}(\CC)$  and   $\xi:{\bf {\cV}^m_{f,{\cQ}}}(\CC)\to {\bf {\cV}^d_{g,{\cQ}}}(\CC)$ are  continuous, and ${\bf {\cV}^d_{g,{\cQ},>}}(\CC)$ is dense in ${\bf {\cV}^d_{g,{\cQ}}}(\CC)$, we conclude that $(\xi\circ \varphi)(\lambda)=\lambda$ for all  $\lambda\in {\bf {\cV}^d_{g,{\cQ}}}(\CC)$.
Similarly, one can prove that $(\varphi\circ \xi)(\mu)=\mu$ for  $\mu\in
{\bf {\cV}^m_{f,{\cQ}}}(\CC)$. Therefore, the map
 $\varphi$ is a homeomorphism.
 Note that in the particular case when $\cQ=\cQ_c$,  we have ${\bf {\cV}^m_{f,{\cQ_c}}}(\CC)={\bf D_{f}^m}(\CC)$ and ${\cV}^d_{g,{\cQ_c}}(\CC)={\bf D_{g}^d}(\CC)$. Using Theorem \ref{classification}, one can complete the proof.
\end{proof}

We remark that a result similar to Corollary \ref{Sunada} holds in the commutative setting. Therefore, if     the variety algebras $\cA({\bf {\cV}^m_{f,{\cQ_c}}})$ and $\cA({\bf {\cV}^d_{g,{\cQ_c}}})$  are   unital completely contractive
  isomorphic, then    $n=\ell$
    and    there exists a permutation $\sigma$ of the set $\{1,\ldots, n\}$ and
scalars $t_1,\ldots, t_n >0$ such that the map
$$
Int({\bf {\cV}^m_{f,{\cQ_c}}}(\CC))\ni(z_1,\ldots, z_n)\mapsto (t_1z_{\sigma(1)},\ldots, t_n z_{\sigma(n)})\in Int({\bf {\cV}^d_{g,{\cQ_c}}}(\CC))
$$
is a biholomorphic map.

The results of this section show    that there are many non-isomorphic polydomain algebras. We consider the following particular case.
If  ${\bf f}=Z_1+\cdots +Z_n$, then $\cA(\cV_{{\bf f},\cQ_c}^1)$ is the universal algebra of commuting row contractions, and
 $ Int(\cV_{f,\cQ_c}^1(\CC)=\BB_n$, the open unit ball of $\CC^n$. When ${\bf g}=(Z_1,\ldots, Z_n)$, then $\cA(\cV_{{\bf g},\cQ_c}^1)$ is the commutative polydisc algeba.
 In this case, we have
 $ Int(\cV_{f,\cQ_c}^1(\CC)=\DD^n$. Since $\BB_n$ and $\DD^n$ are not biholomorphic domains in $\CC^n$  if $n\geq 2$, Theorem \ref{classification2} shows that the universal  algebras  $\cA(\cV_{{\bf f},\cQ_c}^1)$ and  $\cA(\cV_{{\bf g},\cQ_c}^1)$ are not isomorphic.

\bigskip

\section{Dilation theory on noncommutative varieties in polydomains}

In this section we develop a dilation theory  on abstract  noncommutative varieties $\cV_{\bf f, {\it J}}^{\bf m}$, where $J$ is a norm-closed two sided ideal of   the noncommutative polydomain algebra $\cA({\bf D_f^m})$ such that $\cN_J\neq \{0\}$.
The dilation theory can be  refined for the class of noncommutative varieties   $\cV_{\bf q, \cQ}^{\bf m} $, where $\cQ\subset \CC[Z_{i,j}]$ is an ideal generated by homogeneous  polynomials and ${\bf q}=(q_1,\ldots, q_k)$ is a $k$-tuple of   positive regular  noncommutative polynomials. In this case, we also obtain  Wold type
decompositions for non-degenerate $*$-representations of the
$C^*$-algebra $C^*({\bf S}_{i,j})$ generated by the universal model.

\begin{lemma}\label{Psi-X}
Let ${\bf T}=({ T}_1,\ldots, { T}_k)$ be in the noncommutative polydomain ${\bf D_f^m}(\cH)$ and let $X\in B(\cH)$ be a positive operator such that
   $
  {\bf \Delta_{f,T}^p}(X)\geq 0
    $
 for any  ${\bf p}:=(p_1,\ldots, p_k)\in \ZZ_+^k$ with ${\bf p}\leq {\bf m}$. Then
       \begin{equation*}
0\leq
\lim_{q_k\to\infty}\ldots \lim_{q_1\to\infty}  (id-\Phi_{f_k,T_k}^{q_k})\cdots (id-\Phi_{f_1,T_1}^{q_1})(X)\leq X.
\end{equation*}
\end{lemma}
\begin{proof}
For each $i\in \{1,\ldots,k\}$, let $\Omega_i\subset B(\cH)$ be the  set of all $Y\in B(\cH)$, $Y\geq0$, such that
the series $\sum_{\beta_i\in \FF_{n_i}^+} b_{i,\beta_i}^{(m_i)} T_{i,\beta_i}Y T_{i,\beta_i}^*$ is convergent in the weak operator topology,
where
\begin{equation*}
  b_{i,g_0}^{(m_i)}:=1\quad \text{ and } \quad
 b_{i,\alpha}^{(m_i)}:= \sum_{p=1}^{|\alpha|}
\sum_{{\gamma_1,\ldots,\gamma_p\in \FF_{n_i}^+}\atop{{\gamma_1\cdots \gamma_p=\alpha }\atop {|\gamma_1|\geq
1,\ldots, |\gamma_p|\geq 1}}} a_{i,\gamma_1}\cdots a_{i,\gamma_p}
\left(\begin{matrix} p+m_i-1\\m_i-1
\end{matrix}\right)
\end{equation*}
for all  $ \alpha\in \FF_{n_i}^+$ with  $|\alpha|\geq 1$.
We define the map $\Psi_i:\Omega_i\to B(\cH)$ by setting
\begin{equation*}
\label{Psi-i}
\Psi_i(Y):=\sum_{\beta_i\in \FF_{n_i}^+} b_{i,\beta_i}^{(m_i)} T_{i,\beta_i}Y T_{i,\beta_i}^*.
\end{equation*}
Fix $i\in \{1,\ldots,k\}$ and assume that $1\leq p_i= m_i$.
In \cite{Po-Berezin-poly}  (see the proof of Theorem 2.2), we proved that
\begin{equation}
\label{Psi}
\begin{split}
0\leq \Psi_i( {\bf \Delta_{f,T}^p}(X))&={\bf \Delta_{f,T}^{ p-\text{$m_i$}e_i}}
\left(id-\lim_{q_{i}\to \infty} \Phi_{f_i,T_i}^{q_{i}} \right)(X)\\
&
\leq {\bf \Delta_{f,T}^{ p-\text{$m_i$}e_i}}(X)\leq X,
\end{split}
\end{equation}
for any  ${\bf p}:=(p_1,\ldots, p_k)\in \ZZ_+^k$ with ${\bf p}\leq {\bf m}$ and
 $p_i=m_i$.
A repeated application of \eqref{Psi}, leads to the relation
\begin{equation*}
0\leq (\Psi_k\circ\cdots \circ \Psi_1)({\bf \Delta_{f,T}^m}(X))=
\lim_{q_k\to\infty}\ldots \lim_{q_1\to\infty}  (id-\Phi_{f_k,T_k}^{q_k})\cdots (id-\Phi_{f_1,T_1}^{q_1})(X)\leq X.
\end{equation*}
The proof is complete.
\end{proof}

\begin{lemma} \label{ine-ber}
Let ${\bf T}=({ T}_1,\ldots, { T}_k)$ be in the noncommutative polydomain ${\bf D_f^m}(\cH)$ and let ${\bf K_{f,T}}$ be the associated Berezin kernel. Then
$$
{\bf \Delta_{f,T}^p}({\bf K_{f,T}^*}{\bf K_{f,T}})\leq {\bf \Delta_{f,T}^p}(I)
$$
for any  ${\bf p}:=(p_1,\ldots, p_k)\in \ZZ_+^k$ with ${\bf p}\leq {\bf m}$.  The equality occurs if \  ${\bf p}\geq (1,\ldots, 1)$.

\end{lemma}
\begin{proof}
 Let $s\in \{1,\ldots, k\}$ and let  $Y\geq 0$ be such that  $(id-\Phi_{f_s,T_s})\cdots (id-\Phi_{f_1,T_1})(Y)\geq 0$.
 Note that
  $\{(id-\Phi_{f_s,T_k}^{q_s})\cdots (id-\Phi_{f_1,T_1}^{q_1})(Y)\}_{{\bf q}=(q_1,\ldots, q_s)\in \ZZ_+^s}$ is an  increasing sequence of positive operators.
 Indeed,  since  $\Phi_{f_1, T_1}, \ldots, \Phi_{f_k,T_k}$ are commuting, we have
 $$
 0\leq (id-\Phi_{f_s,T_s}^{q_s})\cdots (id-\Phi_{f_1,T_1}^{q_1})(Y)
 =\sum_{t=0}^{q_s-1}\Phi_{f_s, T_s}^t\cdots \sum_{t=0}^{q_1-1}\Phi_{f_1,T_1}^t
(id-\Phi_{f_s,T_s})\cdots (id-\Phi_{f_1,T_1})(Y),
$$
which proves our assertion.
If ${\bf p}=0$, the inequality  in the lemma is due to the fact that ${\bf K_{f,T}^*}{\bf K_{f,T}}\leq I$. Assume that ${\bf p}\neq 0$. Without loss of generality, we can assume that
 $p_j\geq 1$ for any $j\in \{1,\ldots,s\}$ for some $s\in \{1,\ldots, k\}$, and $p_j=0$ for any $j\in \{s+1,\ldots, k\}$ if $s<k$.
 Since
 $$
{\bf K_{f,T}^*}{\bf K_{f,T}}= \lim_{q_k\to\infty}\ldots \lim_{q_1\to\infty}  (id-\Phi_{f_k,T_k}^{q_k})\cdots (id-\Phi_{f_1,T_1}^{q_1})(I),
$$
and taking into account that  the maps $\Phi_{f_i,T_i}$ are WOT-continuous  and commuting, we deduce that
\begin{equation*}
\begin{split}
&(id-\Phi_{f_1,T_1})^{p_1}\cdots (id-\Phi_{f_s,T_s})^{p_s}({\bf K_{f,T}^*}{\bf K_{f,T}})\\
&=\lim_{{\bf q}} (id-\Phi_{f_k,T_k}^{q_k})\cdots (id-\Phi_{f_{s+1},T_{s+1}}^{q_{s+1}}) \left[(id-\Phi_{f_s,T_s})^{p_s}(id-\Phi_{f_s,T_s}^{q_s})\right]\cdots \left[(id-\Phi_{f_1,T_1})^{p_1}(id-\Phi_{f_1,T_1}^{q_1})\right](I)
\end{split}
\end{equation*}
Now, let $j\in\{1,\ldots,s\}$ and let $Y\geq 0$ be such that $(id-\Phi_{f_j,T_j})(Y)\geq 0$. Due to the remark at the beginning of the proof,
 WOT-$\lim_{q_j\to \infty} (id-\Phi_{f_j,T_j}^{q_j})(Y)$ exists and, since $\Phi_{f_i,T_i}$ is WOT-continuous,
 we have
\begin{equation*}
\begin{split}
\lim_{q_j\to \infty} (id-\Phi_{f_j,T_j})^{p_j}(id-\Phi_{f_j,T_j}^{q_j})(Y)&=
(id-\Phi_{f_j,T_j})^{p_j-1}\lim_{q_j\to\infty} (id-\Phi_{f_j,T_j})(id-\Phi_{f_j,T_j}^{q_j})(Y)\\
&=(id-\Phi_{f_j,T_j})^{p_j}(Y).
\end{split}
\end{equation*}
Applying this result repeatedly,  when $j=1$ and $Y=I$,  when $j=2$ and $Y= (id-\Phi_{f_1,T_1})^{p_1}(I)$, and so on,   when  $j=s$ and $Y=(id-\Phi_{f_1,T_1})^{p_1}\cdots (id-\Phi_{f_s,T_s})^{p_s}(I)$, we obtain
\begin{equation*}
\begin{split}
\lim_{q_s\to\infty}\cdots \lim_{q_1\to\infty}\left[(id-\Phi_{f_s,T_s})^{p_s}
(id-\Phi_{f_s,T_s}^{q_s})\right]\cdots& \left[(id-\Phi_{f_1,T_1})^{p_1}(id-\Phi_{f_1,T_1}^{q_1})\right](I)\\
&
=(id-\Phi_{f_1,T_1})^{p_1}\cdots (id-\Phi_{f_s,T_s})^{p_s}(I)
\end{split}
\end{equation*}
Summing up the results above  and using Lemma \ref{Psi-X}, we deduce that
\begin{equation*}
\begin{split}
&(id-\Phi_{f_1,T_1})^{p_1}\cdots (id-\Phi_{f_s,T_s})^{p_s}({\bf K_{f,T}^*}{\bf K_{f,T}})\\
&\qquad =\lim_{ (q_{s+1},\ldots, q_k)} (id-\Phi_{f_k,T_k}^{q_k})\cdots (id-\Phi_{f_{s+1},T_{s+1}}^{q_{s+1}}) (id-\Phi_{f_1,T_1})^{p_1}\cdots (id-\Phi_{f_s,T_s})^{p_s}(I)\\
&\qquad \leq (id-\Phi_{f_1,T_1})^{p_1}\cdots (id-\Phi_{f_s,T_s})^{p_s}(I).
\end{split}
\end{equation*}
The last part of this lemma is now obvious.
The proof is complete.
\end{proof}

Let   ${\bf f}=(f_1,\ldots, f_k)$ be  a $k$-tuple of   positive regular  free holomorphic functions and let ${\bf S}=({\bf S}_1,\ldots,
{\bf S}_n)$ with ${\bf S}_i=({\bf S}_{i,1},\ldots, {\bf S}_{i,n_i})$   be the universal model
associated with the abstract noncommutative variety $\cV_{\bf f, {\it J}}^{\bf m}$, where $J$ is a norm-closed two sided ideal of   the noncommutative domain algebra $\cA({\bf D_f^m})$ such that $\cN_J\neq \{0\}$.
Let ${\bf U}=\{U_{i,j}\} \in \cV_{\bf f, {\it J}}^{\bf m}(\cK)$ be  such that
$$ (id-\Phi_{f_k,U_k})\cdots (id-\Phi_{f_1,U_1})(I)=0,
$$
where $U_i=(U_{i,1},\ldots, U_{i,n_i})$.
A tuple ${\bf V}:=\{V_{i,j}\} $   having the matrix
representation
\begin{equation}\label{Vi}
V_{i,j}:=\left[\begin{matrix}
{\bf S}_{i,j}\otimes I_\cD&0\\
0&U_{i,j}
\end{matrix}
\right], \qquad i\in \{1,\ldots, k\}, j\in\{1,\ldots, n_i\},
\end{equation}
   is called
{\it constrained} (or $J$-{\it constrained}) {\it dilation} of $T=\{T_{i,j}\} \in
\cV_{\bf f, {\it J}}^{\bf m}(\cH)$ if   $\cH$ can be identified with a co-invariant subspace
under each operator
  $V_{i,j}$   such that
$$
T_{(\alpha)}^* = V_{(\alpha)}^*|\cH,\qquad  (\alpha) \in \FF_{n_1}^+\times \cdots \times \FF_{n_k}^+.
$$
The dilation is called {\it minimal} if
$$
(\cN_J\otimes \cD)\oplus \cK=\overline{\text{\rm span}}\left\{
V_{(\alpha)} \cH: \  (\alpha) \in \FF_{n_1}^+\times \cdots \times \FF_{n_k}^+\right\}.
$$
The {\it dilation index} of $T$  is the minimum dimension of $\cD$ such that $V$ is a
constrained   dilation of $T$.

Our first dilation result   on  the abstract noncommutative variety
$\cV_{\bf f, {\it J}}^{\bf m}$  is the following.

\begin{theorem}\label{dil1} Let   ${\bf S}=\{{\bf S}_{i,j}\}$   be the universal model
associated with the abstract noncommutative variety $\cV_{\bf f, {\it J}}^{\bf m}$, where $J$ is a norm-closed two sided ideal of the noncommutative polydomain algebra $\cA({\bf D_f^m})$. If ${\bf T}:=\{T_{i,j}\}$   is an element
   in the noncommutative variety
$\cV_{\bf f, {\it J}}^{\bf m}(\cH)$, then there exists a Hilbert space $\cK$  and ${\bf U}=\{U_{i,j}\} \in \cV_{\bf f, {\it J}}^{\bf m}(\cK)$ with
$$ (id-\Phi_{f_k,U_k})\cdots (id-\Phi_{f_1,U_1})(I)=0$$
and such that
 $\cH$ can be identified with a  co-invariant subspace of \
   $\tilde\cK:=[\cN_J\otimes \overline{{\bf \Delta_{f,T}^m}(I)\cH}]\oplus \cK$ under each  operator
$$
V_{i,j}:=\left[\begin{matrix} {\bf S}_{i,j}\otimes
I_{\overline{{\bf \Delta_{f,T}^m}(I)\cH}}&0\\0&U_{i,j}
\end{matrix}\right],\qquad i\in \{1,\ldots, k\}, j\in\{1,\ldots, n_i\},
$$
where ${\bf \Delta_{f,T}^m}(I)  :=(id-\Phi_{f_1,T_1})^{m_1}\cdots (id-\Phi_{f_k,T_k})^{m_k}(I)$,
 and
$$
T_{i,j}^*=V_{i,j}^*|_\cH,\qquad i\in \{1,\ldots, k\}, j\in\{1,\ldots, n_i\}.
$$
 Moreover,  the following statements hold.
   \begin{enumerate}
   \item[(i)] The dilation index of ${\bf T}$ coincides with $\text{\rm rank}\,
 {\bf \Delta_{f,T}^m}(I)$.

   \item[(ii)]${\bf T}$ is  a pure element in $\cV_{\bf f, {\it J}}^{\bf m}(\cH)$ if and only if the dilation ${\bf V}:=\{V_{i,j}\}$ is pure.
       \end{enumerate}
\end{theorem}

\begin{proof}

We recall   that  the constrained noncommutative  Berezin kernel
associated with the ${\bf T}\in {\bf \cV_{f,{\it J}}^m}(\cH)$ is  the
bounded operator  \ ${\bf K_{f,T,{\it J}}}:\cH\to \cN_J\otimes
\overline{{\bf \Delta_{f,T}^m}(I) (\cH)}$ defined by
$${\bf K_{f,T,{\it J}}}:=\left(P_{\cN_J}\otimes I_{\overline{{\bf \Delta_{f,T}^m}(I) (\cH)}}\right){\bf K_{f,T}},
$$
where ${\bf K_{f,T}}$ is the noncommutative Berezin kernel associated with ${\bf T}\in
{\bf D_f^m}(\cH)$.  Taking into account the properties of the Berezin kernel and the fact that
$
\text{\rm range}\,{\bf K_{f,T}}\subseteq \cN_J\otimes
\overline{{\bf \Delta_{f,T}^m}(I)\cH},
$
 we have
\begin{equation}\label{KJ}
{\bf K_{f,T,{\it J} }}T_{(\alpha)}^*=({\bf S}_{(\alpha)}^*\otimes
I_\cH){\bf K_{f,T,{\it J}}},\qquad (\alpha)  \in \FF_{n_1}^+\times \cdots \times \FF_{n_k}^+
\end{equation}
and ${\bf K^*_{f,T,{\it J}}} {\bf K_{f,T,{\it J}}}={\bf K_{f,T }^*}{\bf K_{f,T}}$.
We consider the Hilbert space $ \cK:=\overline{(I-{\bf K_{f,T}^*}{\bf K_{f,T}})\cH}$  and denote   $Y:=I-{\bf K_{f,T}^*}{\bf K_{f,T}}$.
For each $i\in \{1,\ldots, k\}$, $j\in\{1,\ldots, n_i\}$, define the operator $L_{i,j}:\cK\to
\cK$ by setting
\begin{equation*}
L_{i,j} Y^{1/2}h:=Y^{1/2}T_{i,j}^*h,\quad h\in \cH.
\end{equation*}
Note that each  $L_{i,j}$ is  well-defined. Indeed, due to  Lemma \ref{ine-ber}, we have
${\bf \Delta}_{\bf f,T}^{(1,\ldots,1)}({\bf K_{f,T}^*}{\bf K_{f,T}})\leq {\bf \Delta}_{\bf f,T}^{(1,\ldots,1)}(I)$. Hence, we deduce that
 $\Phi_{f_i,T_i}(Y)\leq Y$. Therefore,
\begin{equation*}
\begin{split}
\sum_{\alpha\in \FF_{n_i}^+, |\alpha|\geq 1}
 a_{i,\alpha}\|L_{i,{\tilde \alpha}}Y^{1/2}h\|^2
=\left< \Phi_{f_i,T_i}(Y)h,h\right>\leq \|Y^{1/2} h\|^2
\end{split}
\end{equation*}
for any $h\in \cH$, where $\tilde\alpha$ is the reverse of $\alpha$.  Consequently, we have
$a_{i,g_j^i}\|L_{i,j}Y^{1/2} x\|^2\leq \|Y^{1/2} x\|^2$, for any $x\in
\cN_J\otimes \cH$.  Since $a_{i,g_j^i}\neq 0$ each  $L_{i,j}$ can be uniquely
be extended to a bounded operator (also denoted by $L_{i,j}$) on the
subspace $\cK$. Denoting $U_{i,j}:=L_{i,j}^*$ and  setting ${\bf U}=(U_1,\ldots, U_k)$ with $U_i=(U_{i,1},\ldots, U_{i,n_i})$,  an
approximation argument shows that $\Phi_{f_i,U_i}(I_\cM)\leq I_\cM$  for  $i\in \{1,\ldots, k\}$ and  $j\in\{1,\ldots, n_i\}$.
The definition of $L_{i,j}$  implies
\begin{equation}
\label{int*}
U_{i,j}^*(Y^{1/2} h)=Y^{1/2} {T}_{i,j}^*h,\qquad h\in  \cH,
\end{equation}
 for any $i\in \{1,\ldots, k\}$ and $j\in \{1,\ldots, n_i\}$. Hence, and using again Lemma \ref{ine-ber}, we deduce that
 $$
 Y^{1/2} {\bf \Delta_{f,U}^p}(I_\cK) Y^{1/2}
 ={\bf \Delta}_{{\bf f,T}}^{\bf p}(I-{\bf K_{f,T}^*}{\bf K_{f,T}})\geq 0
 $$
for any ${\bf p}:=(p_1,\ldots, p_k)\in \ZZ_+^k$ such that ${\bf p}\leq {\bf m}$, ${\bf p}\neq 0$, and   $Y^{1/2}(id-\Phi_{f_k,U_k})\cdots (id-\Phi_{f_1,U_1})(I)Y^{1/2}=0$. Since $Y^{1/2}$ is injective on $\cK=\overline{Y\cH}$, we conclude that  ${\bf U}=(U_1,\ldots, U_k) \in \cV_{\bf f, {\it J}}^{\bf m}(\cK)$  and
$$ (id-\Phi_{f_k,U_k})\cdots (id-\Phi_{f_1,U_1})(I)=0.
$$
On the other hand,    relation \eqref{int*} implies
$$Y^{1/2}q({\bf U}) =q({\bf T})Y^{1/2}=0, \qquad q\in \CC[Z_{i,j}].
$$
Using the von Neumann  type inequality for the elements in the abstract noncommutative polydomain ${\bf D_f^m}$ and the fact that the polynomials in ${\bf W}_{i,j}$ and the identity are dense in the noncommutative polydomain algebra $\cA({\bf D_f^m})$, an approximation argument shows that  $Y^*g({\bf U})=0$ for any $g\in J$.
Once again, since $Y^{1/2}$ is injective on $\cK=\overline{Y\cH}$, we have
$g({\bf U})=0$ for any $q\in J$.
Let $V:\cH\to [\cN_J\otimes \cH]\oplus \cK$ be defined by
$$V:=\left[\begin{matrix}
{\bf K_{f,T,{\it J}}}\\ Y
\end{matrix}\right].
$$
Note that
\begin{equation*}\begin{split}
\|Vh\|^2=\|{\bf K_{f,T,{\it J}}}h\|^2+\|(I-{\bf K_{f,T,{\it J}}^*}{\bf K_{f,T,{\it J}}})^{1/2}h\|^2 =\|h\|^2
\end{split}
\end{equation*}
for any $h\in \cH$. Therefore, $V$ is an isometry. Using relations \eqref{KJ} and \eqref{int*},
we obtain
\begin{equation*}
\begin{split}
VT_{i,j}^*&=\left[\begin{matrix} {\bf K_{f,T,{\it J}}}\\ Y
\end{matrix}\right]T_{i,j}^*h={\bf K_{f,T,{\it J}}}T_{i,j}^*h\oplus YT_{i,j}^*h\\
&=({\bf S}_{i,j}^*\otimes I_\cH){\bf K_{f,T{\it J}}}h\oplus U_{i,j}^*Yh\\
&=\left[\begin{matrix} {\bf S}_{i,j}^*\otimes
I_{\overline{{\bf\Delta_{f,T}^m}\cH}}&0\\0&U_{i,j}^*
\end{matrix}\right]Vh.
\end{split}
\end{equation*}
 Identifying   $\cH$  with $V\cH$, we deduce that
 $
T_{i,j}^*=V_{i,j}^*|\cH$ for $i\in \{1,\ldots, k\}$ and $ j\in\{1,\ldots, n_i\}.
$

  Now, we prove the second part of the theorem.  Assume that  ${\bf T}$ has the dilation ${\bf V}$ given by relation \eqref{Vi}. Since ${\bf \Delta_{f,U}^m}(I)=0$ and
 ${\bf \Delta_{f,S}^m}(I)={\bf P}_\CC$, where ${\bf P}_\CC$ is the
 orthogonal projection from $\cN_J$ onto
  $\CC 1\subset \cN_J$,
  we deduce that
 $
 {\bf \Delta_{f,T}^m}(I)= P_\cH \left[  {\bf P}_\CC \otimes I_\cD
\right]|\cH.
 $
Hence, $\rank {\bf \Delta_{f,T}^m}(I)\leq \dim \cD$. The reverse inequality is due to the first part of the theorem.
To prove item (ii), note that if ${\bf T}$ is pure then ${\bf K_{f,T}}$ is an isometry and, consequently, $\cK=\{0\}$. This implies  ${\bf V}={\bf S}$, which is pure. Conversely, if we assume that ${\bf V}$ is pure, we must have
 $$
    \lim_{q=(q_1,\ldots, q_k)\in \NN^k}   (id-\Phi_{f_1,V_1}^{q_1})\cdots (id-\Phi_{f_k,V_k}^{q_k})(I_{\tilde{\cK}})= I_{\tilde{\cK}}.
    $$
Taking into account the matrix representation of each operator  ${ V_{i,j}}$ and the fact that
 $$
    \lim_{q=(q_1,\ldots, q_k)\in \NN^k}   (id-\Phi_{f_1,U_1}^{q_1})\cdots (id-\Phi_{f_k,U_k}^{q_k})(I_{{\cK}})= 0,
    $$
 we deduce that
 $\cK=\{0\}$. This shows that   the noncommutative Berezin kernel ${\bf K_{f,T}}$ is an isometry, which is equivalent to the fact that ${\bf T}$ is  pure.
The proof is complete.
\end{proof}

 In what follows, we provide a   Wold type
decomposition for non-degenerate $*$-representations of the
$C^*$-algebra $C^*({\bf S}_{i,j})$.

\begin{theorem}\label{wold}  Let ${\bf q}=(q_1,\ldots, q_k)$ be a $k$-tuple of   positive regular  noncommutative polynomials
 and let ${\bf S}=({\bf S}_1,\ldots, {\bf S}_k)$ be the universal model
associated with the abstract noncommutative variety   $\cV_{\bf q, {\it J}}^{\bf m}$, where $J$ is a WOT-closed two sided ideal of $F^\infty({\bf D_q^m})$ such that $1\in \cN_J$. If
\ $\pi:C^*({\bf S}_{i,j})\to B(\cK)$ is  a nondegenerate
$*$-representation  of \ $C^*({\bf S}_{i,j})$ on a separable
Hilbert space  $\cK$, then $\pi$ decomposes into a direct sum
$$
\pi=\pi_0\oplus \pi_1 \  \text{ on  } \ \cK=\cK_0\oplus \cK_1,
$$
where $\pi_0$ and  $\pi_1$  are disjoint representations of \
$C^*({\bf S}_{i,j})$ on the Hilbert spaces
\begin{equation*}
\cK_0:=\overline{\text{\rm span}}\left\{\pi({\bf S}_{(\alpha)}) {\bf \Delta_{q,\pi(S)}^m}
(I_{\cK})  \cK:\ (\alpha)\in \FF_{n_1}^+\times \cdots \times \FF_{n_k} \right\}
\end{equation*}
and $\cK_1:=\cK_0^\perp,$
 respectively, where  $\pi({\bf S}):= (\pi({\bf S}_1),\ldots, \pi({S}_k))   $
 and $\pi({\bf S}_{i}):=(\pi({\bf S}_{i,1}),\ldots, \pi({\bf S}_{i,n_i}))$. Moreover,
  up to an isomorphism,
\begin{equation*}
\cK_0\simeq \cN_J\otimes \cG, \quad  \pi_0(X)=X\otimes I_\cG \quad
\text{ for  any } \  X\in C^*({\bf S}_{i,j}),
\end{equation*}
 where $\cG$ is  a Hilbert space with
$$
\dim \cG=\dim \left\{\text{\rm range}\,  {\bf \Delta_{q,\pi(S)}^m}
(I_{\cK})\right\},
$$
 and $\pi_1$ is a $*$-representation  which annihilates the compact operators   and
$$
 (I-\Phi_{q_1,\pi_1({\bf S}_1)}) \cdots (I-\Phi_{q_k,\pi_1({\bf S}_k)})
(I_{\cK_1})= 0.
$$
If  $\pi'$ is another nondegenerate  $*$-representation of
$C^*({\bf S}_{i,j})$ on a separable  Hilbert space  $\cK'$, then
$\pi$ is unitarily equivalent to $\pi'$ if and only if
$\dim\cG=\dim\cG'$ and $\pi_1$ is unitarily equivalent to $\pi_1'$.
\end{theorem}
\begin{proof}  Note that, due to
 Theorem \ref{compact},   the ideal $\cC(\cN_J)$ of
compact operators  in $B(\cN_J)$ is contained in the
$C^*$-algebra  $C^*({\bf S}_{i,j})$. Due to  standard theory of
representations of the $C^*$-algebras \cite{Arv-book}, the
representation $\pi$ decomposes into a direct sum $\pi=\pi_0\oplus
\pi_1$ on  $ \cK=\tilde\cK_0\oplus \tilde\cK_1$, where
$$\tilde\cK_0:=\overline{\text{\rm span}}\{\pi(X)\cK:\ X\in \cC(\cN_J)\}
\quad \text{ and  }\quad  \tilde\cK_1:=\tilde\cK_0^\perp,
$$
and the  representations $\pi_j:C^*({\bf S}_{i,j})\to
B(\tilde\cK_j)$ are defined by $\pi_j(X):=\pi(X)|_{\tilde{\cK}_j}$ for
$j=0,1$. We remark  that $\pi_0$, $\pi_1$  are disjoint
representations of $C^*({\bf S}_{i,j})$
 such that
 $\pi_1$ annihilates  the compact operators in $B(\cN_J)$, and
  $\pi_0$ is uniquely determined by the action of $\pi$ on the
  ideal $\cC(\cN_J)$ of compact operators.
Since every representation of $\cC(\cN_J)$ is equivalent to a
multiple of the identity representation, we deduce that
$\cK_0\simeq \cN_J\otimes \cG$, $\pi_0(X)=X\otimes I_\cG$,
 for  any  $  X\in C^*({\bf S}_{i,j})$,
 where $\cG$ is  a Hilbert space.
Using Theorem \ref{compact} and its proof, one can
 show that the space $\tilde\cK_0$ coincides with the space
$\cK_0$.
Taking into account that $(I-\Phi_{q_1, {\bf S}_1) })^{m_1}\cdots (I-\Phi_{q_k, {\bf S}_k})^{m_k}
(I )={\bf P}_\CC$    is a
 projection  of rank one in $C^*({\bf S}_{i,j})$,
  we deduce  that
$
 (I-\Phi_{q_1,\pi({\bf S}_1)})^{m_1}\cdots (I-\Phi_{q_k,\pi({\bf S}_k)})^{m_k}
(I_{\cK_\pi})= 0$
  and
$
\dim \cG=\dim \left[\text{\rm range}\,\pi({\bf P}_\CC )\right].
$
The uniqueness of the decomposition is due to standard theory of
representations of $C^*$-algebras and Proposition \ref{eq-mult}.
\end{proof}

We remark that under
 the hypotheses and notations of Theorem $\ref{wold}$, and setting
$V_{i,j}:=\pi({\bf S}_{i,j})$ for any $i\in \{1,\ldots,k\}$ and $ j\in\{1,\ldots, n_i\}$,   one can see  that
${\bf V}:=\{V_{i,j}\}$  is a pure element in $\cV_{\bf q, {\it J}}^{\bf m}(\cK)$  if and only if
$
\cK:=\overline{\text{\rm span}}\left\{ {V}_{(\alpha)} {\bf \Delta_{q,V}^m}
(I_{\cK})  (\cK):\ (\alpha)\in \FF_{n_1}^+\times \cdots \times \FF_{n_k} \right\}.
$

 We can   obtain  a more refined dilation theorem for the class of noncommutative varieties   $\cV_{\bf q, \cQ}^{\bf m}(\cH)$, where $\cQ\subset \CC[Z_{i,j}]$ is an ideal generated by homogeneous  polynomials and ${\bf q}=(q_1,\ldots, q_k)$ is  a $k$-tuple of   positive regular  noncommutative polynomials.

Let $C^*(\Gamma)$ be the $C^*$-algebra generated by  a set of operators $\Gamma\subset B(\cK)$ and the identity.
A subspace $\cH\subset \cK$ is called $*$-cyclic for $\Gamma$ if
$\cK=\overline{\text{\rm span}}\{Xh, X\in C^*(\Gamma), h\in \cH\}$.

\begin{theorem}\label{dil2}  Let ${\bf q}=(q_1,\ldots, q_k)$ be a $k$-tuple of   positive regular  noncommutative polynomials
 and let ${\bf S}=\{{\bf S}_{i,j}\}$  be the universal model
associated with the abstract noncommutative variety   $\cV_{\bf q, \cQ}^{\bf m}$, where $\cQ\subset \CC[Z_{i,j}]$ is an ideal generated by homogeneous  polynomials.
If ${\bf T}=\{{ T}_{i,j}\}$ is   in
$\cV_{\bf q, \cQ}^{\bf m}(\cH)$,
then there exists  a $*$-representation $\pi:C^*({\bf S}_{i,j})\to
B(\cK_\pi)$  on a separable Hilbert space $\cK_\pi$,  which
annihilates the compact operators and
$$
(I-\Phi_{q_1,\pi({\bf S}_1)}) \cdots (I-\Phi_{q_k,\pi({\bf S}_k)}) (I_{\cK_\pi})=0,
$$
where $\pi({\bf S}):= (\pi({\bf S}_1),\ldots, \pi({S}_k))$
 and $\pi({\bf S}_{i}):=(\pi({\bf S}_{i,1}),\ldots, \pi({\bf S}_{i,n_i}))$,
such that
$\cH$ can be identified with a $*$-cyclic co-invariant subspace of
$$\tilde\cK:=\left[\cN_\cQ\otimes \overline{{\bf \Delta_{f,T}^m}(I)(\cH)}\right]\oplus
\cK_\pi$$ under  each operator
$$
V_{i,j}:=\left[\begin{matrix} {\bf S}_{i,j}\otimes
I_{\overline{{\bf \Delta_{f,T}^m}(I)(\cH)}}&0\\0&\pi({\bf S}_{i,j})
\end{matrix}\right],\qquad i\in \{1,\ldots,k\}, j\in\{1,\ldots, n_i\},
$$
where ${\bf \Delta_{q,T}^m}(I):= (id-\Phi_{q_1,T_1})^{m_1}\cdots (id-\Phi_{q_k,T_k})^{m_k}(I)$,
and such that
$$ T_{i,j}^*=V_{i,j}^*|_{\cH} \qquad
\text{ for all  } \quad i\in \{1,\ldots,k\}, j\in\{1,\ldots, n_i\}.
$$
\end{theorem}
 \begin{proof}
 Applying Arveson extension theorem
\cite{Arv-acta} to the map ${\bf \Psi}$ of Theorem \ref{C*-charact2}, we find a unital
completely positive linear map ${\bf \Psi}:C^*({\bf S}_{i,j})\to B(\cH)$ such that ${\bf \Psi} ({\bf S}_{(\alpha)} {\bf S}_{(\beta)})^*={\bf T}_{(\alpha)} {\bf T}_{(\beta)}^*$ for all $(\alpha), (\beta)$ in $\FF_{n_1}^+\times \cdots \times \FF_{n_k}^+$.
 Let $\tilde\pi:C^*({\bf S}_{i,j})\to
B(\tilde\cK)$ be the minimal Stinespring dilation \cite{St}  of
${\bf \Psi}$. Then we have
$${\bf \Psi}(X)=P_{\cH} \tilde\pi(X)|\cH,\quad X\in C^*({\bf S}_{i,j}),
$$
and $\tilde\cK=\overline{\text{\rm span}}\{\tilde\pi(X)h:\ X\in C^*({\bf S}_{i,j}), h\in
\cH\}.$  Now,  one can  show that  that
  that $P_\cH \tilde\pi({\bf S}_{(\alpha)})|_{\cH^\perp}=0$ for any
$(\alpha)\in \FF_{n_1}^+\times \cdots \times \FF_{n_k}^+$. Consequently, $\cH$
 is an
invariant subspace under each  operator $\tilde\pi({\bf S}_{i,j})^*$
and
\begin{equation*}
\tilde\pi({\bf S}_{i,j})^*|_\cH={\bf \Psi}({\bf S}_{i,j}^*)=T_{i,j}^*
\end{equation*}
for any $i\in \{1,\ldots,k\}$ and $ j\in\{1,\ldots, n_i\}$.
Applying    the Wold decomposition of  Theorem \ref{wold}  to the  Stinespring representation $\tilde\pi$, one can complete the proof of the theorem. We omit the details since the proof is now very similar to the corresponding result from \cite{Po-Berezin-poly}.
\end{proof}

Let ${\bf V} $ be the dilation
 of ${\bf T}$ given by Theorem $\ref{dil2}$. One can easily prove that
  ${\bf V}$ is a pure element in ${\bf {\cV}_q^m}(\tilde\cK)$ if and only if
${\bf T}$ is  a pure element in ${\bf {\cV}_q^m}(\cH)$, and
   $
 (I-\Phi_{q_1,{V}_1}) \cdots (I-\Phi_{q_k, {V}_k})
(I_{\tilde\cK})= 0
$ \
if and only if \ $
 (I-\Phi_{q_1, {T}_1}) \cdots (I-\Phi_{q_k,{T}_k})
(I_{\cH})= 0
$.
 We remark that under the
additional  condition that
\begin{equation*}
\overline{\text{\rm span}}\,\{{\bf S}_{(\alpha)}  {\bf S}_{(\beta)}^* :\
 (\alpha), (\beta)\in \FF_{n_1}^+\times \cdots \times \FF_{n_k}^+\}=C^*({\bf S}_{i,j}),
\end{equation*}
 which holds, for  example, for the    polyballs (commutative or noncommutative), one can show  that  the dilation provided by Theorem \ref{dil2} is minimal.
In this case,   taking into account the uniqueness of the
minimal Stinespring representation    and the
Wold type decomposition of Theorem \ref{wold}, we can prove that the dilation is unique up to  unitary equivalence.

\bigskip

\section{Characteristic functions and operator models}

We  provide a characterization for the class of  elements  in  the abstract  noncommutative variety $\cV_{{\bf f},J}^{\bf m}$ which admit constrained  characteristic functions.
  The characteristic function is  a complete unitary invariant  for   completely non-coisometric tuples. In this case, we obtain  operator models  in terms of  the  constrained characteristic functions.

Let ${\bf S}:=\{{\bf S}_{i,j}\}$ be the universal model
 associated to the abstract  noncommutative variety $\cV_{{\bf f},J}^{\bf m}$ and let
   $\Phi:\cN_J\otimes \cH\to
\cN_J\otimes \cK$  be a multi-analytic operator
with respect to ${\bf S}$, i.e.,
  $$\Phi({\bf S}_{i,j}\otimes I_\cH)=({\bf S}_{i,j}\otimes I_\cK)\Phi$$
for any $i\in \{1,\ldots,k\}$ and  $j\in \{1,\ldots, n_i\}$. The {\it support} of $\Phi$  is  the  smallest reducing subspace
 $\supp (\Phi)$ of $ \cN_J\otimes \cH$
 under  each operator ${\bf S}_{i,j}$ containing   the co-invariant
  subspace $\cM:=\overline{\Phi^*(\cN_J\otimes \cK)}$.
  Using Theorem \ref{cyclic} and its proof, we deduce that if $1\in \cN_J$, then
 $$
 \supp (\Phi)=\bigvee_{(\alpha)\in \FF_{n_1}^+\times\cdots \times \FF_{n_k}^+}({\bf S}_{(\alpha)}\otimes I_\cK) (\cM)=\cN_J\otimes \cL,
 $$
 where $\cL:=({\bf P}_\CC\otimes I_\cH)\overline{\Phi^*(\cN_J\otimes \cK)}$.
  We say that  two
multi-analytic operator $\Phi:\cN_J\otimes \cK_1 \to \cN_J\otimes \cK_2$ and
 $\Phi':\cN_J\otimes \cK_1' \to \cN_J\otimes \cK_2'$  coincide if there are two unitary operators $\tau_j\in
B(\cK_j, \cK_j')$ such that
$$
\Phi'(I_{\cN_J}\otimes \tau_1)
=(I_{\cN_J}\otimes \tau_2) \Phi.
$$
As in \cite{Po-Berezin-poly}, one can prove that if    $\Phi_s:\cN_J\otimes \cH_s\to
\cN_J\otimes \cK$, \ $s=1,2$,    are  multi-analytic operators   with respect to ${\bf S}:=\{{\bf S}_{i,j}\}$ such that
$
\Phi_1 \Phi_1^*=\Phi_2 \Phi_2^*,
$
then there is a unique partial isometry $V:\cH_1\to \cH_2$ such that
$\Phi_1=\Phi_2(I_{\cN_J}\otimes V),
$
where $(I_{\cN_J}\otimes V)$ is a inner multi-analytic operator  with initial space $\supp (\Phi_1)$ and  final space $\supp (\Phi_2)$.
In particular,  the  multi-analytic operators $\Phi_1|_{\supp (\Phi_1)}$ and $\Phi_2|_{\supp (\Phi_2)}$ coincide.

\begin{definition} A $k$-tuple  ${\bf T} \in \cV_{{\bf f},J}^{\bf m}(\cH)$  is said to have constrained
  characteristic function   if there is a Hilbert space $\cE$ and
   a multi-analytic operator $\Psi:\cN_J\otimes \cE \to \cN_J\otimes \overline{{\bf \Delta_{f,T}^m}(I) (\cH)}$ with respect to ${\bf S}=\{{\bf S}_{i,j}\}$ such that
$$
{\bf K_{f,T, {\it J}}}{\bf K_{f,T,{\it J}}^*} +\Psi \Psi^*=I,
$$
where ${\bf K_{f,T, {\it J}}}$ is  the constrained noncommutative Berezin kernel associated with ${\bf T}\in \cV_{{\bf f},J}^{\bf m}(\cH)$.
\end{definition}
 According to the remarks above, if $1\in \cN_J$ and  there is a constrained characteristic function for ${\bf T}\in \cV_{{\bf f},J}^{\bf m}(\cH)$, then it is essentially unique.
We also remark that in the particular case when $k=1$ and $m_1=1$,
all the elements in the
 noncommutative variety $\cV_{f_1}^{1}$ have constrained characteristic functions.
Using Theorem \ref{Beur-fact}, one can deduce the following result.
\begin{theorem} An element   ${\bf T}=\{T_{i,j}\}$  in  the noncommutative variety $\cV_{{\bf f}, J}^{\bf m}(\cH)$  admits a constrained characteristic function if and only if
$$
{\bf \Delta}_{{\bf f,S}\otimes I}^{\bf p}(I -{\bf K_{f,T,{\it J}}}{\bf K_{f,T, {\it J}}^*})\geq 0
$$
for  any ${\bf p}:=(p_1,\ldots, p_k)\in \ZZ_+^k$ such that ${\bf p}\leq {\bf m}$, where ${\bf K_{f,T,{\it J}}}$ is the  constrained  Berezin kernel  associated with ${\bf T}$ and ${\bf S}:=\{{\bf S}_{i,j}\}$ is the universal model of $\cV_{{\bf f}, J}^{\bf m}$.
\end{theorem}

 If ${\bf T}$ has characteristic function, the multi-analytic operator $\Gamma$ provided by the  proof of Theorem \ref{Beur-fact} when $G=I -{\bf K_{f,T,{\it J}}}{\bf K_{f,T,{\it J}}^*}$, which we denote by $\Theta_{\bf f,T,{\it J}}$,  is called  the {\it constrained  characteristic function} of ${\bf T}$. More precisely, $\Theta_{\bf f,T, {\it J}}$
  is the multi-analytic operator
 $$\Theta_{\bf f,T,{\it J}}:\cN_J\otimes \overline{{\bf \Delta_{f,M_T}^m}(I)(\cM_T)} \to \cN_J\otimes \overline{{\bf \Delta_{f,T}^m}(I)(\cH)}
 $$
defined by $\Theta_{\bf f,T,{\it J}}:=(I -{\bf K_{f,T,{\it J}}}{\bf K_{f,T,{\it J}}^*})^{1/2} {\bf K_{f,M_T,{\it J}}^*}$, where
$${\bf K_{f,T,{\it J}}}: \cH \to \cN_J \otimes  \overline{{\bf \Delta_{f,T}^m}(I)(\cH)}$$
is the constrained  noncommutative Berezin kernel associated with ${\bf T}$ and
$${\bf K_{f,M_T, {\it J}}}: \cH \to \cN_J \otimes  \overline{{\bf \Delta_{f,M_T}^m}(I)(\cM_{\bf T})}$$
is the constrained noncommutative Berezin kernel associated
 with ${\bf M_T}\in {\bf \cV_f^m}(\cM_{\bf T})$.  Here, we have
$$
\cM_{\bf T}:= \overline{{\rm range}\,(I -{\bf K_{f,T, {\it J}}}{\bf K_{f,T, {\it J}}^*}) }
$$
and  ${\bf M_T}:=\{M_{i,j}\}$, where $M_{i,j}\in B(\cM_{\bf T})$  is  given by $M_{i,j}:=A_{i,j}^*$ and  $A_{i,j}\in B(\cM_{\bf T})$ is uniquely defined by
$$
A_{i,j}\left[(I -{\bf K_{f,T,{\it J}}}{\bf K_{f,T, {\it J}}^*})^{1/2}x\right]:=(I -{\bf K_{f,T,{\it J}}}{\bf K_{f,T, {\it J}}^*})^{1/2}({\bf S}_{i,j}\otimes I)x
$$
for any $x\in \cN_J\otimes \overline{{\bf \Delta_{f,T}^m}(I)(\cH)}$. According to Theorem \ref{Beur-fact}, we have
$$
{\bf K_{f,T, {\it J}}}{\bf K_{f,T, {\it J}}^*}+ \Theta_{\bf f,T, {\it J}}
\Theta_{\bf f, T, {\it J}}^*=I.
$$

We denote by $\cC_{\bf f, {\it J}}^{\bf m}(\cH)$ the set of all ${\bf T}= \{T_{i,j}\}\in \cV_{{\bf f}, J}^{\bf m}(\cH)$  which admit constrained characteristic functions.
 In what follows, we  provide a model theorem for class of  the  completely non-coisometric  elements in
$\cC_{{\bf f}, J}^{\bf m}(\cH)$. Due to the results obtained in  the previous sections, the proof is now  similar to that of Theorem 6.4 from \cite{Po-Berezin-poly}. We shall  omit it.

\begin{theorem}\label{model}  Let ${\bf T}=\{T_{i,j}\}$ be a
 completely non-coisometric  element  in
$\cC_{\bf f, {\it J}}^{\bf m}(\cH)$  and let  ${\bf S}:=\{{\bf S}_{i,j}\}$ be the universal model associated to the abstract  noncommutative variety $\cV_{\bf f,{\it J}}^{\bf m}$.   Set
$$
\cD:=\overline{{\bf \Delta_{f,T}^m}(I)(\cH)},\quad  \quad \cD_*:=\overline{{\bf \Delta_{f,M_T}^m}(I)(\cM_T)},
$$
and $\Delta_{\Theta_{\bf f,T,{\it J}}}:= \left(I-\Theta_{\bf f,T,{\it J}}^*
\Theta_{\bf f,T,{\it J}}\right)^{1/2}$, where $\Theta_{\bf f,T,{\it J}}$ is the characteristic function of ${\bf T}$.
 Then ${\bf T} $ is unitarily equivalent to
$\TT:=\{\TT_{i,j}\}\in \cC_{\bf f}^{\bf m}(\HH_{\bf f,T,{\it J}})$, where $\TT_{i,j}$ is a bounded operator acting on the
Hilbert space
\begin{equation*}
\begin{split}
\HH_{\bf f,T, {\it J}}&:=\left[\left(\cN_J\otimes\cD\right)\oplus
\overline{\Delta_{\Theta_{\bf f,T,{\it J}}}(\cN_J\otimes \cD_*)}\right]\\
& \qquad \qquad\ominus\left\{\Theta_{\bf f,T,{\it J}}\varphi\oplus
\Delta_{\Theta_{\bf f,T,{\it J}}}\varphi:\ \varphi\in \cN_J\otimes  \cD_*\right\}
\end{split}
\end{equation*}
 and is uniquely defined by the relation
$$
\left( P_{\cN_J\otimes\cD}|_{\HH_{\bf f,T,{\it J}}}\right) \TT_{i,j}^*x=
({\bf S}_{i,j}^*\otimes I_{ \cD})\left( P_{\cN_J\otimes \cD}|_{\HH_{\bf f,T,{\it J}}}\right)x
$$
for any $x\in \HH_{\bf f,T,{\it J}}$. Here,
   $ P_{\cN_J\otimes  \cD}$ is the orthogonal
projection of the Hilbert space
$$\cK_{\bf f,T,{\it J}}:=\left(\cN_J\otimes\cD\right)\oplus
\overline{\Delta_{\Theta_{\bf f,T,{\it J}}}(\cN_J\otimes \cD_*)}$$
 onto
the subspace $\cN_J\otimes \cD$.
 \end{theorem}

\begin{corollary}\label{pure-model} Let  ${\bf T}=\{T_{i,j}\}$ be an element  in
$\cC_{\bf f, {\it J}}^{\bf m}(\cH)$. Then ${\bf T}$  is pure if and only if  the constrained characteristic function $\Theta_{\bf f,T, {\it J}}$ is an inner multi-analytic operator with respect to  ${\bf S}:=\{{\bf S}_{i,j}\}$. Moreover, in this case  ${\bf T}=\{T_{i,j}\}$ is unitarily equivalent to ${\bf G}=\{ G_{i,j}\}$, where
$$ G_{i,j}:=P_{\bf H_{f,T, {\it J}}} \left({\bf S}_{i,j}\otimes I\right)|_{\bf H_{f,T, {\it J}}}
$$
and $P_{\bf H_{f,T, {\it J}}}$ is the orthogonal projection of $\cN_J\otimes \overline{{\bf \Delta_{f,T}^m}(I)(\cH)}$ onto
$${\bf H_{f,T,{\it J}}}:=\left\{\cN_J\otimes \overline{{\bf \Delta_{f,T}^m}(I)(\cH)}\right\}\ominus {\rm range}\, \Theta_{\bf f,T, {\it J}}.
$$
\end{corollary}

As consequences of the results above, we can easily  show that
  if  ${\bf T}=\{T_{i,j}\}\in \cV_{\bf f,{\it J}}^{\bf m}(\cH)$, then
 ${\bf T}$  is unitarily equivalent to
$\{{\bf S}_{i,j}\otimes I_\cK\}$ for some Hilbert space $\cK$ if and only if ${\bf T}\in \cC_{\bf f,{\it J}}^{\bf m}$ is completely non-coisometric  and
the characteristic function $\Theta_{\bf f,T,{\it J}}=0$.
 On the other hand, if ${\bf T}\in \cC_{\bf f,{\it J}}^{\bf m}$, then $\Theta_{\bf f,T,{\it J}}$  has dense range if and only if there is no nonzero vector $h\in \cH$ such that
    $$
    \lim_{q=(q_1,\ldots, q_k)\in \NN^k} \left< (id-\Phi_{f_1,T_1}^{q_1})\cdots (id-\Phi_{f_k,T_k}^{q_k})(I_\cH)h, h\right>=\|h\|,
    $$
where $T_i:=(T_{i,1},\ldots, T_{i, n_i})$ for any $i\in \{1,\ldots, k\}$.

An important consequence of Theorem \ref{model} is   that the constrained  characteristic function
$\Theta_{\bf f,T,{\it J}}$  is a complete unitary invariant for the completely non-coisometric
part of the noncommutative domain $\cC_{\bf f,{\it J}}^{\bf m}$. The proof is similar to that of Theorem 6.5 from \cite{Po-Berezin-poly}.

\begin{theorem}\label{u-inv}
Let   ${\bf T}=\{T_{i,j}\}\in \cC_{\bf f, {\it J}}^{\bf m}(\cH)$ and ${\bf T'}=\{T_{i,j}'\}\in
\cC_{\bf f, {\it J}}^{\bf m}(\cH')$ be two completely non-coisometric elements. Then ${\bf T}$ and
${\bf T}'$ are unitarily equivalent if and only if their constrained
characteristic functions $\Theta_{\bf f,T,{\it J}}$ and $\Theta_{\bf f,T',{\it J}}$
coincide.
\end{theorem}

\bigskip

\bigskip

       %

      \end{document}